\documentclass[12pt, reqno,letter]{amsart}
\usepackage{amsfonts,amssymb,latexsym,amsmath, amsxtra}
\usepackage[all]{xy}
\usepackage[dvips]{graphics}
\usepackage{diagbox}
\usepackage{graphicx}
\usepackage[T1]{fontenc}
\usepackage[OT2,T1]{fontenc}
\usepackage{enumitem}
\usepackage{caption}
\usepackage{float}

\DeclareSymbolFont{cyrletters}{OT2}{wncyr}{m}{n}
\DeclareMathSymbol{\Sha}{\mathalpha}{cyrletters}{"58}

\theoremstyle{plain} 
\newtheorem{theorem}{Theorem}[section]
\newtheorem{corollary}[theorem]{Corollary}
\newtheorem{conjecture}[theorem]{Conjecture}
\newtheorem{question}[theorem]{Question}

\newtheorem{lemma}[theorem]{Lemma}

\theoremstyle{definition}
\newtheorem{definition}[theorem]{Definition}
\theoremstyle{remark}
\newtheorem*{remark}{Remark}

\numberwithin{equation}{section}

\newcommand{\ZZ}{{\mathbb Z}}
\newcommand{\RR}{{\mathbb R}}

\newcommand{\QQ}{{\mathbb Q}}

\newcommand{\CC}{{\mathbb C}}

\newcommand{\bU}{{\bf U}}
\newcommand{\bV}{{\bf V}}

\newcommand{\bu}{{\bf u}}
\newcommand{\bv}{{\bf v}}

\newcommand{\sA}{{\mathcal A}}

\newcommand{\sI}{{\mathcal I}}
\newcommand{\sL}{{\mathcal L}}
\newcommand{\sM}{{\mathcal M}}

\newcommand{\XX}{{\mathbb X}}
\newcommand{\UU}{{\mathbb U}}

\newcommand{\Z}{\mathbb Z}

\newcommand{\ir}{\rho}

\renewcommand{\th}{\theta}

\def\cM{\mathcal M}

\def\cM{\mathcal M}

\newcommand{\con}{\operatorname{Conf}}
\newcommand{\bcon}{\operatorname{BConf}}
\newcommand{\DOD}{\operatorname{DOD}}
\newcommand{\FCC}{\operatorname{FCC}}
\newcommand{\HCP}{\operatorname{HCP}}

\newcommand{\ball}[1]{\operatorname{#1}}

\def\SS{\mathbb{S}}

\let\originalleft\left
\let\originalright\right
\renewcommand{\left}{\mathopen{}\mathclose\bgroup\originalleft}
\renewcommand{\right}{\aftergroup\egroup\originalright}

\begin{document}

%
%

\title[The Twelve Spheres Problem]{Configuration Spaces of Equal Spheres
Touching a Given Sphere: The Twelve Spheres Problem}

\author{Rob Kusner }
\address{Dept. of Mathematics \& Statistics, University of Massachusetts, Amherst, MA 01003, USA}
\email{profkusner@gmail.com}

\author{ W\"oden Kusner}
\address{Inst. of Analysis and Number Theory, Graz University of Technology, 8010 Graz.
AUSTRIA}
\email{wkusner@tugraz.at \\ wkusner@gmail.com}

\author{Jeffrey C. Lagarias}
\address{Dept. of Mathematics, University of Michigan, Ann Arbor, MI 48109-1043, USA}
\email{lagarias@umich.edu}

\author{Senya Shlosman}
\address{Skolkovo Institute of Science and Technology, Moscow; Aix Marseille Universit\'{e}, Universit\'{e} de Toulon, CNRS, CPT, UMR 7332, 13288, Marseille, FRANCE; 
Inst. for Information Transmission Problems, RAS, Moscow, RUSSIA
}
\email{senya.shlosman@univ-amu.fr \\ shlos@iitp.ru}

\date{February 26, 2018 version}

\begin{abstract}
The problem of twelve spheres is to understand, as a function of $r \in (0,r_{max}(12)]$, 
the configuration space of $12$ non-overlapping equal spheres of radius $r$ touching a central unit sphere.  
It considers to what extent, and in what fashion, touching spheres can be varied,
subject to the constraint of always touching the central sphere.  
Such constrained motion problems are of interest in physics and materials science, and the problem involves topology and geometry.  
This paper reviews the history of work on this problem, presents some new results, and formulates some conjectures. 
It also presents general results on configuration spaces of $N$ spheres of radius $r$ touching a central unit sphere, with emphasis on $3 \le N \le 14$. 
The problem of determining the maximal radius  $r_{max}(N)$ is a version of  the Tammes problem, to which L\'{a}szl\'{o} Fejes T\'{o}th made significant contributions.
\end{abstract}

\maketitle
\tableofcontents
\section{Introduction}\label{sec:Intro}
This paper studies 
constrained configuration spaces of $N$ equal spheres of radius $r$ touching a central sphere of radius $1$,
with emphasis on small values of $N$, particularly $N=12$.

The  ``problem of  the $13$ spheres,''  so named by Sch\"{u}tte and van der Waerden \cite{SvdW53} and 
Leech \cite{Leech56}, 
asks whether there  exists any configuration of $13$ non-overlapping unit spheres that all
touch a central unit sphere.  It was raised in the time of Newton by David Gregory, 
and eventually resolved mathematically as impossible.
Its resolution established that the ``kissing number'' of equal $3$-dimensional spheres is $12$.
Quantitatively, one can  ask what is the maximum radius for $13$ spheres all touching a central sphere of unit radius.
It is less than $1$ and its exact value was determined by Musin and Tarasov \cite{MusT:2013} in 2013. 

This paper treats a related problem:  
{\em How  can $N$ spheres of equal radius $r$  touch a given central sphere of radius $1$, in what patterns, and how are these patterns related?  
What is the topology of the corresponding constrained configuration space of  such spheres?} 
One may also ask  how the topology
changes as the radius $r$ varies. 
In this paper we review the remarkable history of this problem
for  radius $r=1$, the sphere packing case,
and $N=12$, the kissing number. 
We prove results on the structure of this configuration space 
and formulate several conjectures.
 
This problem has come up in physics and materials science.  
Many atoms and molecules are roughly spherical, and their local interactions are governed by how many of them can get close to a single atom. 
The arrangements possible for $13$ nearby spheres, and allowable motions between them, are relevant to the nature of local interactions, to measuring  the entropy of local configurations, and to phase changes in certain materials. We are especially motivated by a statement that Frank (1952) made in the context of supercooling of fluids, given in Section ~\ref{sec:27}. 
By insisting that exactly $12$ equal spheres touch a $13$-th central sphere, possibly of a different radius, 
we obtain  a mathematical toy problem that can be subjected to careful analysis.

As a mathematical problem, the twelve spheres problem has both a metric geometry
 aspect and a topology aspect.  
L\'{a}szl\'{o} Fejes T\'{o}th made major contributions to the metric geometry of the problem, which concerns extremal questions, formulated as densest packing problems.
In connection with the Tammes problem, described in Section ~\ref{sec:3}, he found the largest radius of $12$ spheres that can touch a central sphere of radius $1$, realized by the dodecahedral configuration $\DOD$, and found other extremal configurations of touching spheres for smaller ~$N$. 
He posed the Dodecahedral Conjecture concerning the minimal volume Voronoi cell in a unit sphere packing, and posed another conjecture characterizing all configurations that pack space with every sphere 
touching exactly $12$ neighboring spheres. 
Both of these conjectures are now proved. This paper focuses on
the  topological  aspect of the twelve spheres  problem;
more generally we discuss the topology of configuration spaces for
general $N$, and the allowable motions and rearrangements of such configurations.
We survey what is known  for small  $N$, in the range $3 \le N \le 14$.

\subsection{Configuration Spaces}\label{sec:11a}
The arrangements of the $N$ touching spheres are encoded in the associated {\em configuration space} of $N$-tuples of points on the surface of a unit sphere that remain at a suitable distance from each other. 
This space has nontrivial topology and geometry. 
In topology the general subject of configuration spaces  started in the 1960s with the consideration of topological spaces whose points denote configurations of a fixed number ~$N$ of labeled points on a manifold. 
 
This paper considers the {\em constrained} configuration space $\con(N)[r] $ of $N$ non-overlapping spheres of radius $r$ which touch a central sphere $\SS^2$ of radius $1$, centered at the origin. 
(Here ``non-overlapping'' means the spheres have disjoint interiors.) 
It can also be visualized as the space of $N$ spherical caps on the sphere, which are obtained as the radial projection of the external spheres onto the surface of the central sphere, whose {\em angular diameter} $\theta = \theta(r)$ is a known function of $r$.   

The centers of these caps define a constrained $N$-configuration on $\SS^2$ where no pair of points can approach closer than angular separation $\theta$. 
For generic (``non-critical'') values of ~$r$ for a range of values $0< r < r_{max}(N)$, this space is a compact $2N$-dimensional manifold with boundary, not necessarily connected.

The group $SO(3)$ acts as global symmetries of $\con(N)[r]$ by rigidly rotating the $N$-configuration of spheres touching the central sphere. 
The {\em reduced constrained configuration space} $\bcon(N)[r] = \con(N)[r]/SO(3)$ is obtained by identifying rotationally equivalent configurations. 
For generic values of $r$ it is a compact $(2N-3)$-dimensional manifold with boundary; for the case of $12$ spheres this is a $21$-dimensional manifold.
The subject of constrained configuration spaces has in part  been developed for applications to fields such as robotics.
For an introduction to the robotics aspect, see generally Abrams and Ghrist  ~\cite{AbramsG:2002} or Farber ~\cite{Farber:2008}. 

This  paper surveys results for small $N$  on the metric geometry problem of determining the maximum allowable radius $r_{max}(N)$ for $\con(N)[r]$ (equivalently $\bcon( N)[r]$) to be nonempty; this is a variant of the Tammes problem, also treated in the literature under the name {\em optimal spherical codes} (see Section ~\ref{sec:3}).

This paper also studies the topology of  configuration spaces of a fixed radius $r$, and the  changes in topology in such spaces as the radius $r$ is varied. 
In the latter case the configuration space changes topology at a set of {\em critical radius values}.
Associated to these special values are {\em critical configurations}, which are extremal in a suitable sense.
The change in topology is described by a generalization of Morse theory applicable to the radius function ~$r$, which we discuss in Section ~\ref{sec:4}. To determine these changes one studies the occurrence and structure of the critical configurations.
The simplest example of such topology change concerns the connectivity of the space of configurations as a function of $r$, reported by the rank of the $0$-th homology group of the configuration space.

The $12$ spheres problem includes as its most important special case that of unit spheres, where the sphere radius $r=1$. 
This special case is the one relevant to  sphere packing in dimension $3$. 
We treat the topological space $\bcon(12)[1]$  in Sections ~\ref{sec:5} and ~\ref{sec:6}, and formulate several conjectures related to it.  
The radius $r=1$ is a critical radius, and two configurations $\FCC$ and $\HCP$ on the boundary of the space $\bcon(12)[1]$ are critical configurations.
The topology of $\bcon(12)[1]$ appears to be very complicated, and its cohomology groups have not been determined.  
In Section ~\ref{sec:6} we describe how it is possible to move in the space $\bcon(12)[1]$ to deform any dodecahedral configuration $\DOD$ of $12$ labeled spheres to any other labeled $\DOD$ configuration, permuting the $12$ spheres arbitrarily, a result due to Conway and Sloane.
This suggests the (folklore) conjecture asserting  that $r=1$ is the largest radius value for which the configuration space $\bcon(12)[r]$ is connected, i.e. it is the largest ~$r$ for which the $0$-th cohomology group of $\bcon(12)[r]$ has rank $1$. 

This paper establishes some new results.
It makes the observation (in Section ~\ref{sec:43a}) that the family of $5$-configurations of spheres achieving $r_{max}(5)$ (see Figure  ~\ref{fig:example5}) is topologically complex .  
It completely determines (in Section ~\ref{sec:47}) the cohomology of $\bcon(\SS^2, 4)[r]$ for allowable $r$.  
It makes precise the notion of $N$-configurations being {\em critical for maximizing} the injectivity radius on $\bcon(\SS^2,N)$, and provides a necessary and sufficient {\em balancing condition} (Theorem ~\ref{thm:converse}) for criticality, prefatory to a ``Morse theory'' for such min-type functions ~\cite{KKLS16+}.  
And it formulates several new conjectures in Sections ~\ref{sec:65} and ~\ref{sec66}. 

\subsection{Physics and Materials Science}\label{sec:12} 
Configuration spaces are of interest in physics and materials science.
Jammed configurations are  a granular materials criterion for a stable packing.
According to Torquato and Stillinger ~\cite[p. 2634]{TorquatoS:2010} they are:
 ``particle configurations in which each particle is in contact with its nearest neighbors in such a way that mechanical stability of a specific type is conferred to the packing.''
Packings of rigid disks and spheres have been studied extensively by simulation (Lubachevsky and Stillinger ~\cite{LubachevskyS:1990}, Donev ~et ~al. ~\cite{DonevTSC:2004}).
It has been empirically discovered that randomly ordered hard spheres achieve in random close packing a density around $66$ percent ~\cite{ScottK:1969}, and pass through a jamming transition around $64$ percent ~\cite[p. 355]{LiuN:2010}.  
The appearance of a jamming phase transition, signaled by  a  change in shear modulus, and the formation of a glass state, is relevant in studying the  behavior of colloidal suspensions and granular materials. 
The large rearrangement of structure required in making a phase transition is relevant in the phenomenon of supercooling of liquids (see Section ~\ref{sec:27}).  
The nature of glass transitions has been called ``the deepest and most interesting unsolved problem in solid state theory'' (Anderson ~\cite{Anderson:1995}).
For articles and reviews of these topics, see generally Ediger ~et ~al. ~\cite{EdigerAN:1996}, O'Hern ~et ~al. ~\cite{OhernSLN:2003}, and Liu and Nagel ~\cite{LiuN:2010}.  
For a survey of hard sphere models, including the idea of a liquid-solid phase transition in packings, see generally L\"{o}wen ~\cite{Lowen:1999}.

One may make an analogy between the configuration spaces $\bcon(N)[r]$  treated here and a sphere packing model for jamming studied in ~\cite{OhernSLN:2003}, which treats spheres having repulsive local potential at zero density and zero applied stress, and includes hard spheres for one model parameter value.  
In the latter model, the order parameter is the packing fraction of the spheres.  
In the configuration space model, a proxy value for the packing fraction is the radius parameter $r$, which determines the fraction of surface area of $\SS^2$ covered by the ~$N$ spherical caps.  
An analogue of the jamming transition value in the configuration space model is then the maximal radius $r_{conn}(N)$ at which the constrained configuration space $\bcon(N)[r]$ remains connected; this property is detected by the $0$-th cohomology group.  
Finer topological invariants of this kind are then supplied by the various critical values $r_j$ at which  the ranks of the individual cohomology groups $H^{k}(\bcon(N)[r], \QQ)$ change.  
Our configuration model is simplified in being $2$-dimensional, with constrained configurations on the surface of a $2$-sphere $\SS^2$, which, however, has the new feature of positive curvature, giving a compact constrained configuration space.  
For the jamming problem itself, the space of (constrained) configurations of hard spheres in a large $3$-dimensional box seems a more appropriate space.  
The general direction of inquiry investigating the transition of topological invariants (like Betti numbers) of configuration spaces as the radius parameter is varying could shed new light on the nature of jamming transitions.  
For further remarks, see Section ~\ref{sec:7}.
 
\subsection{Roadmap}\label{sec:13}
The sections of the paper have been written with the aim to be independently readable.  
We prove some results for general $N$, but
Sections ~\ref{sec:2}, \ref{sec:5} and \ref{sec:6} focus  on the case $N=12$.
Section ~\ref{sec:2} gives a brief history of results on the twelve
 spheres problem, stemming from its special role in connection
with sphere packing 
in dimension $3$.  
Section ~\ref{sec:3} surveys results on the maximal radius ~$r_{\max}(N)$ possible
for a configuration of $N$ equal spheres touching  a central sphere of radius $1$, for small $N$.  
This problem  is a version of  the Tammes problem.  
Section ~\ref{sec:4} begins with the topology of configuration spaces of $N$ points in $\RR^2$ and on the $2$-sphere $\SS^2$, corresponding to radius $r=0$.  
It then considers spaces of configurations of equal spheres of radius $r$ touching a sphere of radius $1$ for variable $0 < r \le r_{\max}(N)$.  
It defines a notion of  critical configuration in the spirit of min-type Morse theory.  
Section ~\ref{sec:5} discusses the special configuration space of $12$ unit spheres touching a $13$-th central sphere, i.e. the case  $r=1$.  
It focuses on properties of the $\FCC$ configuration, the $\HCP$ configuration and the dodecahedral configuration $\DOD$.  
It shows that the $\FCC$ and $\HCP$ configurations are critical  (in the sense of Section ~\ref{sec:42}) in the reduced configuration spaces $\bcon(12)[1]$.  
It also shows that there are continuous deformations  in $\bcon(12)[1]$
moving a dodecahedral configuration to an $\FCC$ configuration, resp. moving it  to an $\HCP$ configuration. 
Section ~\ref{sec:6} considers  the problem of permutability of the spheres of the dodecahedral configuration for $r=1$, conjecturing that  $\bcon(12)[1]$ is connected, and that this is the largest value of $r$ where connectedness holds.  
It also considers the $r>1$ case and formulates  several conjectures about disconnectedness.  
Section ~\ref{sec:7} makes some concluding remarks.

\section{The Twelve Spheres Problem: History }\label{sec:2}
We begin with some historical vignettes  concerning  configurations of $12$ spheres touching a central sphere,
as they have come up in physics, astronomy, biology and materials science. 

\subsection{Kepler (1611) }\label{sec:21}
Johannes Kepler ~(1571--1630) studied packings and crystals in his 1611 pamphlet ``The Six-cornered Snowflake'' ~\cite{Kepler:1611}.
In it he asserts that the densest sphere packing of equal spheres is the $\FCC$ packing, or ``cannonball packing.'' 
He states that this packing has $12$ unit spheres touching each central sphere: 

\begin{quote}
In the second mode, not only is every pellet touched by its four neighbors in
the same plane, but also by four in the plane above and four below, so throughout
one will be touched by twelve, and under pressure spherical pellets will become
rhomboid. This arrangement will be more compatible to the octahedron and
the pyramid. The packing will be the tightest possible, so that in no other
arrangement could more pellets be stuffed into the same container.
\footnote{\ ``Iam si ad structuram solidorum quam potest fieri arctissimam
progredaris, ordinesque ordinibus superponas, in plano prius coaptatos
aut ii erunt quadrati A aut trigonici B: si quadrati aut singuli globi ordinis
superioris singulis superstabunt ordinis inferioris aut contra
singuli ordinis superioris sedebunt inter quaternos ordinis inferioris.
Priori modo tangitur quilibit globis a quattuour cirucmstantibus in eodem plano,
ab uno supra se, et ab uno infra se: et sic in universum a six aliis, eritque ordo
cubicus, et compressione facta fient cubi: sed no erit arctissima coaptatio.
Posteriori modo praeterquam quod quileibet globus a quattuor circumstantibus
in eodem plano tangitur etiam a quattuor infra se, et a quattuor supra se, et sic
in universum a duodecim tangetur; fientque compressione ex globosis rhombica.
Ordo hic magis assimilabitur octahedro et pyramidi. Coaptatio fiet arctissima,
ut nullo praetera ordine plures globuli in idem vas compingi queant.''
[Translation: Colin Hardie ~\cite[p. 15]{Kepler:1611}]}
\end{quote}

He expands on the construction as follows: 

\begin{quote}
\begin{minipage}{\linewidth}
\begin{minipage}{.65\linewidth}
Thus, let $B$ be a group of three balls; set one $A$, on it as apex;
let there be also another group $C$, of six balls, and another $D$, of ten, and
another $E$, of fifteen. Regularly superpose the narrower on the wider
to produce the shape of a pyramid. Now, although in this
construction each one in the upper layer is seated between three in the
lower, yet if you turn the figure round so that not the apex but the whole
side of the pyramid is uppermost, you will find, whenever you peel off
one ball from the top, four lying below it in  square pattern. Again as before,
one ball will be touched by twelve others, to with, by six neighbors in
the same plane, and by three above and three below. Thus in the closest
pack in three dimensions, the triangular pattern cannot exist without
the square, and vice versa. It is therefore obvious that the loculi of the
pomegranate are squeezed into the shape of a solid rhomboid....
\footnotemark
\end{minipage}
\begin{minipage}{.34\linewidth}
\begin{figure}[H] 
   \centering
   \vspace{-1em}
    \includegraphics[width = .7\linewidth]{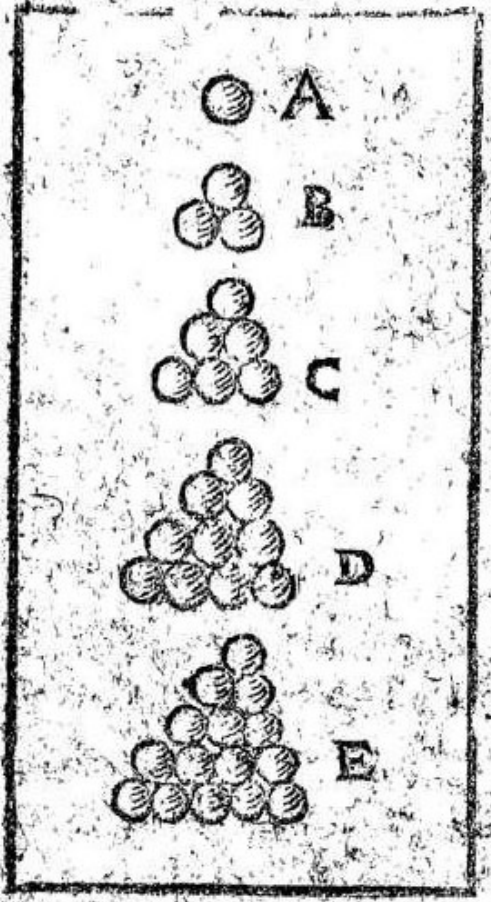}\\ \vspace{1em}
\captionsetup{justification=centering}
   \caption{\newline Woodcut of \newline Kepler\,sphere \newline arrangements}
   \label{fig:0}
\end{figure}
\end{minipage}
\end{minipage}
\end{quote}
\footnotetext{\ ``Esto enim $B$ copula trium globorum. Ei superpone $A$ 
unum pro apice; esto et alia copula senum globorum $C$, et alia
denum $D$ et alia quindenum $E.$ Impone semper angustiorem
latiori, ut fiat figura pyramidis. Etsi igitur per hanc impositionem singuli
superiores sederunt into trinos inferiores: tamen iam versa figura, ut non
apex sed integrum latus pyramidis sit loc superiori, quoties unum globulum
deglberis e summis, infra stabunt quattuor ordine quadrato. Et rursum
tangetur unus globus ut prius, et duodecim aliis, a sex nempe circumstantibus
in eodem plano tribus supra et tribus infra. Ita in solida coaptatione arctissima
non potest ess ordo triangularis sine quadrangulari, nec vicissim.
Patet igitur,
acinos punici mali, materiali necessitate concurrente cum rationaibus
incrementi acinorum, exprimi in figuri rhombici corporis...''
[Translation by Colin Hardie ~\cite[p. 17]{Kepler:1611}]}

The cannonball packing had been studied earlier by the English mathematician Thomas Hariot ~[Harriot] ~(1560--1621).  
Hariot was mathematics tutor to Sir Walter Raleigh, designed some of his ships, wrote a treatise on navigation, and went on an expedition to Virginia in 1585--1587 as surveyor, reporting on it in 1590 in ~\cite{Harriott:1590}, his only published book. 
He computed a chart in 1591 on how to most efficiently stack cannonballs using the $\FCC$ packing, and computed a table of the number of cannonballs in such stacks (see Shirley ~\cite[pp. 242--243]{Shi83}). 
Hariot supported the atomic theory of matter, in which case macroscopic objects may be packed in arrangements of very tiny spherical objects, i.e. atoms ~\cite[Chap. III]{Kar66}.  
He corresponded with Kepler in 1606--1608 on optics, and mentioned the atomic theory in a December 1606 letter as a possible way of explaining why some light is reflected, and some refracted, at the surface of a liquid.  
Kepler replied in 1607, not supporting the atomic theory.  
The known correspondence of Hariot with Kepler does not deal directly with sphere packing.

The statement  that the maximal density of a sphere packing in $3$-dimensional space equals $\frac{\pi}{\sqrt{18}} \approx 0.74048$, which is attained by the $\FCC$ packing, is called the  {\em Kepler Conjecture.}  
It was settled affirmatively in the period 1998--2004 by Hales with Ferguson ~\cite{Lagarias:2011}.  
A second generation proof, which is a formal proof checked entirely by computer, was recently completed in a project led by Hales ~\cite{HalesA:2015}.

\subsection{Newton and Gregory (1694) }\label{sec:23}
In 1694  Isaac Newton (1642--1727) and David Gregory (1659?- 1708)
had a discussion of touching spheres related to preparing a second edition of Newton's {\em Principia}.  
It concerned  the question whether the ``fixed stars'' are subject to gravitational attraction. What force is ``balancing'' their apparent fixed positions?

Gregory ~\cite[Vol III, p. 317]{Newt} summarized in a memorandum a conversation with Newton on 4 May 1694 concerning the
brightest stars as:  

\begin{quote}
To discover how many stars there are of a given magnitude, he [Newton] considers how many spheres, nearest, second from them, third etc. surround a sphere in a space of three dimensions, there will be $13$ of first magnitude, $4 \times 13$ of second, $9 \times 4 \times 13$ of third.\footnote{\ ``Ut noscatur quot sunt stellae magnitudinis 1 ae, 2 dae, 3 ae \& c. considerando quot spherae proximae,  
seundae ab his 3 ae \& c. spheram in spatio trium dimensionis circumstent: erunt 13 primae, $4 \times 13$ $2$-dae, $9 \times 4 \times  13$ 3 ae.''}   
\end{quote} 

\noindent Newton's own star table ``A Table of ye fixed Starrs for ye yeare 1671'' records $13$ first magnitude stars, $43$ of the second magnitude, $174$ of third magnitude (see ~\cite[Vol ~II, ~p. ~394]{Newt}).  

Newton  drafted a new Proposition to be included in a second edition of the {\em Principia}, stating~\cite[p. 81, in translation]{Hoskin:1977}: 

\begin{quote}
Proposition XV. Theorem XV. The fixed stars are at rest in the heavens and are separated by enormous distances from our Sun and from each other.
 \end{quote}
 
In a draft proof he wrote~\cite[p. 85, in translation]{Hoskin:1977}: 

\begin{quote}
That the stars are at huge distances from our Sun is clear enough from the absence of parallax; and that they lie at no less distances
from each other may be inferred from their differing apparent magnitudes. For there are $13$ stars of the first magnitude and roughly
the same number of equal spheres can be arranged about a central sphere equal to them.
\end{quote}

and:

\begin{quote}
For if around some sphere there are arranged more spheres of about the same size, the number of spheres which surround
it closely will be $12$ or $13$; at the second stage about $50$; at the third about $110$ [roughly $9 \times 12$];
at the fourth, $200$ [$16 \times 12 \frac{1}{2}$], ...
\end{quote}
This argument is similar to one of Kepler ~\cite[Liber I, Pars II, p. 138]{Kepler:1618} (translation in Koyr\'{e} ~\cite[p. 80]{Koyre:1957}), with roots in the claim of Giordano Bruno, that all stars are suns.  

After further work, through several drafts, Newton abandoned this Proposition (according to Hoskin ~\cite{Hoskin:1977}).  
It was not included in the second edition of the {\em Principia}, when it  was later published in 1713.

Gregory continued study of  the geometric problem underlying the spacing of stars.  
In an (unpublished) notebook\footnote{\ This notebook is at Christ Church, Oxford, according to J. Leech ~\cite{Leech56}.} he considered the packing problem of $2$-dimensional disks in concentric rings and, in $3$ dimensions, that of equal spheres, noting that $13$ spheres might touch a given equal sphere ~\cite[Vol III, Letter 441, Note (10), p. 321]{Newt}.  
He considered the $13$ sphere question in later years, making the following memorandum in 1704~\cite[p. 21]{Hiscock:1937}: 
\begin{quote}
{\em Oxon. 23 Nov$^{r}$ 1704.} Mr. Kyl\footnote{\ John Keill (1671--1721) succeeded Gregory as Savilian Professor.} said that if $13$ equal spheres touch an equal inmost sphere, $9 \times 13$ must touch one that include these former $14$, because there is nine times as much surface to stand on. I told him that we must reckon by the surface passing through their centers. 
\end{quote}
A manuscript of Gregory on Astronomy, translated into
English and  posthumously 
published  in 1715, states \cite[p. 289, sic]{Gregory:1715}: 
\begin{quote}
For if every Fix'd Star did the office of a Sun, to a portion of the Mundane space nearly equal to this that our Sun commands, there will be as many Fix'd Stars 
of the first Magnitude, as there can be Systems of this sort touching and surrounding ours; that is, as many equal Spheres as can touch an equal one in the middle of them. Now, 'tis certain from Geometry, that thirteen Spheres can touch and surround one in the middle equal to them, (for {\em Kepler} is wrong in asserting, in {\em B.} {\sc I} of the 
{\em Epit.}\footnote{This is Kepler \cite{Kepler:1618}.} that there may be twelve such, according to the number of Angles of an {\em Icosaedrum},) 
\end{quote}
Thus Gregory expressed a definite opinion that $13$ spheres might touch.

\subsection{Bender, Hoppe, G\"{u}nther (1874)}\label{sec:24}
The issue of whether $13$ equal spheres  might touch a central equal sphere was discussed in the physics literature in the period 1874--1875, with contributions by C. Bender ~\cite{Bender:1874}, Reinhold Hoppe ~\cite{Hoppe:1874} and Siegmund G\"{u}nther ~\cite{Gunther:1875}.  
Hoppe noted a mathematical gap in the argument of Bender.  
 G\"unther offered a physical intuition, but no proof.  
They all concluded that at most $12$ unit spheres could touch a central unit sphere.  
In 1994 Hales ~\cite{Hales:1994} noted a mathematical gap in the argument of Hoppe.

\subsection{Barlow (1883)}\label{sec:25}
In another context the crystallographer William Barlow (1845--1934) noted another optimal sphere packing, the {\em Hexagonal Close Packing} ($\HCP$).  
In a paper ``Probable nature of the internal symmetry of crystals'' ~\cite[p. 186]{Barlow:1883} he considered five symmetry types for crystal structure.  
The third kind of symmetry he describes is the $\FCC$ packing (Fig. 4 and 4a).  
He then stated:

\begin{quote}
A fourth kind of symmetry, which resembles the third in that each point
is equidistant from the twelve nearest points, but which is of a widely
different character than the three former kinds, is depicted if layers of spheres in
contact arranged in the triangular pattern (plan d) are so placed that the
sphere centers of the third layer are over those of the first, those of the fourth
layer over those of the second, and so on. The symmetry produced is
hexagonal in structure and uniaxial (Figs. 5 and 5a).
\end{quote}
Here ``plan d'' is the two-dimensional hexagonal packing, and Figs. 5 and 5a depict the $\HCP$ packing.  
He suggested that the atoms in a crystal of quartz ($Si O_2$) occur with the fourth kind of symmetry (see Figure ~\ref{fig:0-2}).

\begin{figure}[htbp]
   \centering
   \vspace{-1em}
      \includegraphics[scale=.75]{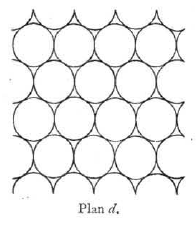}\quad\quad\quad
   \includegraphics[scale=.75]{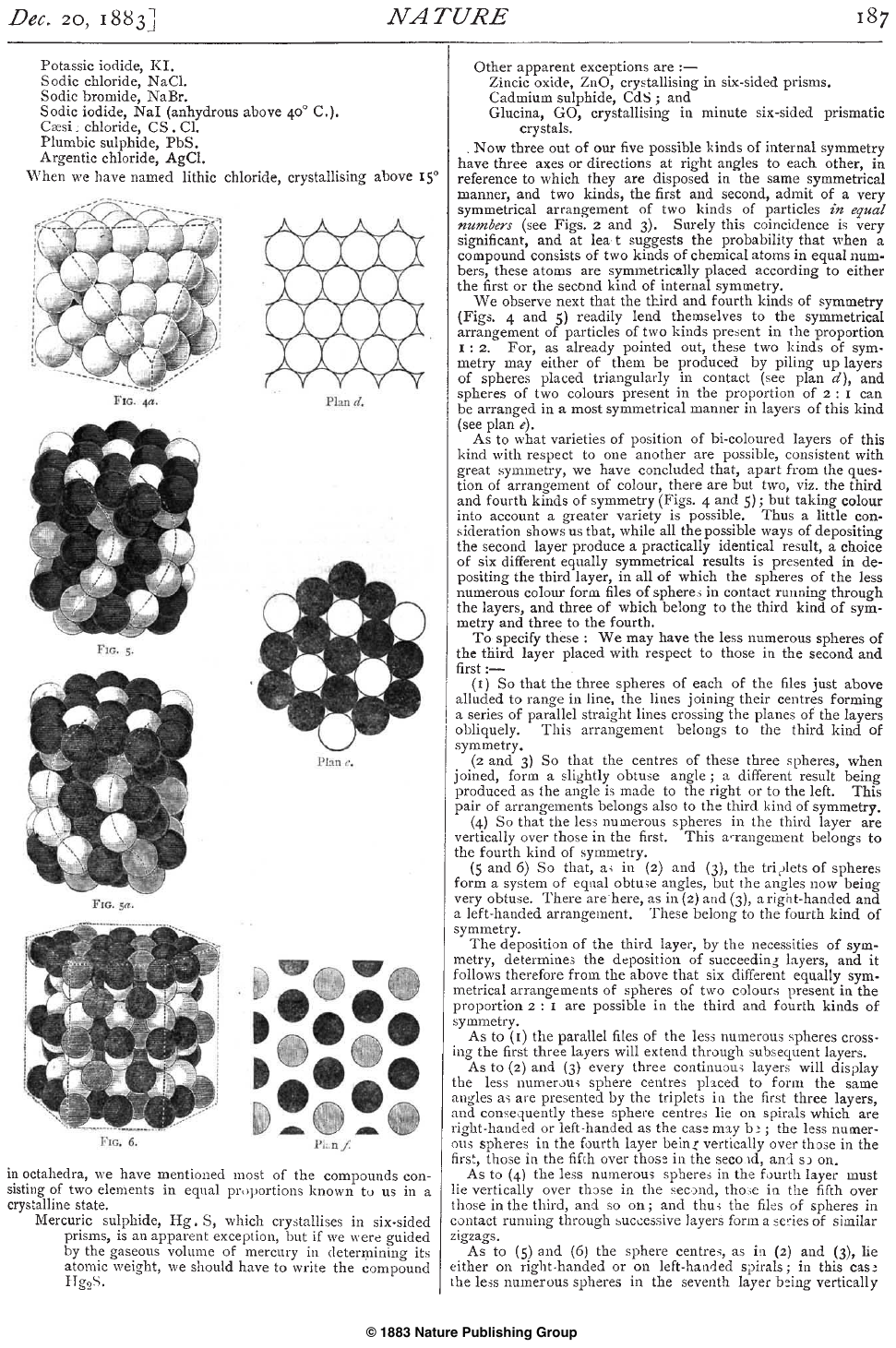}\vspace{-.5em}\\
   \includegraphics[scale=.75]{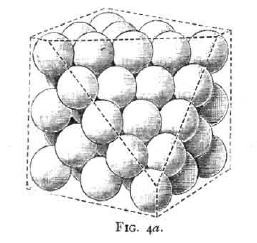}\quad\quad\quad
   \includegraphics[scale=.75]{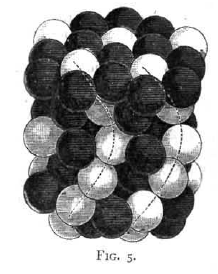}\quad\quad\quad
      \includegraphics[scale=.75]{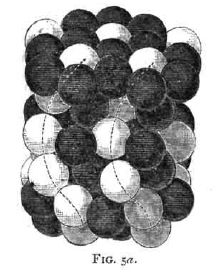}\vspace{-.9em}
   \caption{Barlow $\FCC$ and $\HCP$ packings}
   \label{fig:0-2}
\end{figure}

Barlow also stated later in the paper ~\cite[Figs. 7 and 8, p.207]{Barlow:1883} the following about twinned crystal arrays with a connecting layer: 

\begin{quote}
The peculiarities of {\em crystal-grouping}
displayed in twin crystals can be shown to favour the supposition that we
have in crystals symmetrical arrangement rather than symmetrical shape of atoms
or small particles. Thus if an octahedron be cut in half by a plane parallel to two
opposite faces, and the hexagonal faces of separation, while kept in contact and their
centres coincident, are turned one upon the other through $60^{\circ}$, we know that
we get a familiar example of a form found in some twin crystals. And a stack can be
made of layers of spheres placed triangularly in contact to depict this form as
readily as to depict a regular octahedron, the only modification necessary being for
the layers above the centre layer to be placed as though turned bodily through $60^{\circ}$,
from the position necessary to depict an octahedron (compare Figs. 7 and 8). The modification,
as we see, involves {\em no departure from the condition that each particle is equidistant
from the twelve nearest particles.}
\end{quote}

\begin{figure}[htbp] 
   \centering
   \vspace{-1em}
   \includegraphics[scale=.75]{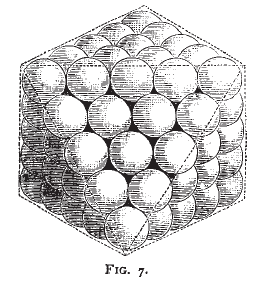}\quad\quad\quad
   \includegraphics[scale=.75]{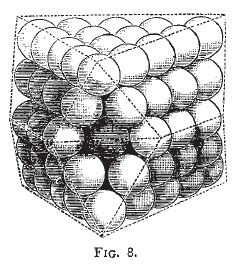}\vspace{-.9em}
   \caption{Barlow Twinned Crystal Packing}
   \label{fig:0-3}
\end{figure}

\subsection{Tammes (1930)}\label{sec:26}  
The Dutch botanist {Pieter Merkus Lambertus Tammes} made in 1930 a study of the equidistribution of pores on pollen grains ~\cite{Tam30}.  
He asked the question: What is the maximum number of circular caps $N(\theta)$ of angular diameter $\theta$ that can be placed without overlap on a unit sphere?  
Here $\theta$ is measured from the  center of the unit sphere $\SS^2$ in $\mathbb{R}^{3}$.  
Tammes ~\cite[Chap. 3]{Tam30} empirically determined that $N(\frac{\pi}{2}) = 6$, while $N(\theta) \le 4$ for $\theta > \frac{\pi}{2}$.  
Let $\theta= \theta(N)$ denote the maximal value of $\theta$ having $N(\theta) =N$.  
He concluded that $\theta(5)= \theta(6) = \frac{\pi}{2}$. 

The problem of determining various values of $N(\theta)$ is now called the {\em Tammes problem}.  
It is related to a  dual question of determining the maximal radius $r(N)$ possible for $N$ equal spheres all touching a central sphere of radius $1$.  
Namely, the maximal value of $\theta := \theta(N)$ having $N(\theta)=N$, determines the maximal allowable radius $r(N)$ of $N$  spheres touching a central unit sphere by a formula given in Lemma ~\ref{lemma:31} below.

\subsection{Fejes T\'{o}th (1943)}\label{sec:26b}
In 1943 L\'{a}szl\'{o} Fejes T\'{o}th ~\cite{FT:1943b} conjectured that the volume of any Voronoi cell of any sphere packing of $\RR^3$ by unit spheres is minimized by the dodecahedral configuration of $12$ unit spheres touching a central sphere.  
The Voronoi cell of the central sphere is then a regular dodecahedron circumscribed about the sphere.  The packing density of the dodecahedron is approximately $0.7546$, which is larger than the density of the known FCC packing of $\RR^3$.  
This conjecture became known as the {\em Dodecahedral Conjecture} and was settled affirmatively in 2010 (see Section ~\ref{sec:53}). 

\subsection{Frank (1952)}\label{sec:27}
The problem of molecular rearrangement in the liquid-solid phase transition is relevant in materials. 
The structure of ordinary ice, the $H_2 O$  phase labeled ice $I_{h}$, has an $\HCP$ packing
of its oxygen atoms, as observed in 1921 by Dennison ~\cite{Dennison:1921}. 
Note that the hydrogen atoms are free to change their orientations to some extent (Pauling ~\cite{Pauling:1935}).
Water exhibits a phenomenon of supercooling at standard pressure down to $-48^{\circ} C$; under special rapid cooling it can avoid freezing down to  $-137^{\circ} C$, and enter a glassy phase (Angell ~\cite{Angell:2008}). 

In 1952 Frederick Charles Frank ~\cite{Frank:1952}  argued that supercooling can occur because the  common arrangements of molecules in liquids assume configurations far from what they would assume if frozen.  
He wrote:

\begin{quote}
 Consider the
question of how many different ways one can put twelve billiard balls in
simultaneous contact with another one, counting as different the arrangements
which cannot be transformed into each other without breaking contact with
the centre ball?' The answer is {\em three}. Two  which come to the mind
of any crystallographer occur in the face-centred cubic and hexagonal close
packed lattices. The third comes to the mind of any good schoolboy, and it
is to put one at the centre of each face of a regular dodecahedron.
That body has five-fold axes, which are abhorrent to crystal symmetry:
unlike the other two packings, this one cannot be continuously extended in
three dimensions. You will find that the outer twelve in this packing do not
touch each other. If we have mutually interacting deformable spheres, like
atoms, they will be a little closer to the centre in this third kind of packing;
and if one assumes they are argon atoms (interacting in pairs with attractive
and repulsive potentials proportional to $r^{-6}$ and $r^{-12}$) one may
calculate that the binding energy of the group of thirteen is $8.4 \%$ greater
than for the other two packings. This is $40 \%$ of the lattice energy per
atom in the crystal. I infer that this will be a very common grouping in
liquids, that most of the groups of twelve atoms around one will be of this form,
that freezing involves a substantial rearrangement, and not merely an
extension of the same kind of order from short distances to long ones;
a rearrangement which is quite costly of energy in small localities, and which only becomes
 economical when extended over a considerable volume, because unlike
the other packing it can be so extended without discontinuities.
\end{quote}

The three local arrangements Frank specifies we shall label as $\FCC$ (face-centered-cubic), $\HCP$ (hexagonal close packing)  and $\DOD$ (dodecahedral), for convenience.  
The crystalline arrangements of $\FCC$ and $\HCP$ are ``extremal'' (i.e.~on the boundary of the configuration space), while the balls in $\DOD$ configuration are free to move independently. 

Frank's assertion that there are exactly three possible arrangements is  {\em false} if taken literally.  
There are continuous deformations between any arrangement of types $\FCC$, $\HCP$ and $\DOD$ and any of the other types (see Section ~\ref{sec:53}).  
There is however an important kernel of truth in Frank's statement, which buttresses his argument made concerning the existence of supercooling:  
each of the three arrangements above is ``remarkable'' in some sense (see Section ~\ref{sec:5}).  
To move from a large arrangement of spheres having many $\DOD$ configurations to one frozen in the $\HCP$ packing requires substantial motion of the spheres. 

\subsection{Sch\"{u}tte and van der Waerden (1953)}\label{sec:27b}
In a paper titled ``Das Problem der dreizehn Kugeln'' [``The problem of the thirteen balls''] Kurt Sch\"{u}tte and  Bartel Leendert van der Waerden ~\cite{SvdW53} gave a rigorous proof that one cannot have $13$ unit spheres touching a given central sphere.

There has been much further work on this problem.  
In his 1956 paper titled ``The problem of $13$ spheres'' John Leech ~\cite{Leech56} gave a two page proof of the impossibility of $13$ unit spheres touching a unit sphere.  
More accurately he stated: ``In the present paper I outline an independent proof of this impossibility, certain details which are tedious rather than difficult have been omitted.''  
Various authors have written to fill in such details,  which balloon the length of the proof.  
These include  work of  Maehara ~\cite{Maeh:2001} in 2001, who gave in 2007 a simplified proof  ~\cite{Maeh:2007}.  
Other proofs of the thirteen spheres problem were  given by Anstreicher ~\cite{Anst:2004} in 2004 and Musin ~\cite{Mus:2006} in 2006.

\subsection{Fejes T\'{o}th (1969)}\label{sec:28}
In 1969 L\'aszl\'o Fejes T\'{o}th ~\cite{FT:1969} discussed the problem of characterizing those sphere packings in space that have the property that every sphere in the packing touches exactly $12$ neighboring spheres.  
The $\FCC$ and $\HCP$ packing both have this property, as already noted by Barlow (1883).  
There are in addition uncountably many other packings, obtained by stacking plane layers of hexagonally packed spheres (``penny packing''), where there are two choices at each level of how to pack the next level.  
Fejes T\'{o}th conjectured that all such packings are obtained in this way. 

This conjecture of Fejes T\'{o}th's was settled affirmatively by  Hales ~\cite{Hales:2013} in 2013. 

\subsection{Conway and Sloane (1988)}\label{sec:29}
In their book: {\em  Sphere Packings, Lattices and Groups,} John H. Conway and Neil J. A. Sloane  considered the question: 
{\em What  rearrangements of the $12$ unit spheres are possible using motions that maintain contact with the central unit sphere at all times?}  
In ~\cite[Chap. 1, Appendix: Planetary Perturbations]{Splag3} they sketch a  result asserting:  
The configuration space of $12$ unit spheres touching a $13$-th allows arbitrary permutations of all $12$ touching spheres in the configuration.  
That is, if the spheres are labeled and in the DOD configuration, it is possible, by moving them on the surface of the central sphere, to arbitrarily permute the spheres in a DOD configuration.

We will describe the motions in detail to obtain such permutations in Section ~\ref{sec:6}. 

\section{Maximal Radius Configurations of $N$ Spheres: The Tammes Problem}\label{sec:3}
What is the maximal radius $r(N)$ possible for $N$ equal spheres all touching a central sphere of radius $1$?  
This problem is closely related to the {\em Tammes problem} discussed above, which concerns instead the maximum number of circular caps $N(\theta)$ of angular diameter $\theta$ that can be placed without overlap on a sphere.  
The latter problem is also the problem of constructing good spherical codes (see ~\cite[Chap. 1, Sec. 2.3]{Splag3}).

\subsection{Radius versus Angular Diameter Parameterization}\label{sec:30} 
One can convert the angular measure $\theta$ into the radius of touching spheres;  
for a sphere touching a central unit sphere, its associated spherical cap on the central sphere is the radial projection of its points onto the surface of the central sphere. 

%
%
\begin{figure}[htbp] 
   \centering
   \includegraphics[scale=.8]{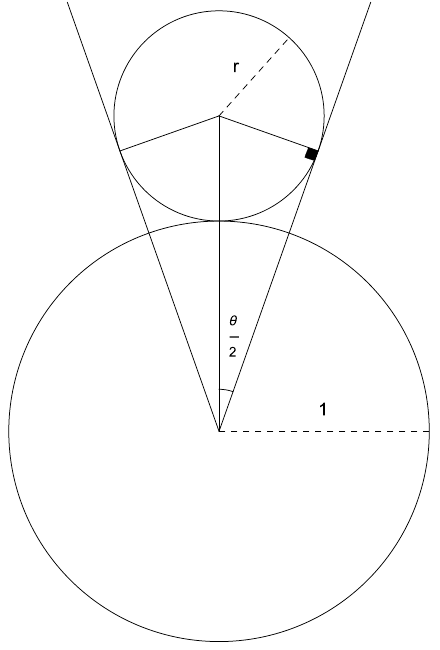}
   \caption{Angular measure $\theta$ related to  radius $r$ }
   \label{fig:7-2}
\end{figure}

\begin{lemma}\label{lemma:31}
For a fixed $N>2$, the maximal value of\, $\theta := \theta_{max}(N)$ having $N(\theta)=N$ determines the maximal allowable radius $r_{max}(N)$ of $N$  spheres touching a central unit sphere, using the formula
\[
r_{max}(N) = \frac{\sin\left(\frac{\theta(N)}{2}\right)}{1- \sin \left(\frac{\theta(N)}{2}\right)}.
\]
Conversely, given $r=r_{max}(N)$, we obtain
\[
\theta_{max}(N) = 2 \arcsin\left(\frac{r}{1+r}\right),
\]
choosing $0  < \theta(N) < \pi$.
\end{lemma}

\begin{proof}
From the right triangle in Figure ~\ref{fig:7-2}
we have
\[
\sin\frac{\theta}{2}= \frac{r}{1+r}.
\]
This relation gives a bijection of the interval $ 0 \le \theta < \pi$ to the interval $0 \le r < \infty.$
\end{proof}

\subsection{Rigorous Results for Small $N$}
The Tammes problem has been solved exactly for only a few values of $N$, including $3 \le N \le 14$ and $N=24$.

\subsubsection{ {\bf Fejes T\'{o}th:} \!\!$N\!\!=\! 3, 4, 6,12$}
The Tammes problem was solved for $N\!\!=\! 3, 4, 6$ and~$12$ by L\'aszl\'o Fejes T\'{o}th ~\cite{FT:1943} in 1943, 
where extremal configurations of touching points for $N\!=3$ are attained by vertices of an equilateral triangle arranged around the equator, 
and for $N\!\!=\!4,6,12$ by vertices of regular polyhedra (tetrahedron, octahedron and icosahedron) inscribed in the unit sphere.  
Fejes T\'{o}th  proved the following inequality:  
For $N$ points on the surface of the unit sphere, at least two points can always be found with spherical distance
\[
d \le \arccos \Big( \frac{ (\cot \omega)^2 -1}{2}\Big), \quad \mbox{with} \quad \omega = \Big(\frac{N}{N-2} \Big) \frac{\pi}{6}.
\]
\noindent
Note that $d$ is the edge-length of a spherical equilateral triangle with the expected area for an element of an $N$-vertex triangulation of $\SS^2$.  
The inequality is sharp for $N\!=3, 4, 6$ and~$12$ for the specified configurations above.

In 1949 Fejes T\'{o}th ~\cite{FT:1949} gave another proof of his inequality.  
His result was re-proved by Habicht and van der Waerden ~\cite{HvdW:1951} in 1951.  
After converting this result to the $r$-parameter using Lemma ~\ref{lemma:31}, we may re-state his result for $N= 12$ as follows. 

\begin{theorem}\label{th72}
{\rm (Fejes T\'{o}th (1943))}\ 

$(1)$ The maximum radius of $12$ equal spheres touching a central sphere of radius $1$ is:

\[
r_{max}(12) = \frac{1}{\sqrt{\frac{5+ \sqrt{5}}{2}} -1} \approx 1.1085085.
\]

\noindent 
Here $r_{max}(12)$ is a real root of the fourth degree equation  $x^4 -6x^3 +x^2+4x+1=0$.

$(2)$ An extremal configuration achieving this radius is the $12$ vertices of an inscribed regular icosahedron (equivalently, face-centers of a circumscribed regular dodecahedron).
\end{theorem}

\subsubsection{ {\bf Sch\"{u}tte and van der Waerden:} $N=5,7,8,9$}
The Tammes problem was  solved for $N=5$ in 1950 by van der Waerden, building on work of  Habicht and van der Waerden ~\cite{HvdW:1951}.  
It was solved for $N=7$ by Sch\"{u}tte. These solutions, plus those of van der Waerden for $N=8$ and Sch\"{u}tte for  $N=9$ appear in Sch\"{u}tte and van der Waerden ~\cite{SvdW51}.  
They give a history of these developments in ~\cite[p. 97]{SvdW51}.

Their paper used geometric methods, introducing and studying the allowed structure of the graphs describing the touching patterns of arrangements of $N$ equal circles on $\SS^2$.  
These graphs are now called {\em contact graphs}, and Sch\"{u}tte and van der Waerden credit their introduction to Habicht.  
Sch\"{u}tte also conjectured candidates for optimal configurations  for $N=10,  13, 14, 15, 16$ and van der Waerden conjectured candidates for $N=11, 24, 32$ (see ~\cite{SvdW51}).

 L.~Fejes T\'{o}th presented the work of Sch\"{u}tte and van der Waerden in his 1953 book on sphere-packing ~\cite[Chapter VI]{FT:1953}.  
 This book uses the terminology of {\em maximal graph} for the graph of a configuration achieving the maximal radius for $N$.  
 In 1959 Fejes T\'{o}th ~\cite{FT:1959} noted that the set of vertices of a square antiprism gave an extremal $N=8$ configuration on the $2$-sphere.

\subsubsection{{\bf Danzer:} $N=6,7, 8, 9, 10, 11$}
In his 1963 Habilitationsschrift ~\cite{Danz:1963} (see the 1986 English translation ~\cite{Danz:1986}), Ludwig Danzer made a geometric study of the contact graph for a configuration of $N$ circles on the surface of a sphere.  
This graph has a vertex for each circle and an edge for each pair of touching circles.  
A contact graph is called {\em maximal} if it occurs for a set of circles achieving the maximal radius $r_{max}(N)$.  
It is called {\em optimal} if it has the minimum number of edges among all maximal contact graphs.  
A contact graph is called {\em irreducible} if the radius cannot be improved by altering a single vertex.  
For each small $N$, Danzer found a complete list of irreducible contact graphs.   
He used this analysis to  prove the conjectures of Sch\"{u}tte and van der Waerden ~\cite{SvdW51} above for the cases $N=10, 11$. 

\begin{theorem}\label{th73}
{\rm (Danzer (1963))}\par~$(1)$ For $7 \le N \le 12$ there is, up to isometry, a unique $\theta$-maximizing unlabeled configuration of spheres with $N(\theta)=N$.

$(2)$ For $N=12$, the vertices of a regular icosahedron form the unique $\theta$-maximizing configuration.  
The $\theta$-maximizing configuration for $N=11$ is a regular icosahedron with one vertex removed.
\end{theorem}

In ~\cite[Theorem II]{Danz:1986} Danzer classified irreducible sets for $72^{\circ} = \frac{2 \pi}{5} < \theta < 90^{\circ} = \frac{2 \pi}{4}$.  
There are additional $N$-irreducible graphs for $N=7,8, 9, 10$ in these cases. For $N=6,7,8,9$ he finds one optimal set and one irreducible set with one degree of freedom.  
He also finds for $N=8$ an irreducible set with two degrees of freedom.  
For $N=10$ he finds one optimal set, two irreducible sets with no degrees of freedom, and five with at least one degree of freedom.  
Danzer states that the irreducible sets with no degrees of freedom (presumably) give relative optima.  
An irreducible graph having a degree of freedom fails to be relatively optimal, since deforming along its degree of freedom leads to a boundary graph with an additional edge, where the extrema is reached.

Danzer's work was not published in a journal until the 1980's.  
In the interm, B\"{o}r\"{o}czky ~\cite{Bor83} gave another solution for $N=11$, and H\'{a}rs ~\cite{Hars86} for $N=10$.

\subsubsection{ {\bf Musin and Tarasov:} $N= 13, 14$}
Very recently the Tammes problem  was solved for the cases $N=13$ and $N=14$ by Oleg Musin and Alexey Tarasov ~\cite{MusT:2012, MusT:2015}.  
Their proofs were computer-assisted, and made use of an enumeration of all irreducible configuration contact graphs (see ~\cite{MusT:2013}).  
Earlier work on configurations of up to $17$ points was done by B\"{o}r\"{o}czky and Szab\'{o} ~\cite{BorS03, BorS03b}).

\subsubsection{ {\bf Robinson:} $N=24$}
The case $N=24$ was solved  in 1961 by Raphael M. Robinson ~\cite{Rob:1961}.  
He proved a 1959 conjecture of Fejes T\'{o}th ~\cite{FT:1959b}, asserting  that the extremal $\theta(24) ~\approx ~43^{\circ} \,\, 42^{'}$ and that the extremal configuration of $24$ sphere centers are the vertices of a snub cube.  
Coxeter ~\cite[p. 326]{Cox62} describes the  snub cube.

\subsection{The Tammes Problem: Optimal Contact Graphs and Optimal Parameters}\label{sec:33}
Table ~\ref{tab711} summarizes optimal angular parameters and radius parameters on the Tammes problem for $3 \le N \le 14$ and $N=24$ (see Aste and Weaire ~\cite[Sect. 11.6]{AstW:2000}).  
The configuration name given is associated to the vertices in the corresponding polyhedron being inscribed in a sphere, e.g. an icosahedron has $N=12$ vertices (see  Melnyk et al. ~\cite[Table 2]{MelnykKS:1977}). 
In the case $N=5$ the polyhedron is any from a family of trigonal bipyramids, including the square pyramid as a degenerate case.  
For $N=11$ the polyhedron is a singly capped pentagonal antiprism, i.e. the icosahedron with one vertex deleted.  
The cases $N=7, 9$ and $10$ are described in ~\cite[pp. 1747--1749]{MelnykKS:1977}.
Figure ~\ref{fig:contact} shows schematically the optimal contact graphs for $3 \le N \le 14$.

 \begin{table}[h]
 \centering
\renewcommand{\arraystretch}{.70}
\begin{tabular}{|r|c|l|l|l|}
\hline&&&&\\[-.7em]
\multicolumn{1}{|c|}{$N$} & \multicolumn{1}{c|}{$\theta_{max}(N)$} & \multicolumn{1}{c|}{$r_{max}(N)$} & \multicolumn{1}{c|}{Configuration} & \multicolumn{1}{c|}{Source}\\
[.3em]\hline&&&&\\[-.7em]
3 & $ \frac{2\pi}{3}= 120^{\circ}$ & \vtop{\hbox{\strut $ 3+ 2 \sqrt{3}$}\hbox{\strut $\approx 6.4641$}} & Equilateral Triangle & Fejes T\'{o}th (1943)\\
[.3em]\hline &&&&\\[-.7em]
4 & \vtop{\hbox{\strut\hspace{1.75em}   $\cos^{-1}(-\frac{1}{3})$ \quad}\hbox{\strut\hspace{3.25em}$\approx 109.4712^{\circ}$}} & \vtop{\hbox{\strut $ 2+ \sqrt{6}$ }\hbox{\strut $\approx 4.4495$}} & Regular Tetrahedron & Fejes T\'{o}th (1943)\\
[.3em] \hline&&&&\\[-.7em]
5 & $ \frac{2 \pi}{4}  = 90^{\circ}$ & \vtop{\hbox{\strut$1+\sqrt{2}$}\hbox{\strut$\approx 2.4142$ }} & \vtop{\hbox{\strut Triangular Bipyramid }\hbox{\strut }} & Fejes T\'{o}th (1943)\\
[.3em] \hline&&&&\\[-.7em]
6 & $\frac{2 \pi}{4} = 90^{\circ}$ & \vtop{\hbox{\strut$1+\sqrt{2}$}\hbox{\strut$\approx 2.4142$ }} & Regular Octahedron & Fejes T\'{o}th (1943)\\
[.3em] \hline&&&&\\[-.7em]
7 & $ \approx 77.8695^{\circ}$  &$\approx 1.6913$ & [No name] & \vtop{\hbox{\strut Schutte and}\hbox{\strut van der Waerden (1951) }}\\
[.3em]\hline&&&&\\[-.7em]
8 & $\approx 74.8585^{\circ}$ & $\approx 1.5496$ & Square Antiprism & \vtop{\hbox{\strut Schutte and}\hbox{\strut van der Waerden (1951) }}\\
 \hline&&&&\\[-.7em]
9 &  $\approx 70.5288^{\circ}$ &\vtop{\hbox{\strut  $\frac{1+\sqrt{3}}{2}$}\hbox{\strut $\approx 1.3660$ }}  & [No name] & \vtop{\hbox{\strut  Schutte and}\hbox{\strut  van der Waerden (1951) }}\\
[.3em]\hline&&&&\\[-.7em]
10 & $\approx 66.1468^{\circ}$ & $\approx 1.2013$ &  [No name]  & Danzer (1963) \\
 [.3em] \hline&&&&\\[-.7em]
11& $\approx 63.4349^{\circ}$ & \vtop{\hbox{\strut $ \frac{1}{\sqrt{\frac{5+ \sqrt{5}}{2}} -1}$}\hbox{\strut $\approx 1.1085$}} & \vtop{\hbox{\strut Regular Icosahedron}\hbox{\strut minus one vertex}} & Danzer (1963)\\
[.3em] \hline&&&&\\[-.7em]
12& $\approx 63.4349^{\circ}$ &  \vtop{\hbox{\strut $ \frac{1}{\sqrt{\frac{5+ \sqrt{5}}{2}} -1}$}\hbox{\strut $\approx 1.1085$}} & Regular Icosahedron & Fejes T\'{o}th (1943)\\
  [.3em]  \hline&&&&\\[-.7em]
13 & $\approx 57.1367^{\circ}$ & $\approx 0.9165 $ & [No name] & Musin and Tarasov (2013)\\
[.3em] \hline&&&&\\[-.7em]
14 & $\approx 55.6706^{\circ}$ & $\approx 0.8759$ & [No name] & Musin and Tarasov (2015)\\
[.3em]\hline&&&&\\[-.7em]
24 & $\approx 43.6908^{\circ}$ & $\approx 0.5926$ & Snub Cube & Robinson (1961)\\
[.3em]
\hline
\end{tabular}
\vspace{.5em}
	\captionsetup{justification=centering}
\caption{The Tammes Problem for small $N$ }
\label{tab711}
\end{table}

\begin{figure}
\centering
\vspace{1em}
\begin{minipage}{\textwidth}
\begin{minipage}{.33\textwidth}
  \centering
\includegraphics[height=1.65in]{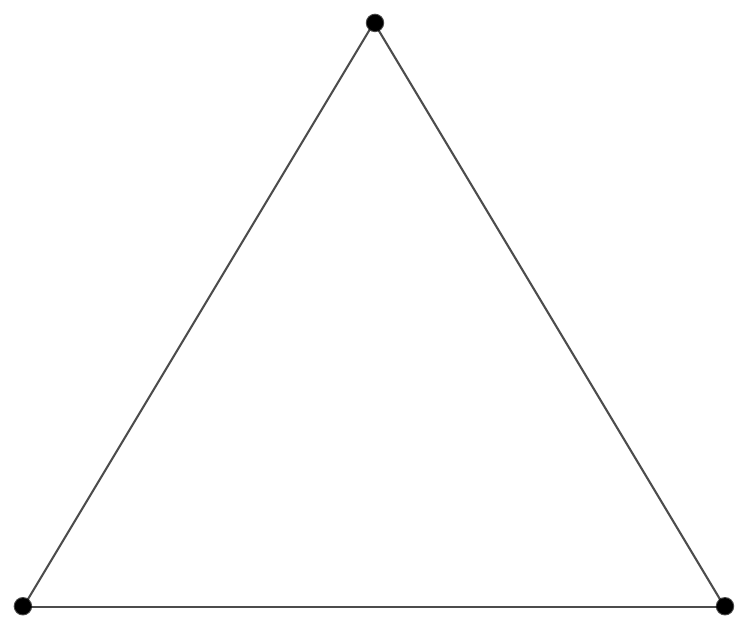}\\
$N=3$
\end{minipage}%
\begin{minipage}{.33\textwidth}
  \centering
\includegraphics[height=1.65in]{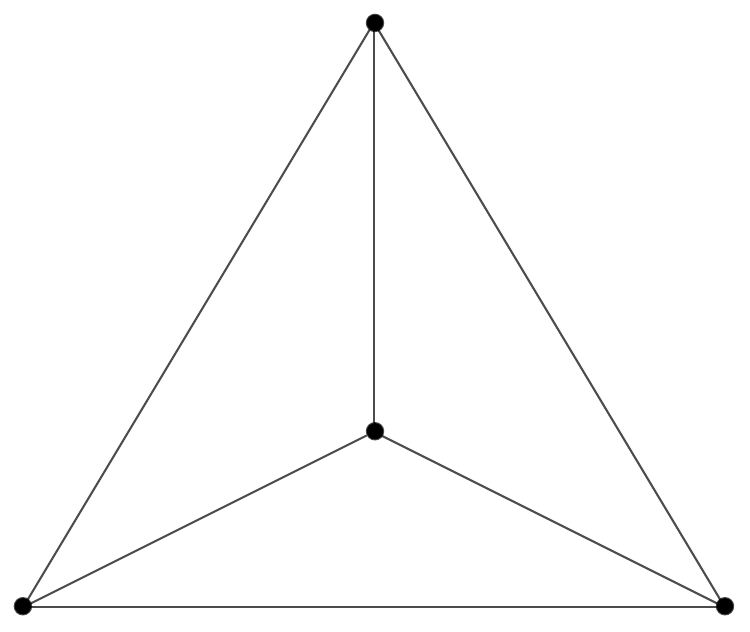}\\
$N=4$
\end{minipage}
\begin{minipage}{.33\textwidth}
  \centering
\includegraphics[height=1.65in]{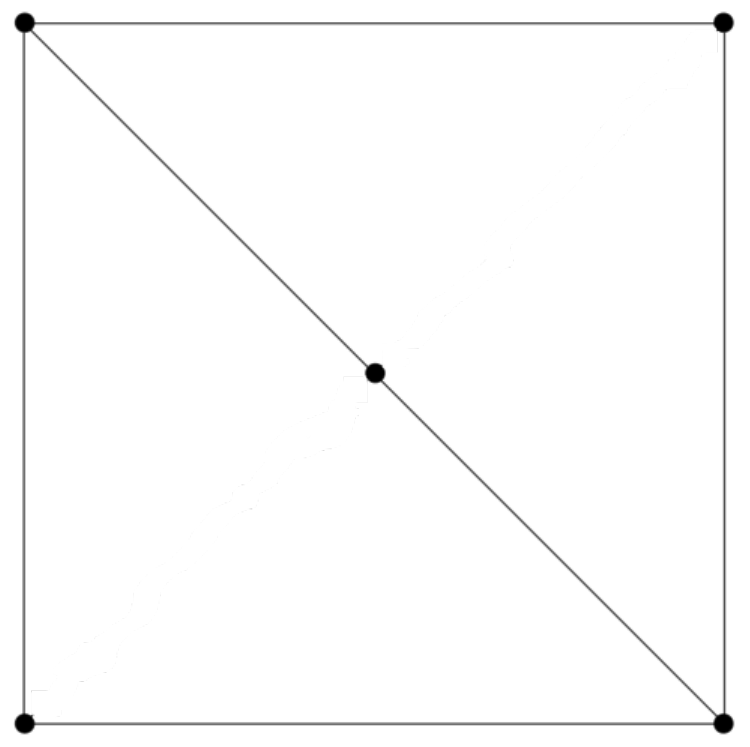}\\
$N=5$
\end{minipage}
\end{minipage}\vspace{1em}

\begin{minipage}{\textwidth}
\begin{minipage}{.33\textwidth}
  \centering
\includegraphics[height=1.65in]{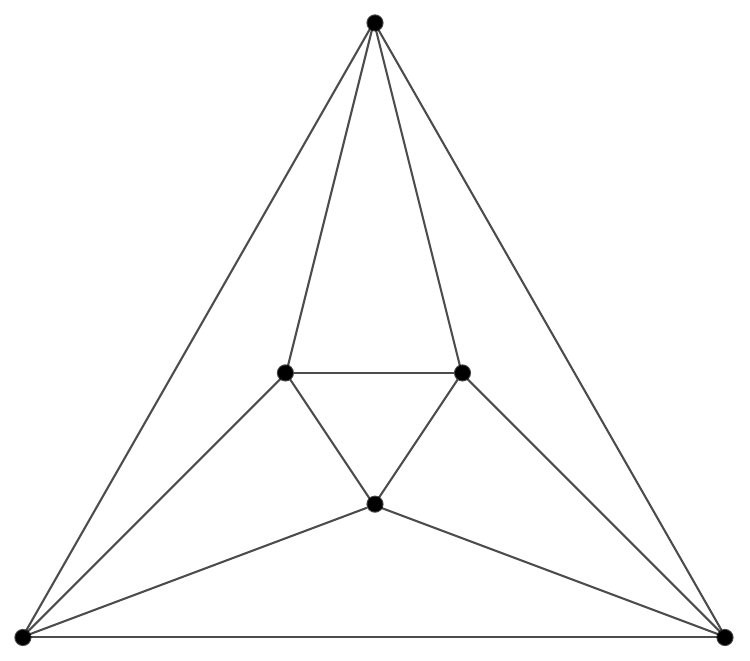}\\
$N=6$
\end{minipage}%
\begin{minipage}{.33\textwidth}
  \centering
\includegraphics[height=1.65in]{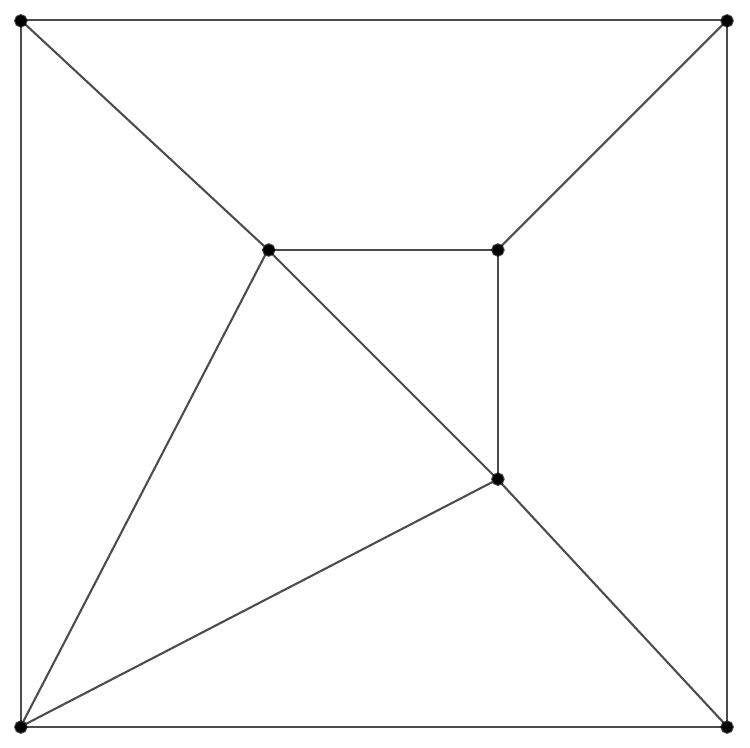}\\
$N=7$
\end{minipage}
\begin{minipage}{.33\textwidth}
  \centering
\includegraphics[height=1.65in]{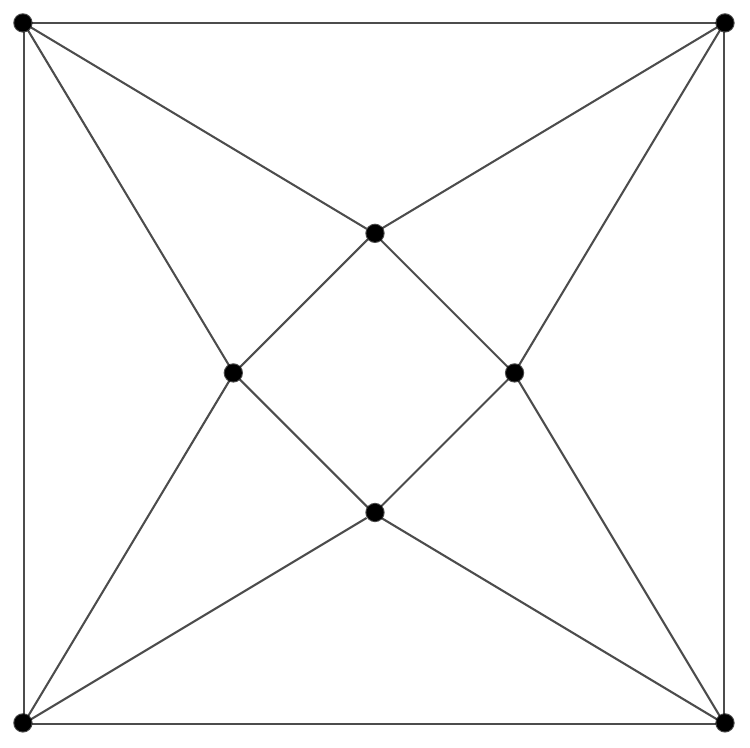}\\
$N=8$
\end{minipage}
\end{minipage}
\vspace{1em}

\begin{minipage}{\textwidth}
\begin{minipage}{.33\textwidth}
  \centering
\includegraphics[height=1.65in]{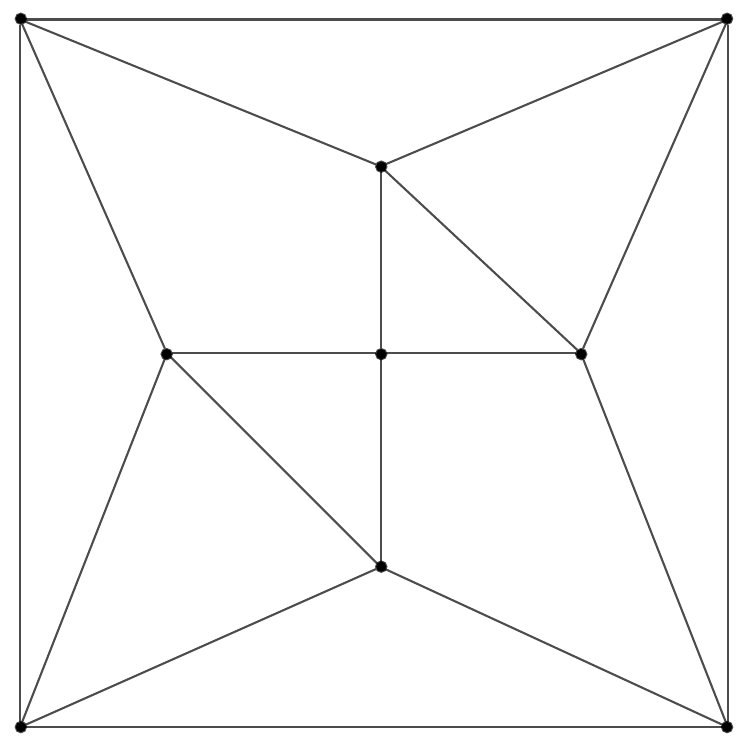}\\
$N=9$
\end{minipage}%
\begin{minipage}{.33\textwidth}
  \centering
\includegraphics[height=1.65in]{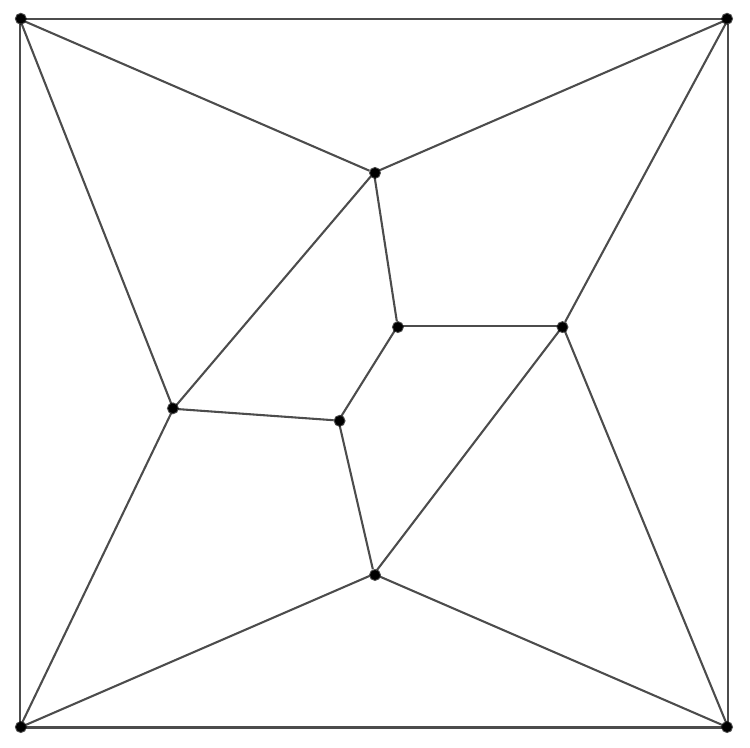}\\
$N=10$
\end{minipage}
\begin{minipage}{.33\textwidth}
  \centering
\includegraphics[height=1.65in]{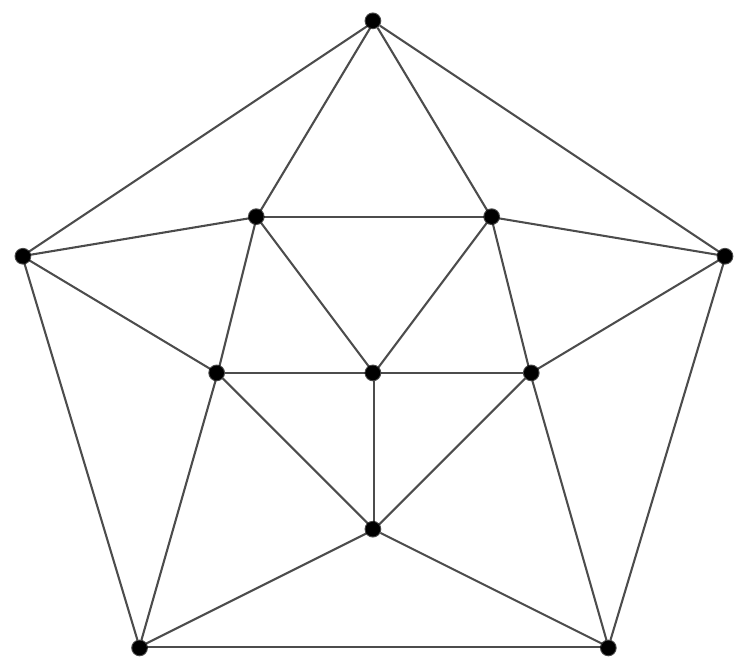}\\
$N=11$
\end{minipage}
\end{minipage}
\vspace{1em}

\begin{minipage}{.33\textwidth}
  \centering
\includegraphics[height=1.65in]{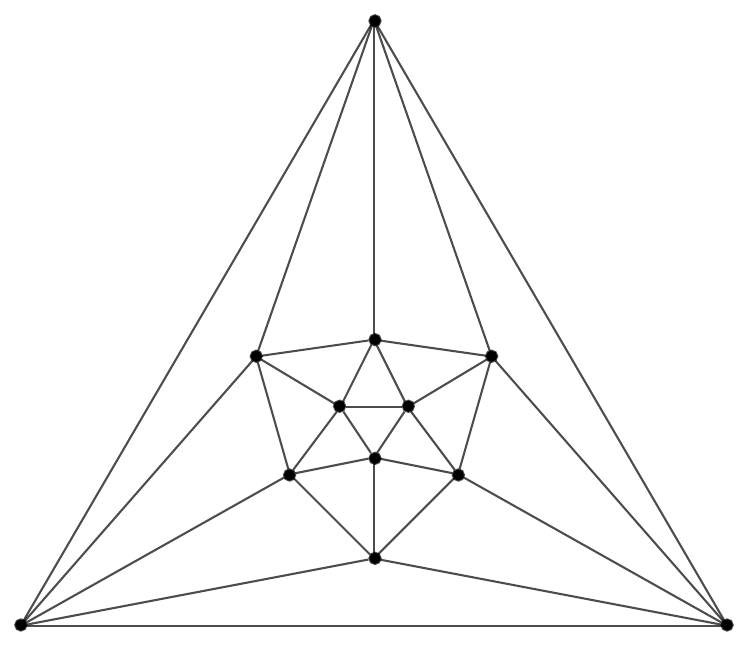}\\
$N=12$
\end{minipage}%
\begin{minipage}{.33\textwidth}
  \centering
\includegraphics[height=1.65in]{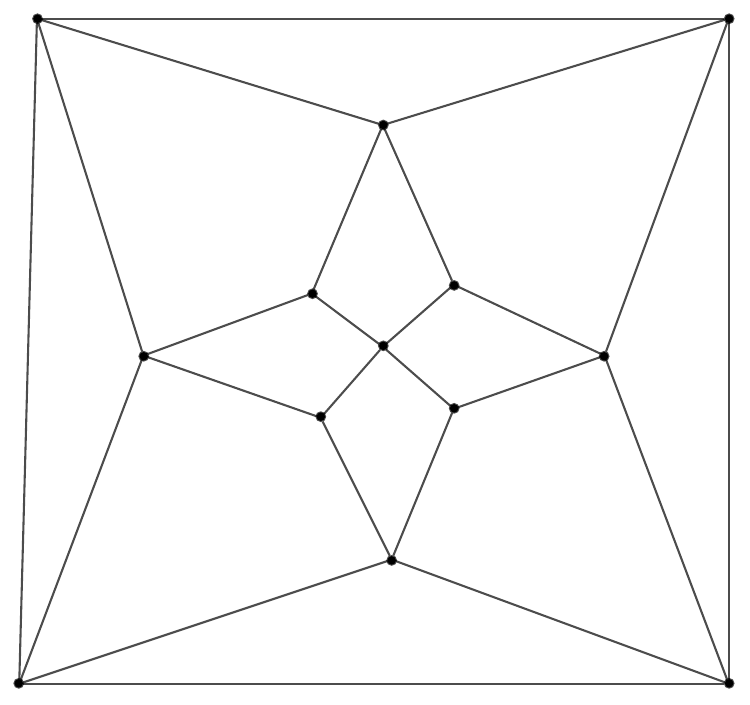}\\
$N=13$
\end{minipage}
\begin{minipage}{.33\textwidth}
  \centering
\includegraphics[height=1.65in]{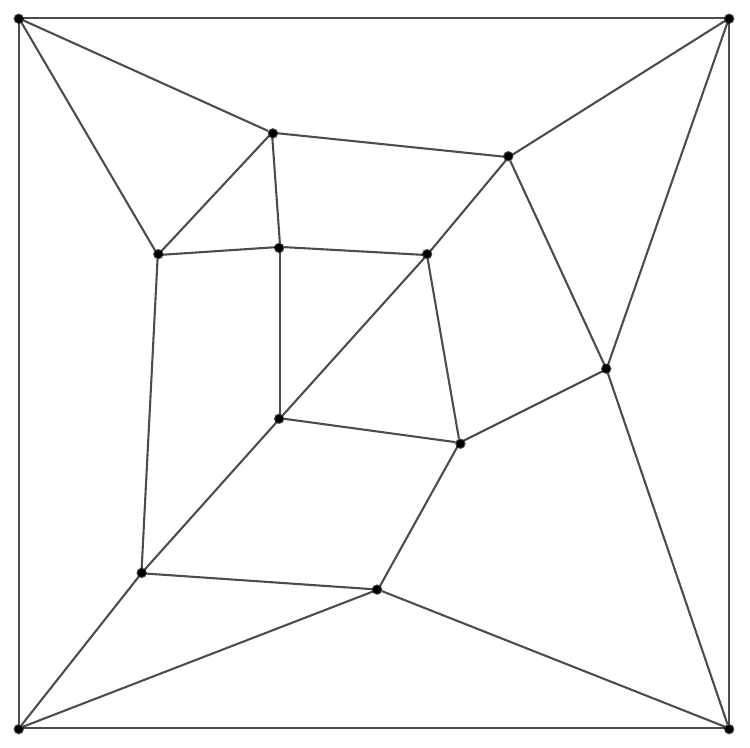}\\
$N=14$
\end{minipage}
\centering
\captionsetup{justification=centering}
   \caption{Optimal contact graphs associated to the Tammes configurations\newline for $N=3$ to $N=14$}
   \label{fig:contact}
\end{figure}

\subsection{The Tammes Problem Maximal Radius: General $N$}\label{sec:37}
The most basic question about the maximal radius in the Tammes problem concerns the distinctness of maximal values. The following conjecture was proposed by R. M. Robinson ~\cite[p. 297]{Robinson:1969}.

\begin{conjecture}\label{conj:74a}
{\rm (Robinson (1969))} 
For all $N \ge 4$ the  maximal radius satisfies 
\[
r_{max}(N) < r_{max}(N-1)
\]
except possibly for $N = 6, 12, 24, 48, 60$ and $120$.
\end{conjecture}

One has $r_{max}(N) = r_{max}(N-1)$  for $N=6$ and $N=12$ by results already given above.  
In 1991 Tarnai and G\'{a}sp\'{a}r ~\cite{TarnaiG:1991} established that $r_{max}(24) < r_{max}(23)$.   
The remaining cases $N\!\!=\!48$, $60$ and $120$ are open, but since  strict inequality holds for $n=24$ we expect strict inequality to hold for these values too.  
However up to now it has been computationally difficult to determine $r_{max}(N)$ for such large $N$.

We turn  to a potentially easier question.  
The known exact values of $r_{max}(N)$ are  algebraic numbers, i.e. roots of some univariate polynomial  having integer coefficients.  
Table ~\ref{tab722} below presents algebraicity data for $3 \le N \le 14$. 

 \begin{table}[h]
 \centering
\renewcommand{\arraystretch}{.85}
\begin{tabular}{|r|c|r|c|}
\hline &&&\\[-.7em]
\multicolumn{1}{|c|}{$N$} & \multicolumn{1}{c|}{$r_{max}(N)$} & \multicolumn{1}{c|}{Minimal Equation} & \multicolumn{1}{c|}{Figure}\\
[.3em]\hline&&&\\[-.7em]
3 & $3+ 2 \sqrt{3}$ & $X^2 -6X -3$ & Equilateral Triangle\\
~ & $\approx  6.4641$ &  ~ & ~ \\
[.3em]\hline&&&\\[-.7em]
4 &     $2 +\sqrt{6}$  &   $X^2-4X-2$  & Regular Tetrahedron\\
~ & $\approx 4.4495$ &     ~&  ~  \\
[.3em]\hline&&&\\[-.7em]
6 & $   1+ \sqrt{2}$ & $X^2 -2X-1$ & Regular Octahedron\\
~ &  $\approx 2.4142$ &     ~& ~ \\
[.3em]\hline&&&\\[-.7em]
7 &  $\approx 1.6913  $&$ X^6 - 6X^5- 3X^4 $ \quad\quad\quad\quad\quad& [No name]\\
~ &  ~ & $ + 8X^3 + 12X^2 + 6X +1$ ~& ~\quad\quad\\
[.3em]\hline&&&\\[-.7em]
8 & $ \approx 1.5496 $ & $X^4 - 8X^3 +4X^2 +8X+2$ & Square Antiprism\\
[.3em]\hline&&&\\[-.7em]
9 & $ \approx1.3660 $&$ 2X^2 - 2X - 1$ & [No name]  \\
[.3em]\hline&&&\\[-.7em]
10 & $ \approx 1.2013  $ &$ 4X^6 - 30X^5 +17X^4 $ \quad\quad\quad\quad& [No name]\\
~ &  ~& $+ 24X^3 -4X^2 - 6X - 1 $~&      \\
[.3em]\hline&&&\\[-.7em]
12& $  \frac{1}{\sqrt{\frac{5+ \sqrt{5}}{2}} -1}$ &  $X^4 - 6X^3 +X^2 +4X+1$ & Regular Icosahedron\\
~ & $\approx 1.10851$ & ~ & ~\\
[.3em]\hline&&&\\[-.7em]
24 & $ \approx  0.59262$ &   $X^6 -10X^5 +23X^4 $   \quad\quad\quad\quad & Snub Cube\\
~ & ~ & $+20X^3 -5X^2 -6X-1$ & ~\\
\hline
\end{tabular}
\vspace{.5em}
	\captionsetup{justification=centering}
\caption{The Tammes Problem radii given as algebraic numbers} 
\label{tab722}
\end{table}

\noindent We ask: {\em Is the optimal radius $r_{max}(N)$  an algebraic number for each $N \ge 3$?}
The reason to expect such an algebraicity result  to hold is that such a radius should be specified by at least one optimal graph that is {\em rigid}, i.e. it permits no local deformations preserving optimality, up to isometry.  
Danzer showed rigidity to be  the case for $7 \le N \le 12$.  
The equal length constraints of the edges of the contact graph give a system of polynomial equations with integer coefficients that the coordinates of the sphere centers must satisfy.  
One may expect that its real solution locus will  include some real algebraic solutions for the sphere centers, leading to algebraicity of the radius.  
Even if the rigidity result fails and deformations occur (as happens for $N=5$), it could still be the case that the optimal radius is algebraic. 

\section{Configuration Spaces of $N$ Spheres Touching a Central Sphere}\label{sec:4}
We start with the classical {\em configuration space} $\con(N) := \con(\SS^2,N)$ of $N$ distinct labeled points on the unit $2$-sphere $\SS^2$.  
One may regard these as the points where $N$ surrounding spheres touch a central sphere.  
Note that $\con(N)$ is an open submanifold of the $N$-fold product $(\SS^2)^N$ of unit spheres.  
We will also consider the {\em reduced configuration space}
\[
\bcon(N) := \con(N) / SO(3),
\]
which divides out the space $\con(N)$ by the orientation-preserving isometry group $SO(3)$ of the unit $2$-sphere $\SS^2$ in $\RR^3$.  
The elements of $SO(3)$ move all configurations to isometric configurations, and these moves are permitted on any configuration.  
The space $\con(N)$  is a non-compact $(2N)$-dimensional manifold and the space $\bcon(N)$ is a non-compact $(2N-3)$-dimensional manifold.  
We assume $N \ge 3$ to avoid degenerate cases.

We denote a configuration $\bU:= (\bu_1, \bu_2, ..., \bu_N)$, where the  $\bu_j \in \SS^2$ are distinct points.  
The {\em angular distance} between points $\bu_1, \bu_2 \in \SS^2$ is the angle $\bu_1, 0, \bu_2$ subtended at the center of the unit sphere that the $N$ spheres all touch; 
its value is at most $\pi$.

\begin{definition}\label{def:41}
The {\em injectivity radius function} $\rho: \, \con(\SS^2, N) \to \RR_{+}$ assigns to a configuration $\bU:= (\bu_1, \bu_2, \dots, \bu_N) \in (\SS^2)^N$ the value 
\[
\rho(\bU) := \frac{1}{2} \big( \min_{i \ne j} \theta(\bu_i, \bu_j) \big),
\]
where $\theta(\bu_i, \bu_j)$ denotes the angular distance between $\bu_i$
and $\bu_j$.  In particular $0 < \rho(\bU) \le \frac{\pi}{2}.$  Since the function $\rho$ is invariant under the action of $SO(3)$,
it yields a well-defined function on $\bcon(N; \theta)$, which we also denote $\rho$.
\end{definition}

Our main topic in this section is the study of spaces which are {\em superlevel sets} for the injectivity radius function $\ir$, and how these change at configurations which are critical for maximizing $\ir$. 

\begin{definition}\label{def:42}
We define the  {\em (constrained) angular configuration spaces}
\[
\con(N; \theta) = \con( \SS^2, N ; \theta) := \{ \bU= (\bu_1, ..., \bu_N) : \,  \theta(\bu_i, \bu_j) \ge \theta \, \,\mbox{for} \,\, 1 \le i < j \le N \}
\]
for angles $0< \theta \le \pi$, whose points label configurations of $N$ distinct marked labeled points $\bu_i$ which are pairwise at {\em angular distance} at least $\theta$ from each other.
Equivalently
\[
\con(N; \theta) :=  \{ \bU= (\bu_1, ..., \bu_N) :  \, \rho(\bU) \ge \frac{\theta}{2} \}.
\]
The {\em reduced (constrained) angular configuration spaces} $\bcon(N; \theta)$ are
\[
\bcon(N; \theta) := \con(N; \theta) / SO(3).
\]
This space is well-defined since rotations preserve angular distance.
\end{definition}

The spaces $\con(N; \theta)$ and $\bcon(N; \theta)$ are compact topological spaces.
Away from critical values $\theta$ the spaces $\con(N; \theta)$ are closed manifolds with boundary; at critical points they need not be manifolds.
The descriptions as superlevel sets show that the spaces $\con(N; \theta)$ are ordered by set inclusion as decreasing functions of $\theta$.  
If $\theta_1 > \theta_2$ then we have
\[
\con(N; \theta_1) \subset \con(N; \theta_2) \subset \con(N).
\]
We have similar inclusions for $\bcon(N; \theta)$.
{The interiors of these spaces are
\[
\con^{+}(N; \theta) :=   \{ \bU= (\bu_1, ..., \bu_N) :  \, \rho(\bU) > \frac{\theta}{2} \} 
\]
and
\[
\bcon^{+}(N; \theta):= \con^{+}(N; \theta)/ SO(3), 
\]
respectively.  
They are open manifolds for all values of the parameter $\theta$.}

The (constrained) angular configuration space $\con(N; \theta)$ can be reparametrized as the {\em (constrained) radial configuration space}
\[
\con(N)[r] := \con(\SS^2,N)[r]
\]
which consists of $N$ marked labeled spheres of equal radius $r= r(\theta)$ in $\RR^3$ which all touch a given unit radius central sphere $\SS^2$, with the $N$ touching points being the labeled points of the configuration.  
The radius $r(\theta)$ is determined by the condition that the spherical cap on the central sphere $\SS^2$ obtained by radial projection of all the points of the given touching sphere has angular diameter exactly $\theta$.
A configuration belongs to $\con(N; \theta)$ exactly when the $N$ spheres of radius $r(\theta)$ have disjoint interiors.  
There is a maximal angle $\theta_{max}:= \theta_{max}(N)$ for which $N$ spherical caps of that angular measure can fit on the surface of $\SS^2$ without overlap of their interiors.

Recall from the proof of Lemma ~\ref{lemma:31} that the value $r= r(\theta)$ is implicitly given by the equation
\[
\sin \frac{\theta}{2} = \frac{r}{1+r}.
\]
For $N \ge 3$ the function $r(\theta)$ is monotone increasing in $\theta$ up to a maximal value $r_{max}(\theta)$, so we may use either $r$ or $\theta$ to parametrize the family of all spaces $\con( N; \theta)$ or $\con(N)[ r]$.  
Note that we can identify the configuration spaces $\con(N)$ with $\con(N; 0)$, as well as with $\cup_{\theta>0}\con(N; \theta)$ and $\cup_{r>0}\con(N)[r]$.

In the following subsections we review the topology and geometry of the spaces $\con(N; \theta)$ and $\bcon(N; \theta)$:

\begin{itemize}
\item
In Section ~\ref{sec:41} we describe results on the homotopy type and cohomology of the configuration spaces $\con(N)$ and reduced configuration spaces $\bcon(N)$, which are well studied.

\item
In Section ~\ref{sec:42} we use ideas from Morse theory, applied to the injectivity radius function $\ir$, to study general features of the change in topology for fixed $N$ as the angular parameter $\theta$ (or radius parameter $r$) is increased. 
The topology changes at certain {\em critical values} of $\theta$.  
Since the injectivity radius function $\ir$ is semi-algebraic, for each $N$ we expect there to be a finite set of critical values of $\ir$ on $\bcon(N)$.  
We give a balancing criterion for a configuration in $\bcon(N)$ to be critical. 

\item
In Section ~\ref{sec:43} we show that for small enough $\theta$, the angular configuration space $\con(N; \theta)$ has the same homotopy type as $\con(N)$ and hence the same cohomology.

\item
In Section ~\ref{sec:43a} we treat large $\theta$ near $\theta_{max}$, in terms of the radius parameter $r$.
For larger angular diameter $\theta$, the topology of $\con(N; \theta)$ may differ drastically from that of $\con(N)$.  
For example, in Table ~\ref{tab711} for the Tammes problem many of the maximal configurations are isolated, and the associated labeled spheres cannot be continuously interchanged for large $\theta$ near $\theta_{max}(N)$.  
In such situations, the space $\con(N; \theta)$ is disconnected for large $\theta$, while the space $\con(N)$ is connected.
We conjecture  that near $\theta_{max}(N)$ the cohomology of $\bcon(N)[r]$ is concentrated in dimension $0$ and discuss the associated Betti number.

\item
{In Section ~\ref{sec:46} we show by examples that the set of critical configurations at a critical  value can have many connected components and can have variable dimension.}

\item
In Section ~\ref{sec:47} we treat the case of topology change as $\theta$ varies for the simplest case $N=4$.
We determine all the critical values for the injectivity radius function $\ir$ on $\bcon(N)$.

\item
In Section ~\ref{sec:48} we briefly discuss  topology change in cases $N \ge 5$.
The study of properties of the  $N=12$ case is deferred to Sections ~\ref{sec:5} and ~\ref{sec:6}.
\end{itemize}

\subsection{Topology of Configuration Spaces}\label{sec:41}

Configuration spaces $\con(\XX, N)$ of $N$ distinct, labeled points on a $d$-dimensional manifold $\XX$ have been studied as fundamental spaces in topology.  
Recall that the {\em configuration space}  of labeled $N$-tuples on a manifold $\XX$ is
\[
\con(\XX, N) :=\{ (x_1, x_2, ..., x_N) \in \XX^N :\, x_i \ne x_j  \, \textrm{{\mbox if}} \,\, i\ne j\}.
\]
The symmetric group $\Sigma_N$ acts freely on the space $\con(\XX, N)$ to permute the points, and
\[
B(\XX, N) := \con(\XX, N)/\Sigma_N
\]
is the configuration space of {\em unlabeled} (i.e. {unordered}) $N$-tuples of points on $\XX$. 
Configuration spaces of this type were first considered in the 1960's by Fadell and Neuwirth ~\cite{FadellN:1962, Fadell:1962}. 
The state of the art for $\XX =\RR^d$ and $\SS^d$ as of 2000 is given in Fadell and Husseini ~\cite{FadellH:2001}.
Other useful references are Totaro ~\cite{Totaro:1996} and Cohen ~\cite{Cohen:2010}.

We begin with the most well-known of these spaces, the configuration space $\con(\RR^2, N)$ of labeled $N$-tuples of points on  $\XX = \RR^2$, the plane. 
Fadell and Neuwirth ~\cite{FadellN:1962} showed that the unlabeled configuration space $B(\RR^2, N)$ is a classifying space (Eilenberg-Maclane space ~$K(\pi, 1)$), with fundamental group $\pi_1=\pi$ isomorphic to Artin's {\em braid group} $B_N$ on $N$ strings. 
Thus the cohomology of $B(\RR^2, N)$ is just the cohomology of $B_N$; it was computed by Fuks ~\cite{Fuks:1970} and Cohen ~\cite{Cohen:1988}.  
 
 The labeled configuration space $\con(\RR^2, N)$ is by definition the complement of a finite set of complex hyperplanes given by $x_i =x_j$ for $i \ne j$ in $\CC^N \cong (\RR^2)^N$.  
 This arrangement is sometimes called the (complexified) $A_{N-1}$-arrangement of hyperplanes (see Postnikov and Stanley ~\cite{PostnikovS:2000}), where $A_{N-1}$ refers to a Coxeter group.  
 The space $\con(\RR^2, N)$ is also a classifying space with fundamental group equal to the {\em pure braid group} $P\kern-.2emB_N$, the subgroup of $B_N$ consisting of all $N$-strand braids which induce the identity permutation.  
 It sits in a short exact sequence $0 \to P\kern-.2emB_N \to B_N \to \Sigma_N \to 0$.  The rational cohomology of $\con(\RR^2, N)$ is then the cohomology of the pure braid group; the cohomology ring structure of $P\kern-.2emB_N$ was determined by Arnold ~\cite{Arnold:1969} in 1969.

The {\em Betti numbers} of a topological space are the ranks of its homology groups (which equal the ranks of its cohomology groups,
with  coefficients in a field, here $\QQ$ or $\CC$.)   
The generating function for this sequence of ranks is called the {\em Poincar\'e polynomial}.  
Arnold ~\cite[Corollary ~2]{Arnold:1969} determined the Poincar\'{e} polynomial for the pure braid group $P\kern-.2emB_N$ on $N$ strands to be 
\begin{equation}\label{braid-poly}
P_N(t) = (1+t) (1+ 2t) \cdots (1+ (N-1) t).
\end{equation}
Table ~\ref{table3-braid} gives these Betti numbers for small $N$.
They are of combinatorial interest, being unsigned Stirling numbers of the first kind,
\[
\dim H^k(P\kern-.2emB_N, \QQ) = \big{[} {{N}\atop{N-k}} \big{]}, \quad \mbox{ for} \quad 0 \le k \le N-1
\]
(see ~\cite[Sect. 6.1]{GKP:1994}).

\begin{table}[h]
\renewcommand{\arraystretch}{0.8}
\begin{center}
	\begin{tabular}{| c | r |r | r| r | r |r |r|r|r|r|r|}
	\hline
	 \backslashbox{$N$} 
	 {\raisebox{-1mm}{$k$}} &  $0$ & $1$ & $2$ & $3$ & $4$ & $5$ & $6$  & $7$ & $8$ \\ 
	\hline &&&&&&&&&\\[-.7em]  
	$1$ & $1$ &$0$ & $0$ & $0$ & $0$ & $0$ & $0$ & $0$ & $0$  \\
        $2$ &$1$ & $1$ & $0$ & $0$ & $0$ & $0$ & $0$ & $0$ & $0$  \\
        $3$ & $1$ &$3$ & $2$ & $0$ & $0$ & $0$ & $0$ & $0$ & $0$  \\
	$4$ & $1$ &$6$ &$11$ & $6$ & $0$ & $0$ & $0$ & $0$ & $0$  \\
	$5$ & $1$ &$10$ &$35$ & $50$ & $24$ & $0$ & $0$    & $0$ & $0$  \\
	$6$ & $1$ &$15$ &$85$ & $225$ & $274$ & $120$ & $0$        & $0$ & $0$ \\
	$7$ & $1$ &$21$ &$175$ & $735$ & $1624$ & $1764$  & $720$ & $0$ & $0$ \\
        $8$ & $1$ &$28$ &$322$ & $1960$ & $6769$ & $13132$   & $13068$ & $5040$ & $0$  \\
        $9$ & $1$ &$36$ &$546$ & $4536$ & $22449$ & $67284$ & $118124$   & $109584$ & $40320$\\
	\hline
	\end{tabular}
	\vspace{.5em}
		\captionsetup{justification=centering}
	\caption{Betti numbers of pure braid group cohomology\newline $H^k(P\kern-.2emB_N, \QQ) \cong H^k(\bcon(\RR^2, N), \QQ)$} 
	\label{table3-braid}
\end{center}
\end{table}

Our interest here is  configuration spaces on 
\[
\XX= \SS^2= \{\bu = (x, y, z): \bu \cdot \bu = x^2+y^2+z^2=1\},
\]
 the unit $2$-sphere  embedded in $\RR^3$.
The  configuration spaces of $\SS^2$ have a close relationship to those of $\RR^2$, since $\SS^2$ is the one-point compactification of $\RR^2$.
Note also that $\SS^2$ is homeomorphic to $\mathbb{CP}^1$, the complex projective line. 
Tautologically the configuration space $\con(\SS^2, 1)\cong\SS^2$; and the space $\con(\SS^2, 2)$ is homeomorphic to an $\RR^2$-bundle over $\SS^2$, hence homotopy equivalent to $\SS^2$ (see ~\cite[Example 2.4]{Cohen:2010}).  
For $N \ge 3$ points we have a well-known result, given as follows in Feichtner and Ziegler ~\cite[Theorem 1]{FZ:2000}. 

\begin{theorem}\label{thm:411}
For $N \ge 3$ the configuration space $\con(\SS^2, N)$ of $N$ distinct labeled points on the $2$-sphere is the total space of a trivial $PSL(2, \CC)$-bundle over $\sM_{0,N}$, the moduli space of conformal structures on the $N$-punctured complex projective line, modulo conformal automorphisms.
Hence there is a homeomorphism
\[
\con(\SS^2, N) \cong PSL(2, \CC) \times \sM_{0,N}.
\]
\end{theorem}

\noindent
Note that $PSL(2, \CC)$ is homotopy equivalent to its maximal compact subgroup $SO(3)$, the group of orientation preserving isometries of $\SS^2$.

The $SO(3)$-action on $\con(\SS^2,N)$ permits us to rotate the first point to the north pole, from which we stereographically project the rest of the unit sphere to the plane $\RR^2$.
Since we are still free to rotate about the north pole, which corresponds to $SO(2)$ rotations in the plane, we can identify  $\con(\RR^2, N-1)/SO(2)$ with the reduced $N$-configuration space $\bcon(\SS^2,N)$.  
The action of $SO(2)$ on $\con(\RR^2,N-1)$ is free if $N\ge3$, so we can regard $\con(\RR^2,N-1)$ as a principal $SO(2)$-bundle over $\bcon(\SS^2,N)$. 

This principal bundle has a section, and thus is a product bundle, so the Poincar\'{e} polynomial  $\tilde{P}_N(t)$ of the base $\bcon(\SS^2,N)$ may be computed as the quotient of the well-known Poincar\'e polynomials $P(t)=P_{N-1}(t)=\big(1+t\big)\big( 1+ 2t\big) \cdots\big(1+(N-2)t\big)$ for $\con(\RR^2, N-1)$ from Equation \eqref{braid-poly} and $p(t)=(1+t)$ for $SO(2)$.  
It follows that $\bcon(\SS^2,N)$ and $\cM_{0,N}$ both have Poincar\'e polynomial
\begin{equation}\label{eq4102}
\tilde{P}_N(t)=P(t)/p(t)=\big(1+2t\big)\cdots\big(1+(N-2)t\big).
\end{equation}

\noindent We give the first few Betti numbers in the following table:
\vspace{.6em}
\begin{table}[h]
\renewcommand{\arraystretch}{0.9}
\begin{center}
	\begin{tabular}{| c | r |r | r| r | r |r |r|r|r|r|r|}
	\hline
	\backslashbox{\raisebox{.5mm}{$N$}}{\raisebox{-1mm}{$k$}}  &  $0$ & $1$ & $2$ & $3$ & $4$ & $5$ & $6$  & $7$ & $8$ &$9$\\ 
	\hline &&&&&&&&&&\\[-.7em]
	$3$ & $1$ &$0$ & $0$ & $0$ & $0$ & $0$ & $0$ & $0$ & $0$ & $0$\\
	$4$ & $1$ &$2$ &$0$ & $0$ & $0$ & $0$ & $0$ & $0$ &  $0$ & $0$\\
	$5$ & $1$ &$5$ &$6$ & $0$ & $0$ & $0$ & $0$ & $0$ & $0$  & $0$\\
	$6$ & $1$ &$9$ &$26$ & $24$ & $0$ & $0$ & $0$  & $0$ & $0$ & $0$\\
	$7$ & $1$ &$14$ &$71$ & $154$ & $120$ & $0$  & $0$ & $0$ & $0$ & $0$\\
	$8$ & $1$ &$20$ &$155$ & $580$ & $1044$ & $720$   & $0$ & $0$ & $0$ & $0$\\
        $9$ & $1$ &$27$ &$295$ & $1665$ & $5104$ & $8028$   & $5040$ & $0$ & $0$ & $0$\\
        $10$ & $1$ &$35$ &$511$ & $4025$ & $18424$ & $48860$    & $69264$ & $40320$ & $0$ & $0$\\
        $11$ & $1$ & $44$ & $826$ & $8624$ & $54649$ & $214676$ & $509004$ & $663696$ & $362880$ & $0$\\
	$12$ & $1$ & $54$ & $1266$ & $16884$ & $140889$ & $761166$ & $2655764$ & $5753736$ & $6999840$ & $3628800$\\
	\hline
	\end{tabular}
	\vspace{.5em}
	\captionsetup{justification=centering}
	\caption{Betti numbers for  reduced configuration space cohomology $H^k(\bcon(\SS^2, N); \QQ)$}
	\label{table4}
\end{center}
\end{table}

\noindent 
By taking the alternating sum of each row, or more directly by evaluating $P_N(-1)$, we can compute the Euler characteristic
\begin{equation}
\chi(\bcon(N))=(-1)^{N-3} (N-3)!
\end{equation}
of $\bcon(N)$ for $N \ge 3$.

We also have ~\cite[Proposition 2.3]{FZ:2000}: 

\begin{theorem}\label{thm:412} {\rm (Feichtner-Ziegler (2000))}
For $N \ge 3$ the moduli space $\sM_{0,N}$ is homotopy equivalent to the complement of the affine complex braid arrangement of hyperplanes $\sM( {}^{\textrm{aff}}\sA_{N-2}^{\CC})$ of rank $N-2$, since
\[
\sM_{0,N} \times \CC \simeq \sM( {}^{\textrm{aff}}\sA_{N-2}^{\CC}).
\]
Its integer cohomology algebra is torsion-free. 
It is generated by $1$-dimensional classes $e_{i,j}$ with $1 \le i < j \le N-1$ with $(i, j) \ne (1, 2)$ and has a presentation as an exterior algebra
\[
H^{\ast} (\sM({}^{\textrm{aff}}\sA_{N-2}^{\CC}))\cong \Lambda^{*} \ZZ^{{{N-1}\choose{2}} -1}/ \sI,
\]
where the ideal $\sI$ is generated by elements
\[ 
e_{1, i} \wedge e_{2, i}, \quad 2<i \le N-1
\]
and
\[ 
e_{i, \ell} \wedge e_{j, \ell} - e_{i, j} \wedge e_{j, \ell} + e_{i,j} \wedge e_{i, \ell} \quad 1 \le i < j < \ell \le N-1, \, (i, j) \ne (1,2).
\]
\end{theorem}
\noindent
Here the complexified $A_{N-2}$-arrangement of hyperplanes $\sA_{N-2}^{\CC}$ of rank $N-2$ is cut out by the hyperplanes
\[
z_i - z_j =0 \quad 1\le i<j \le N-1.
\]
Its complement $\sM(\sA_{N-2}^{\CC} ) := \CC^{N-1} \smallsetminus \bigcup \sA_{N-2}^{\CC}$ is homeomorphic to $\con(\CC, N-1)$. The associated affine arrangement is:
\[
{}^{\textrm{\it{aff}}}\sA_{N-2}^{\CC} := \{ (z_1, z_2, \cdots, z_{N-1} \in \sA_N^{\CC}:  z_2 - z_1 = 1\}.
\]
Treating $\CC^{N-2} \cong \{ (z_1, z_2, \cdots, z_{N-1}) : z_2 - z_1 =1\}$, we set
\[
\sM ({}^{\textrm{\it{aff}}}\sA_{N-2}^{\CC}) := \CC^{N-2} \smallsetminus {}^{\textrm{\it{aff}}}\sA_{N-2}^{\CC}.
\]

A more refined result determines the integral cohomology ring for the configuration spaces of spheres, which includes torsion elements. 
It was determined by Feichtner and Ziegler, who obtained in the special case of $\con(\SS^2,N)$ the following result (see ~\cite[Theorem 2.4]{FZ:2000}). 

\begin{theorem} \label{F-Z}
{\rm (Feichtner-Ziegler (2000))}
For  $N\ge3$, the  integer cohomology ring\\
$H^{\ast} ( \con(\SS^2,N), \ZZ)$ has only $2$-torsion. 
It is given as
\[
H^{\ast}( \con(\SS^2,N), \ZZ)  \cong ( \ZZ(0) \oplus \ZZ/ 2\ZZ(2) \oplus \ZZ(3)) ) \otimes
\Lambda^{\ast} (\oplus_{i=1}^{{{N-1}\choose{2}} -1}\ZZ(1))/\sI,
\]
in which $\sI$ is the ideal of relations given in {\rm Theorem ~\ref{thm:412}}.
\end{theorem}

In this result the expression $G(i)$ denotes a direct summand of $G$ in cohomology of degree~$i$, e.g. there is a $\ZZ/ 2\ZZ$ direct summand in $H^2( \con(\SS^2,N), \ZZ)$.

\subsection{Generalized Morse Theory and Topology Change}\label{sec:42}
{Morse theory}, as treated in Milnor ~\cite{Milnor:1963}, concerns how topology changes for the {\em sublevel sets}
\[
\UU^t := \{ u \in \UU: \, f(u) \le t\}  
\]
of a given, sufficiently nice, real-valued function $f$ on a manifold $\UU$, as the level set parameter ~$t$ varies.
At the {\em critical values} of the function, where its gradient vanishes, the topology changes.
This change can be described  by adding up the contributions of individual {\em critical points} of the function that occur at the critical values.
More precisely, a  {\em Morse function} is a smooth enough function that has only isolated {\em critical points}, each of  which is  non-degenerate, and arranged so that only one critical point occurs at each critical level $f(u) = t$. 
Here {\em non-degenerate} means that the function $f$ is twice-differentiable and its Hessian matrix $[\frac{\partial^2\!f}{\partial u_i \partial u_j}]$ is nonsingular at the critical point. The topology of a sublevel set  $\UU^t$ is changed as $t$ ascends past a critical value, up to homotopy, by attaching a cell of dimension equal to the \emph{index} of the critical point: the number of negative eigenvalues of the Hessian.

Our interest here will be in 
 {\em superlevel sets} 
 \[
 \UU_{s} : =  \{ u\in \UU: \, f(u) \ge s\}, 
 \]
whose topology changes as $s$ descends past a critical value by attaching a cell of dimension equal to the \emph{co-index} of the critical point: the number of positive eigenvalues of the Hessian.

In the 1980's Goresky and MacPherson ~\cite{GoreskyM:1988} developed Morse theory on more general topological spaces than manifolds, namely {\em stratified spaces} in the sense of Whitney ~\cite{Whitney:1964}, and applicable  to a wider class of real-valued functions.
The configuration spaces such as $\con(N; \theta)$ studied here are in general stratified spaces in Whitney's  sense, because viewed using the $r$-parameter they are real semi-algebraic varieties. 

For the case at hand of $\UU=\con(\SS^2,N)$ and the injectivity radius function $\ir$, we have a further problem that $\rho$  is not a Morse function. 
Its critical points are degenerate and non-isolated, and even the notion of ``critical'' needs care in defining, since $\ir$ is a min-function of a finite number of smooth functions (see Definition ~\ref{def:41}). 
Technically, the angular distance function from $\bu$ is not smooth at the antipodal point $-\bu$, with angular distance $\pi$ on $\SS^2$;
however we can treat these functions as if they were smooth using the following trick, valid for the nontrivial cases $N\ge3$ where $\ir_{\max} \le\frac{\pi}{3}$: 
simply include the constant function $\frac{\pi}{3}$ among those functions over which we take the min, and smoothly cut off the other pairwise angular distance functions $\theta(\bu_i, \bu_j)$ if they exceed $\frac{2\pi}{3}$.

An appropriate version of Morse theory that applies in this context, {called {\em min-type Morse theory}, has only recently been sketched by Gershkovich and Rubinstein ~\cite{GR97} (see also Baryshnikov et al. ~\cite{BBK:2014}).
Related work includes Carlsson et al. ~\cite{CGZKM:2012} and Alpert ~\cite{Alpert:2016}.} 
The treatment of ~\cite{BBK:2014} studies  a notion of topologically critical value.

In what follows we develop an {alternative} max-min approach to criticality and a Morse theory for the injectivity radius function $\ir$ on configurations that is in the spirit of the criticality theory for maximizing {\em thickness} or normal injectivity radius (also known as {\em reach}) on configurations of curves subject to a length constraint (or in a compact domain of ~$\RR^3$, or in ~$\SS^3$) studied earlier in optimal ropelength and rope-packing problems by Cantarella ~et ~al. ~\cite{CFKSW:2006}.
This approach provides a notion of critical configuration, refining the notion of  a critical value. 
The Farkas Lemma (and its infinite-dimensional generalizations in the case of the ropelength problem) is a key tool used in these works that relates criticality to the existence of a balanced system of forces on the configuration. A more detailed treatment is planned  in ~\cite{KKLS16+}. 

To understand criticality for the injectivity radius function $\ir$ on $\con(N)=\con(\SS^2; N)$, we first need to make sense of varying a configuration $\bU=(\bu_1,...,\bu_N) \in \con(N)\subset (\SS^2)^N$ along a tangent vector $\bV=(\bv_1,...,\bv_N)$ to $\con(N)$ at $\bU$; 
here $\bv_i$ is a tangent vector to $\SS^2$ at $\bu_i$, for $i=1, 2, ... , N$.
For sufficiently small $t$ we can define a nearby configuration
\[
\bU\#t\bV:=\left(\frac{\bu_1+t\bv_1}{| \bu_1+t\bv_1|},..., \frac{\bu_N+t\bv_N}{| \bu_N+t\bv_N|}\right) \in \con(N)\subset (\SS^2)^N 
\]
by translating and projecting each factor back to $\SS^2$. 
In particular, the $\bV$-directional derivative $f'$ of a smooth function $f$ on $\con(N)$ at $\bU$ is simply $f'=\frac{d}{dt}|_{t=0} f(\bU\#t\bV)$, so $\bU$ is a critical point for smooth $f$ provided all its $\bV$-directional derivatives vanish at $\bU$; this means that the increment $f(\bU\#t\bV)-f(\bU) = o(t)$, where $o(t)$ is a function which tends to $0$ faster than linearly.
\begin{remark} The operation taking $\bU$ to $\bU\#t\bV$ can be thought of as the spherical analog of translating $\bU$ by $t\bV$ via vector addition in the linear case, hence the suggestive sum notation.  The map taking $t\bV$ to $\bU\#t\bV$ approximates (to within $o(t^2)$) the exponential map at $\bU$. \end{remark}

Now we make precise ``max-min criticality'' for the injectivity radius function~$\ir$. 

\begin{definition}\label{def:46}
A configuration  $\bU=(\bu_1,...,\bu_N) \in \con(N)$ is {\em critical for maximizing} $\ir$ provided for every tangent vector $\bV=(\bv_1,...,\bv_N)$ 
to $\con(N)$ at  $\bU$ 
 we have, as $t \to 0$,
\[
[\ir(\bU\#t\bV)-\ir(\bU)]_+ = o(t),
\]
where $[g]_+=\max\{g,0\}$ denotes the positive part of $g$.  
Equivalently, a configuration $\bU$ is critical if {\em no} variation $\bV$ can {\em increase} $\ir$ to first order. 
\end{definition}
\noindent

Otherwise, a configuration $\bU$  is {\em regular}, that is, there exists a variation $\bV$ which {\em does} increase $\ir$ to first order, and so, by the definition of $\ir$ as a min-function, this means that for all pairs $(\bu_i, \bu_j)$ realizing the minimal angular distance $\theta(\bu_i, \bu_j)= \theta_o$, their distances increase to first order under the variation $\bV$ as well.  
Note that the set of regular configurations is open.
If each configuration in this $\{\ir=\frac{\theta_o}{2}\}$-level set is regular, then this level is {\em topologically regular:}
that is, there a deformation retraction from $\con(N;\frac{\theta_o}{2}-\varepsilon)$ to $\con(N;\frac{\theta_o}{2}+\varepsilon)$ for some $\varepsilon>0$ (see ~\cite[Lemmas 3.2, 3.3 and Corollary 3.4] {BBK:2014}).

\begin{definition}\label{def:47}
For $\bU \in \con(N; \th)$, the {\em contact graph} of $\bU$ is the graph embedded in $\SS^2$ with vertices given by points $\bu_i$ in $\bU$ and edges given by the geodesic segments $[\bu_i, \bu_j]$ when $\theta(\bu_i , \bu_j) =\theta$.
\end{definition}
\noindent Examples of contact graphs for extremal values of the Tammes problem were given in Figure~\ref{fig:contact} of Section~\ref{sec:3}. 


\begin{definition}\label{def:48}
 A {\em stress graph} for $\bU \in \con(N; \th)$ is a contact graph with nonnegative weights $w_e$ on each geodesic edge $e=[\bu_i, \bu_j]$.
\end{definition}
\noindent
A stress graph gives rise to a system of {\em tangential forces} associated to each geodesic edge $e=[\bu_i, \bu_j]$ of the contact graph. 
These forces have magnitude $w_e$, are tangent to $\SS^2$ at each point $\bu_i$ of $\bU$, and are directed along the outward unit tangent vectors $T_e\!\!\mid_{\bu_i},T_e\!\!\mid_{\bu_j}$ to the edge $e$ at its endpoints  $\bu_i, \bu_j$, respectively.  

\begin{definition} 
A stress graph is  {\em balanced} if the vector sum of the forces in the tangent space of $\SS^2$ at $\bu_i$ is zero for all points of $\bU.$  
A configuration $\bU$ is {\em balanced} if its underlying contact graph has a balanced stress graph for some choice of non-negative, not-everywhere-zero weights on its edges.
\end{definition}

\begin{theorem} \label{thm:contact}
To each critical value $\frac{\theta}{2}$ for the injectivity radius $\ir$, there exists a balanced configuration $\bU$ with $\ir(\bU)=\frac{\theta}{2}$. 
The vertices of the contact graph are a subset of the points in $\bU$ and the geodesic edges of the contact graph all have length $\theta$.
\end{theorem}
\begin{proof}
As in ~\cite[Corollary 3.4 and Equation 2]{BBK:2014}, since $\ir$ is a min-function on $\con(N)\subset(\SS^2)^N$, if $\frac{\theta}{2}$ is not a topologically regular value of $\ir$, then some configuration $\bU \in \ir^{-1}(\frac{\theta}{2})$ is balanced. 
Because $\ir(\bU) = \frac{\theta}{2}$, the conditions on the vertices and edge lengths are clearly met.
\end{proof}

We now prove a converse result.

\begin{theorem} \label{thm:converse}
 If a configuration ${\bf U}$ on $\SS^2$ is balanced, then ${\bf U}$ is critical for maximizing the injectivity radius $\ir$.
\end{theorem}

We will need a preliminary lemma.
Consider a planar graph $G$ embedded on the unit sphere $\SS^2$ via a map $\bu: 
G\rightarrow\SS^2$ which is $C^2$ on the edges of $G$.  
(By slight abuse of notation, a point on its image in $\SS^2$ may also be denoted by $\bu$.)
Suppose each edge $e$ of $G$ is assigned a {\em nonzero} weight $w_e\in \RR$.  
Let $L_e(\bu)$ denote the length of edge $\bu(e)$ induced by the map $\bu$, and let $\sL(\bu)=\sum w_e L_e(\bu)$ be the total {\em weighted length} of the embedded graph $\bu(G)$. 
We can vary the map $\bu$ using a $C^2$ vector field $\bv$, just as we varied a configuration: 
for sufficiently small $t$, each point $\bu \in \SS^2$ on the image of the graph is moved to $\bu\#t\bv=\frac{\bu+t\bv}{|\bu+t\bv|}$.
Let $\sL'(\bv)$ denote the first derivative at $t=0$ of weighted length for this varied graph, i.e. the {\em first variation} of $\sL(\bu)$ along $\bv$. 

\begin{lemma}\label{lem:412}
The first variation $\sL'(\bv)$ of the weighted lengh $\sL(\bu)$ for the embedded graph $\bu(G)$ vanishes for every vector field $\bv$ on $\SS^2$ if and only if the following two conditions hold:

$(1)$ each edge $e$ joining a pair of vertices $e^-, e^+$ of $G$ maps to a geodesic arc $\bu(e)=[\bu(e^-),\bu(e^+)]=[\bu^-,\bu^+]$ in the embedded graph $\bu(G)$; 

$(2)$ at any vertex $\bu$ of the embedded graph $\bu(G)$, the weighted sum $\sum w_{e^*} T_{e^*}\!\!\mid_\bu=0$, where the sum is taken over the subset of edges $\bu(e^*)$ incident to $\bu$, and where $T_{e^*}\!\!\mid_\bu$ is the outer unit tangent vector of $\bu(e^*)$ at $\bu$.
\end{lemma}

\begin{proof}  This lemma is a direct consequence of the first variation of length formula 
\[
L'_e(\bv) = \bv\cdot T\!\! \mid_{\bu(e^-)}^{\bu(e^+)} - \int_{\bu(e)}{\bv}\cdot {\bf k}  
\]
(see, for example,  Hicks ~\cite[Chapter 10, Theorem 7, page 148]{Hicks:1965}). 
Here $T$ is the unit tangent vector field of the edge $\bu(e)$, and ${\bf k}$ is the geodesic curvature vector of $\bu(e)$; 
with respect to any local arclength parameter on $\bu(e)$, the geodesic curvature vector is the projection to $\SS^2$ of the acceleration: 
${\bf k} = \ddot{\bu} + \bu$, which is tangent to $\SS^2$ and normal to $\bu(e)$, and which vanishes iff $\bu(e)$ is a geodesic arc.

Now express $\sL'(\bv)=\sum w_e L'_e(\bv)$ as a sum of edge terms and vertex terms. 
The geodesic arc condition (1) -- that ${\bf k}=0$ along every edge -- implies the edge terms in $\sL'(\bv)$ all vanish for any variation $\bv$ of the map $\bu$; 
and the force balancing condition (2) implies all vertex terms vanish for any variation $\bv$.

Conversely, given any interior image point $\bu$ of an edge, take a variation $\bv$ supported in an arbitrarily small neighborhood of $\bu$, and orthogonal to $\bu(e)$ at $\bu$: 
the vanishing of $\sL'(\bv)$ implies condition (1) that ${\bf k}=0$; 
similarly, at any given vertex ${\bu}$, consider a pair of variations $\bv_1, \bv_2$ supported in an arbitrarily small neighborhood of  $\bu$ which approximate an orthogonal pair of translations of the tangent space to $\SS^2$ at $\bu$: the vanishing of $\sL'(\bv)$ for both of these $\bv_1, \bv_2$ implies the forces balance (2).
\end{proof}

\begin{remark}
\noindent In case $w_e=0$, vanishing for the first variation of $\sL$ does not imply $\bu(e)$ is a geodesic arc: 
instead, the edges with nonzero weights form a balanced geodesic subgraph of the original embedded graph $\bu(G)$. \end{remark}

\noindent  
Lemma ~\ref{lem:412} suggests the following definition. 

\begin{definition}\label{def:413}
An embedded graph satisfying properties (1) and (2) is called a {\em balanced geodesic graph.}  
(Note that there is no requirement here that the geodesic edge lengths are integer multiples of some basic length, as would be the case for a contact graph.)
\end{definition}

\noindent Lemma ~\ref{lem:412} shows that a  balanced geodesic graph has vanishing first variation of weighted length $\sL$, even if some of its edge weights $w_e$ are zero. 

\begin{proof}[Proof of Theorem ~\ref{thm:converse}]
By hypothesis, there are non-negative edge weights (not all zero) so that the resulting stress graph $\bu(G)$ for the configuration {\bU} is balanced.  
By Lemma ~\ref{lem:412} the first variation $\sL'(\bv)$ of weighted length for $\bu(G)$ vanishes for all variation vector fields $\bv$ on $\SS^2$.  

Suppose (to the contrary) that {\bU} were {\em not} critical for maximizing the injectivity radius $\ir$.  
Then there would be a variation ${\bV}$ of ${\bU}$ so that every geodesic edge of the stress graph has length increasing at least linearly in ${\bV}$.  
Extend ${\bV}$ to an ambient $C^2$ variation vector field $\bv$ on $\SS^2$. 
Since the edge weights are {\em nonnegative}, and not all zero, that implies the weighted length of the stress graph also increases at least linearly in $\bv$, a contradiction.
\end{proof}

\begin{remark}  
A key property of balanced configurations is that for each $N \ge 3$  the  set of radii $r$ such that $\bcon(N)$ contains a balanced configuration $\bU$ of injectivity radius $\frac{\theta(r)}{2}$ is finite. 
It follows that {\em the set of critical radius values for $\bcon(N)$ is finite}. 

This finiteness result can be proved using the structure of the spaces $\bcon(N)[r]$ as real semi-algebraic sets, which we consider in ~\cite{KKLS16+}.  
We will assume this finiteness result holds in the discussions in ~\ref{sec:43a};
it can be directly verified for small $N$.
\end{remark} 

\subsection{Small Radius Case}\label{sec:43}
For small radii, it is convenient to state results for $r =r(\theta)$ in terms of the angle parameter $\theta$.  
For sufficiently small angles, the superlevel sets $\con(N; \th)$  will have the same homotopy type as the full configuration space $\con(N)$.  
In terms of the radius function, the conclusion of this result applies for $0 \le r < r_1(N)$, where $r_1(N) = \sin \left(\frac{\pi }{N}\right)/\left(1-\sin \left(\frac{\pi }{N}\right)\right)$ is the smallest critical value for $\con(N)[r]$.

\begin{theorem} \label{thm:43}
Suppose $N \ge 3$. 
The smallest critical value for maximizing $\ir$ on $\bcon(N)$ is $\frac{\pi}{N}$, achieved uniquely by the $N$-Ring configuration of equally spaced points along a great circle. 
Moreover, for angular diameter $0 \le\th<\frac{2\pi}{N}$ the following hold.

$(1)$ The space $\con(N; \th)$ is a strong deformation retract of the full configuration space $\con(N)=\con(N; 0)$.

$(2)$ The reduced space $\bcon(N;\th)=\con(N; \th)/SO(3)$ is a strong deformation retract of the full reduced configuration space $\bcon(N)$.

 Consequently each has, respectively, the same homotopy type and cohomology groups as the corresponding full configuration space.
 \end{theorem}

\begin{proof}
This result corresponds to ~\cite[Theorem 5.1]{BBK:2014}. 
First note that by using equal weights on each of its edges, the $N$-Ring is balanced and hence a critical configuration by Theorem ~\ref{thm:converse}. 
The balanced contact graph on $\SS^2$ of a $\frac\th2$-critical $N$-configuration has geodesic edges with angular length $\th$. 
In order to balance, its total angular length must be at least $2\pi$, the length of a complete great circle. 
Thus if $\th<\frac{2\pi}{N}$, then the total length $N\th<2\pi$ and there is no balanced $N$-configuration in $\con(N; \th)$ and $\frac{\theta}{2}$ is not a critical value for $\ir$.  
In this case, a weighted $\ir$-subgradient-flow provides the strong deformation retraction of $ \con(N)$ to $\con(N; \th)$. 
\end{proof}

\begin{corollary}\label{cor:48}
For $\th < \frac{2\pi}{N}$ and $N \ge 4$ the configuration spaces $\con(N; \th)$ and $\bcon(N;\th)$ are path-connected, but not simply-connected.
\end{corollary}

\begin{proof}
These spaces have the same homotopy type as $\con(N)$ (resp. $\bcon(N)$),  which is  connected since $H^0(\con(N), \QQ)=\QQ$ (resp. $H^0(\bcon(N), \QQ)=\QQ$).
They each are  closures of open manifolds and are connected, so  are path-connected. 
We have $H^1(\bcon(n),  \QQ) = \QQ^k$ for some $k =k(N) \ge 2$, using the  formula  \eqref{eq4102} applied for $N \ge 4$, so $\bcon(N)$ is not simply-connected. 
Finally, $\con(N)$ is not simply connected via the product decomposition in Theorem ~\ref{thm:411}. 
\end{proof}

\subsection{Large Radius Case}\label{sec:43a}

We consider reduced configuration spaces $\bcon(N)[r]$ having radius parameter $r$ sufficiently close to  $r_{max}(N)$, depending on $N$. 
Using the finiteness of the set of critical values, there is an $\epsilon(N) >0$ such that the upward ``gradient flow'' of the injectivity radius function $\ir$ (or of the corresponding touching-sphere radius function~$r$) defines a deformation retraction from $\bcon(N)[r]$ to $\bcon(N)[r_{max}(N)]$ for the range $r_{\max}(N) - \epsilon(N) < r < r_{max}(N)$.  

The simplest topology that may occur at $r_{max}(N)$ is  where $\bcon(N)[r_{\max}(N)]$ has all its connected components contractible;
the property holds for most small $N$ -- in fact, for all $N \le 14$ except $N=5$. 
When it holds, the cohomology groups for $\bcon(N)[r]$ in this range of $r$ will have the following very simple form: 
\medskip

{\bf  Purity Property.}
{\em There is some $\epsilon = \epsilon(N) >0$ such that for
\[
r_{\max}(N) - \epsilon(N) < r < r_{max}(N)
\]
there is an integer $s=s(N) \ge 1$ such that  the cohomology groups of the reduced configuration space are
\[
H^k( \bcon(N)[r], \QQ)  = \left\{
\begin{array}{cl}
\ZZ^{s(N)}& \mbox{if} \quad k=0 \\
~~~ \\
0 & \mbox{if}  \quad k \ge 1.
\end{array}
\right.
\]
}\medskip

For $N=5$ the cohomology does {\em not} have the Purity Property.
The reduced configuration space $\bcon(5)[r]$ is  $7$-dimensional for $r < r_{max}(5)$ but becomes $2$-dimensional at $r=r_{max}(5)$. 
Some optimal maximum radius configurations at $r_{max}(5)= 1+ \sqrt{2}$ have room for an extra sphere (giving $N=6$): the sphere centers form five vertices of an octahedron, and either vertex in an antipodal pair of vertices can freely and independently move towards the unoccupied sixth vertex of the octahedron. 
The resulting reduced configuration space $\bcon(5)[1+\sqrt{2}]= \bcon(5; \frac{\pi}{2})$ is a simplicial $2$-complex which is not contractible; 
it is pictured schematically in Figure ~\ref{fig:example5}.  It has a single connected component having  Euler characteristic $\chi(\bcon(5; \frac{\pi}{2}))= -10$.
 For further discussion of this space as a critical stratified set, see Section ~\ref{sec:46}. 

\begin{figure}[htbp] 
   \centering
   \includegraphics[width=2in]{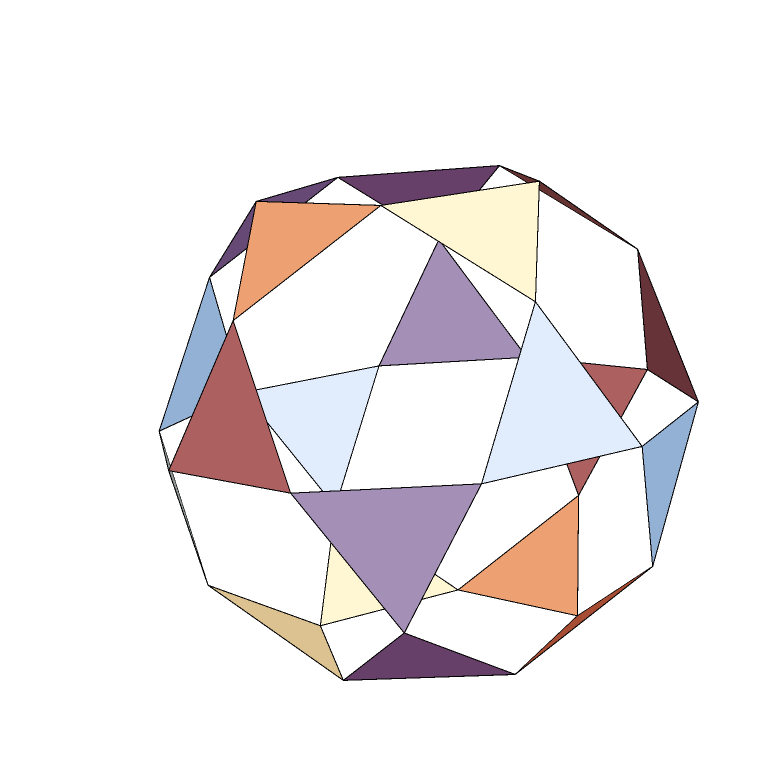} 
   \caption{The $r$-maximal stratified set for $N=5$} 
   \label{fig:example5}
\end{figure}

Does the purity property hold for  all or most  large $N$? 
We do not know. 
One might expect that  extremal configurations for high values of $N$ at  $r= r_{max}(N)$ will have most spheres are held in a rigid structure, and for $r$ near it all individual spheres will only be able to move in a tiny area around them, each contributing a connected component to the reduced configuration space.
Against this expectation, computer experiments packing $N$ equal-radius two-dimensional disks confined to a unit disk suggest the possibility for some $N$ that extremal configurations could have \emph{rattlers}, which are loose disks that have motion permitted even at $r=r_{max}$ (Lubachevsky and Graham ~\cite{LubG:1995}). 
However, even with rattlers one could still have contractibility of individual connected components.
The hypothesis of  extremal configurations being rigid (and unique) is known to hold for $6 \le N \le 12$.

When the purity property holds one can (in principle) determine the number of connected components for the set of near-maximal configurations; call it $s = s(N)$.  This value depends on the symmetries of each maximal configuration under the $SO(3)$ action.  
Denoting  the  isomorphism types of the connected components of maximal rigid (labeled) configurations of $N$ points at $r= r_{max}(N)$ by $C_{i,N}$ for $1 \le i \le e(N)$, one would have 
\[
s(N) = \sum_{i} \frac{N!}{|Aut(C_{i,N})|}.
\]
For $3 \le N \le 12$, excluding $N \ne 5$, the extremal configurations for the Tammes problem are known to be unique up to isometry; 
call them $C_{1, N}$. 
The analysis of Danzer given in Theorem ~\ref{th73} covers the cases  $7 \le N \le 10$.  
For the case $N=12$, the unique extremal configuration $\DOD$ of vertices of an icosahedron has  $Aut(C_{1, 12}) = A_5$, the alternating group, of order $60$, whence $s(12) = \frac{12!}{60} = 7983360.$ 

\subsection{Structure of Critical Strata}\label{sec:46}
 Connected components of critical strata necessarily have dimension at least three from the $SO(3)$-action.  
 In what follows we consider reduced critical strata that  quotient out by this action.  
 At a critical value $\rho$ there can be several disconnected reduced critical strata, and such strata  can have positive dimension.  
 We give examples of each.

For $N=5$ a  positive dimensional reduced critical stratified set occurs at the maximal radius value $r_{max}(5) = 1 + \sqrt{2}$.  
The set of (reduced) critical configurations forms a family, which is two-dimensional, containing multiple strata. 
A generic contact graph at the maximal injectivity radius $\ir=\frac{\pi}{4}$ is a $\Theta$-graph having  $2$ polar vertices and $3$ equatorial vertices.  
This contact graph, depicted in Figure ~\ref{fig:contact}, has $3$ faces and $6$ edges and is optimal.  
The three angles between equatorial vertices can range between $\frac\pi2$ and $\pi$, with the condition that their sum is $2\pi$, defining a $2$-simplex.  
As long as none of the equatorial angles is $\pi$, criticality is achieved using weights that are non-zero on all the edges.  
When an equatorial angle is $\pi$, corresponding to a corner of the $2$-simplex, these equatorial vertices may be regarded as a new pair of polar vertices.  
In this configuration, as the angles between equatorial angles go to $\pi$, some weights of the stress graph can go to $0$ and the support of the weights degenerates to a $4$-Ring.  
The limit contact graph consists of the edges of a square pyramid whose base is that $4$-Ring.  
This gives a non-optimal contact graph with $5$ faces and $8$ edges.

For $N=12$ there are several distinct reduced critical strata at the critical value $\rho=1$, two of which correspond to the $\FCC$-configuration and $\HCP$-configurations, singled out in Frank's discussion in Section ~\ref{sec:27}.  
These configurations are  defined in Section ~\ref{sec:52} below, and their  criticality is shown in Theorem ~\ref{th913bb}.

\subsection{Topology Change for Variable Radius: $N=4$}\label{sec:47}
For very small $N$ it is possible to completely work out all the critical points and the changes in topology.
We illustrate such an analysis on the simplest nontrivial example $N=4$ (see Figure ~\ref{fig:morse4}, explained below).

\begin{figure}[htbp] 
   \centering
   \includegraphics[scale=1]{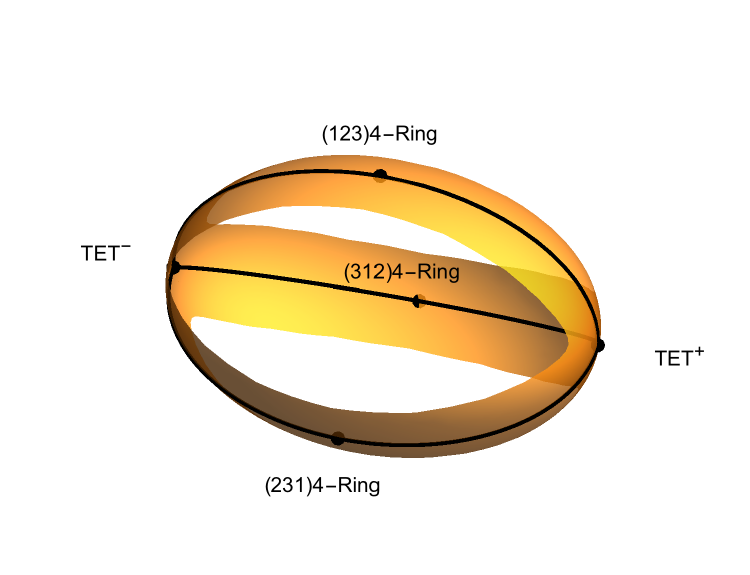}
   \caption{Part of the configuration space for $N=4$}  
   \label{fig:morse4}
\end{figure}

We consider the reduced superlevel sets $\bcon(4)[r]$. Since $\con(4)$ is $8$-dimensional, away from the critical values these spaces are $5$-dimensional manifolds with boundary.

If we ignore the labelling of points and classify the contact graphs for four vertices, there are exactly two geometrically distinct $\ir$-critical $4$-configurations in $\bcon(4)$:
\begin{enumerate}

\item[(1)]
The  $4$-Ring of four equally spaced points around a great circle on $\SS^2$ with $\th=\frac\pi2$ which is a saddle configuration for $\ir$. 
There is a $1$-dimensional subspace of the tangent space to $\bcon(4)$ at the $4$-Ring along which $\ir$ increases to second order, i.e. the co-index is $k=1$.
The critical value for $r$ for the $4$-Ring is $r_1= 1 + 2 \sqrt{2} \approx 3.8284$.

 \item[(2)]
The vertices of the regular tetrahedron \textrm{TET} with $\th=\cos^{-1}(-\frac{1}{3})$, which is the maximizing configuration for $\ir$ on $\bcon(4)$, i.e the co-index $k=0$.
The critical value for $r$ for TET is $r_2= r_{max} (4) = 2+ \sqrt{6} \approx 4.4495.$

 \end{enumerate}
 There are two intervals $(0,r_1)$ and $(r_1, r_2]$ on which the topology of $\bcon(4) [r]$ remains constant. 
 From Theorem ~\ref{thm:411}, it can be seen that on the interval $(0, r_1)$, $\bcon(4)[r]$ is homeomorphic to $\bcon(4)$. 
 This has the homotopy type of $\RR^2$ punctured at two points, hence
\[
 H^{0}(\bcon(4)[r], \Z) = \Z, \quad H^{1}(\bcon(4)[r], \Z) = \Z^2,\quad  H^{k}(\bcon(4)[r], \Z) = 0, \, \mbox{for}\,\, k \ge 2.
\]

On the open interval $(r_1, r_{2})$, the manifold  $\bcon(4)[r]$ has two connected components, each diffeomorphic to a $5$-ball, which can be seen from the strong deformation retraction to $\bcon(4)[r_2]$ consisting of the two points associated to the orientated labelings of $\textrm{TET}$ configurations, $\textrm{TET}^+$ and $\textrm{TET}^-$. 
Hence
\[
 H^{0}(\bcon(4)[r], \Z) = \Z^2, \quad  H^{k}(\bcon(4)[r], \Z) = 0, \, \mbox{for}\,\, k \ge 1.
\]

Figure ~\ref{fig:morse4} above is only a schematic picture, since we cannot draw a $5$-dimensional manifold. It compresses four of the
dimensions.  
The visible points take $r$-values with $r_1 \le r \le r_2$. 
The value $r=r_1$ is a circular vertical ring in the middle, and the values of $r$  increase as one moves to the left or right, reaching a maximum at $\textrm{TET}^+$ and at $\textrm{TET}^-$.

From Table ~\ref{table4}, we can easily compute the Euler characteristic $\chi(\bcon(4))=-1$. 
The indexed sum of critical points of the function $\rho: \bcon(4) \rightarrow \RR$ gives an alternative computation of the Euler characteristic as
\[
\chi(\bcon(4)) = \sum _k (-1)^k  \# (\,\textrm{critical points of co-index = } k \, ).
\]

We count the {\em labeled} configurations in $\bcon(4)$: 
since the $4$-Ring has symmetry group ~$D_4$ of order $8$ in $SO(3)$, there are $3=|\Sigma_4/D_4|$ critical points of this type with co-index $1$; 
and since TET has symmetry group $A_4$ of order $12$ in $SO(3)$, there are really $2=|\Sigma_4/A_4|$ critical points of this type with co-index $0$; 
and so we obtain
 \[
 \chi(\bcon(4))=2-3=-1,
 \]
 as predicted. 
 In fact, the Morse complex for $\ir$ captures the fact that $\bcon(4)$ itself has the homotopy type of the
  $\Theta$-graph:  
  there are $2$ vertices ($0$-cells) in the complex corresponding to the maxima (co-index $0$) $\textrm{TET}^+$ and $\textrm{TET}^{-}$ configurations; 
  there are $3$ edges ($1$-cells) corresponding to the $3$ saddle (co-index $1$) $4$-Ring configurations.

\subsection{Topology Change for Variable Radius: $N \ge 5$}\label{sec:48}

The complexity of the changes in topology of the configuration space grows rapidly with $N$.
For larger  values of $N$ there are many $\ir$-critical configurations which are not maximal.

The value $N=12$ is large enough to be extremely challenging to obtain a complete analysis of the critical configurations of the configuration space, and to analyze the variation of the topology as a function of the radius $r$. 
The Betti numbers for $N=12$ for radius $r=0$ given in Table ~\ref{table4} differ greatly from those at $r=r_{max}(12)$ where the cohomology of $\bcon(12)[r_{max}(12)]$ is entirely in dimension $0$, according to the Purity Property, which holds for $N=12$ by results in Section ~\ref{sec:3}. 
This topology change involves millions of (labeled) critical points. Its full investigation remains a task for the future.

\section{Unit Radius Configuration Space for $12$ Spheres}\label{sec:5}
In this section, we discuss $\con(12)[1]$ and $\bcon(12)[1]$, the configuration spaces of $12$ unit spheres touching a central unit sphere $\SS^2$.
These configuration spaces are remarkable and have some special properties.
The value $r = 1$ is a critical value and 
that $\bcon(12)[1]$
has (at least) two geometrically distinct critical points, the $\FCC$ and $\HCP$ configurations. 
We believe that $r=1$ is the maximal radius  where the spheres $\bcon(12)[r]$ are arbitrarily permutable with motions remaining on $\SS^2$ (see Section ~\ref{sec:65}).


The  case where all spheres are unit spheres has been extensively studied in connection with sphere packing. 
The value $r=1$ is a critical value of the radius function ~$r$, and we will see that the associated configuration spaces $\con(12)[1]$ and  $\bcon(12) [1]$ are not manifolds.
To better understand their topology, it is useful to consider the 
spaces $\con(12)[r]$ and $\bcon(12)[r]$ for $r$ in a neighborhood of $1$.
These are stratified spaces naturally embedded in $(\SS^{2})^{12}$ and filtered by $r$.
For noncritical values of $r$, the spaces $\con(12)[r]$ and $\bcon(12)[r]$ are submanifolds with boundary.
For all $r< r_{max}(12)$, the space $\con(12)[r]$ has top dimension $24$. After factoring out the ambient $SO(3)$-action, the space $\bcon(12)[r]$ has top dimension $21$.

\subsection{Three Remarkable Configurations of Unit Spheres: DOD, FCP, HCC}\label{sec:51}
We now consider the three configurations of $12$ touching spheres singled out by Frank (1952) ~\cite{Frank:1952}. 
In Figure ~\ref{fig:configs}, the three polyhedra have vertices located at the $12$ touching sphere centers of these configurations and centroids located at the center of the central sphere. 
The edges of these polyhedra specify the contact graphs of these configurations, also pictured schematically in Figure ~\ref{fig:configs}.
The $\DOD$ configuration realizes the optimal contact graph for $N=12$  given in Figure ~\ref{fig:contact} in Section ~\ref{sec:33}.  

\begin{itemize}[leftmargin=1.8em]
\item The $\DOD$ configuration is obtained by placing $12$ spheres touching a central $13$-th sphere at the vertices of an inscribed icosahedron;
such touching points are also the centers of the faces of a circumscribed dodecahedron. 
It has oriented symmetry group $A_5$, the icosahedral group, of order $60$ and in $\bcon(12)[1]$ there are $\frac{|\Sigma_{12}|}{|A_5|}=\frac{12!}{60}=7983360$ of these.  
 
\item The $\FCC$ configuration is obtained by stacking three layers of the hexagonal lattice, with the third layer not lying over the first layer.
The inscribed polyhedron formed by the convex hull of the $12$ points of the $\FCC$ configuration where the spheres touch the central sphere is a {cuboctahedron.} 
The circumscribed dual polyhedron which has the $12$ points as the center of its faces is a {rhombic dodecahedron}. 
It has oriented symmetry group $\Sigma_4$, the octahedral group, of order $24$ and in $\bcon(12)[1]$ there are $\frac{|\Sigma_{12}|}{|\Sigma_4|}=\frac{12!}{24}=19958400$ of these. 

\item The $\HCP$ configuration is obtained by stacking three layers of the hexagonal lattice, with the third layer lying directly over the first layer.
The inscribed polyhedron  formed by the convex hull of the $12$ points of $\HCP$ where the spheres touch the central sphere is a { triangular orthobicupola.}
This polyhedron is the Johnson solid $J_{27}$. The circumscribed dual polyhedron which has the $12$ points as the center of its faces is a {trapezoidal rhombic dodecahedron}.
It has oriented symmetry group $D_3$, the dihedral group of order $6$, and in $\bcon(12)[1]$ there are $\frac{|\Sigma_{12}|}{|D_3|}=\frac{12!}{6}=79833600$ of these. 
\end{itemize}
The $\DOD$ configuration is an interior point of $\bcon(12)[1]$, while
the $\FCC$ and $\HCP$ configurations are elements of $\partial \bcon(12)[1]$.

\begin{figure}[htbp] 
   \centering
		\includegraphics[scale=.39]{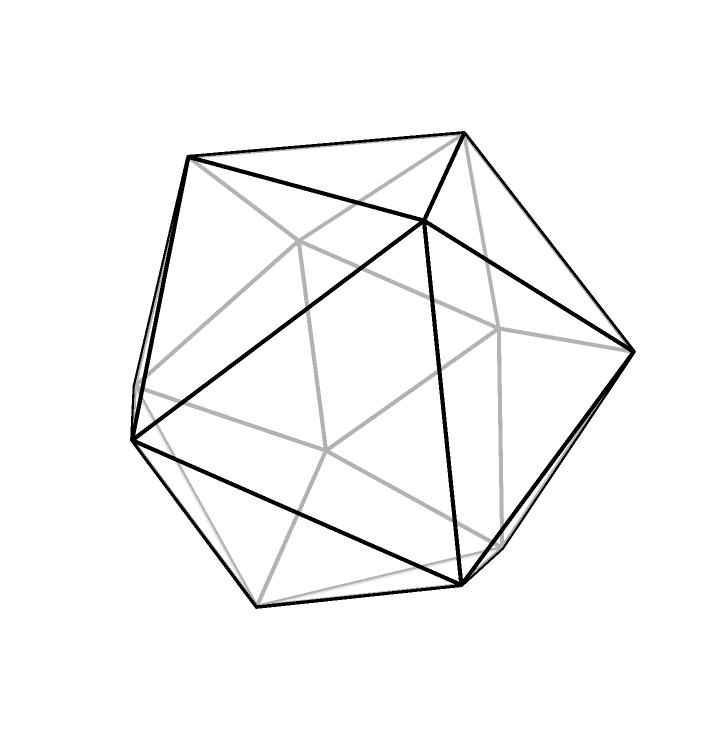} \quad\quad\quad
		\includegraphics[scale=.42]{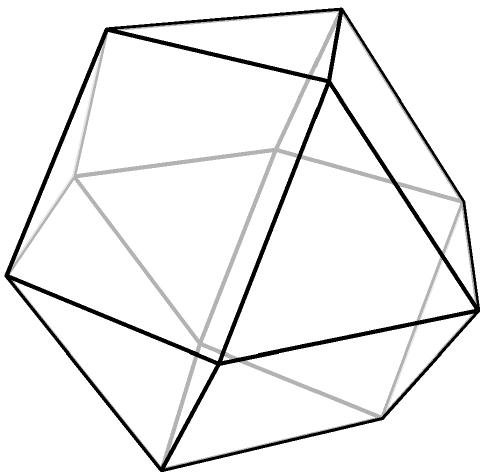}\quad\quad\quad
		\vspace{.2in}\includegraphics[scale=.38]{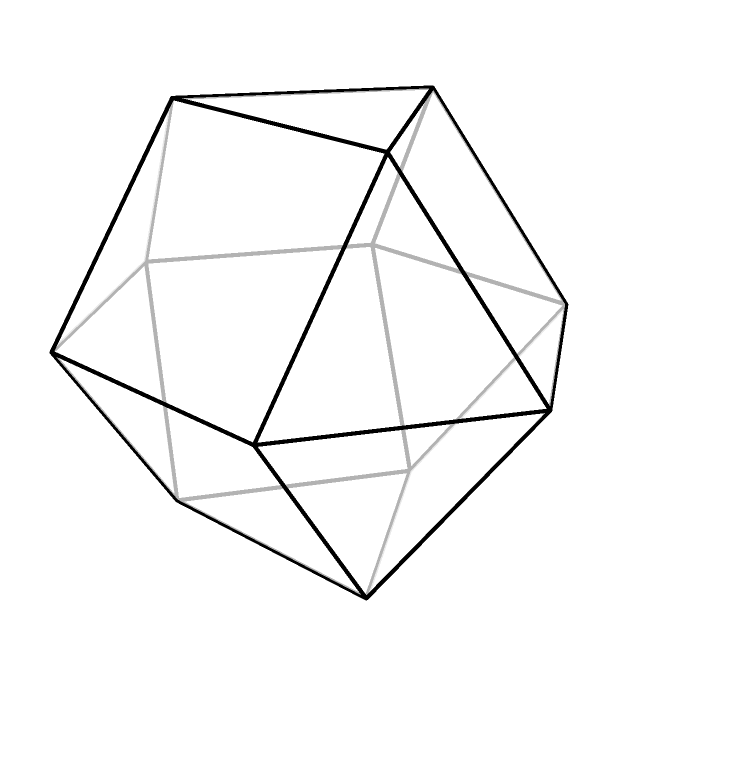}
		
		\includegraphics[scale=.38]{figures/n12.pdf} \quad
		\includegraphics[scale=.37]{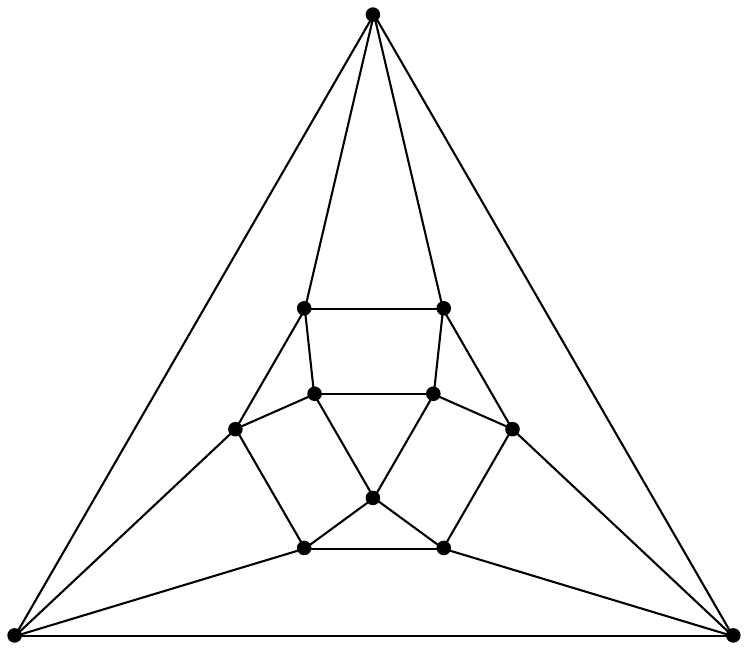}\quad
		\includegraphics[scale=.37]{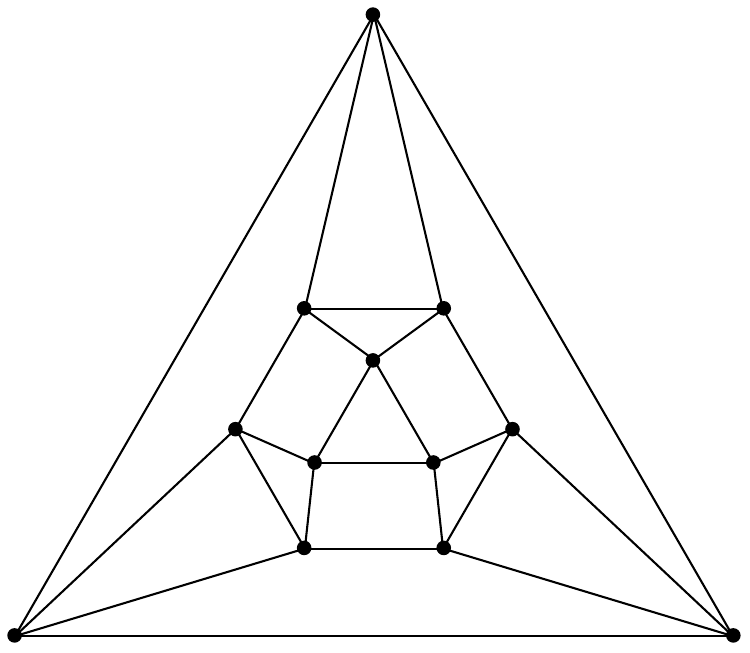}
		\captionsetup{justification=centering}

   \caption{The DOD, FCC, and HCP configurations with their corresponding contact graphs }
   \label{fig:configs}
\end{figure}

\subsection{Three Remarkable Configurations of Unit Spheres: Rigidity Properties}\label{sec:52}
We first consider rigidity properties of these packings.  
In the following definition we identify a sphere tangent to the $2$-sphere with the circular disk (i.e. spherical cap) on $\SS^2$ that it produces by radial projection. 

\begin{definition} (cf. Connelly ~\cite[p. 1863]{Connelly:2008})
A packing of disks on $\SS^2$ is {\em locally jammed} if each disk is held fixed by its neighbors.  
That is, no disk in the packing can be moved if all the other disks are held fixed.
We say a configuration of disks is {\em jammed} if it can only be moved by rigid motions.
We call it {\em completely unjammed} if each disk can be moved slightly while holding all the other disks fixed. 
\end{definition}

\begin{theorem}\label{th913}
The $\DOD$ configuration in $\con(12)[1]$ is completely unjammed. Its space of (infinitesimal) deformations has dimension $24$.
The deformation space is $21$ dimensional if viewed in $\bcon(12)[1]$.
\end{theorem}

\begin{proof}
The maximal radius for $12$ spheres is $r_{\max}(12) > 1$ and is achieved in the $\DOD$ configuration.
Therefore the deformation space at $\DOD$ is full dimensional.
\end{proof}

In contrast, both the $\FCC$ and $\HCP$ configurations are \emph{locally jammed}, i.e. they are rigid against motion of any one disk while holding all the other disks fixed;  
each of their infinitesimal deformation spaces has codimension at least $2$. 

In Section ~\ref{sec:53}, we will describe a deformation of the $\DOD$ packing to the $\FCC$ packing.
This deformation, properly adjusted, has $6$ moving balls during its final phase arriving at $\FCC$.
(The $6$ fixed balls form an antipodal pair of ``triangles''.)
We believe this value $6$ to be the smallest number of moving balls needed to unjam the $\FCC$ configuration.
For a manual on how to unlock $\FCC$, see Section ~\ref{manual}.

A deformation of the $\DOD$ configuration to the $\HCP$ configuration, also described in Section ~\ref{sec:53}, requires $9$ moving balls at the instant of arrival at $\HCP$. 
We believe this value ~$9$ to be the smallest number of moving balls needed to unjam the $\HCP$ configuration.
For a manual on how to unlock $\HCP$, see Section ~\ref{manual}.
A possible reason for the larger number of moving balls needed to unjam the $\HCP$ configuration compared with that of the $\FCC$ configuration is that the $\HCP$ configuration has fewer local symmetries.

\subsection{Three Remarkable Configurations of Unit Spheres: Criticality Properties}\label{sec:52a}
The $\FCC$ and $\HCP$ configurations are critical configurations for $r=1$. 
According to Theorem ~\ref{thm:converse} it suffices to show that these configurations carry balanced contact graph structures.

\begin{theorem}\label{th913bb}
The $\FCC$ configuration and $\HCP$ configuration for $r=1$ carry balanced contact graph structures.
Consequently, $r=1$ is a  critical value of the radius function on $\bcon(12)$. 
\end{theorem}

\begin{proof}
By Theorem ~\ref{thm:converse}, a sufficient condition for the criticality of a configuration for maximizing injectivity radius is that its contact graph can be balanced. 
That is, a set of positive weights may be assigned to the edges of the contact graph so that at each vertex the weighted vector sum, defined by the outward tangent vectors to the incident edges, vanishes. 

\begin{figure}[htbp] 
  \begin{centering}
          \quad\quad\quad\quad\quad\includegraphics[width=2.4in]{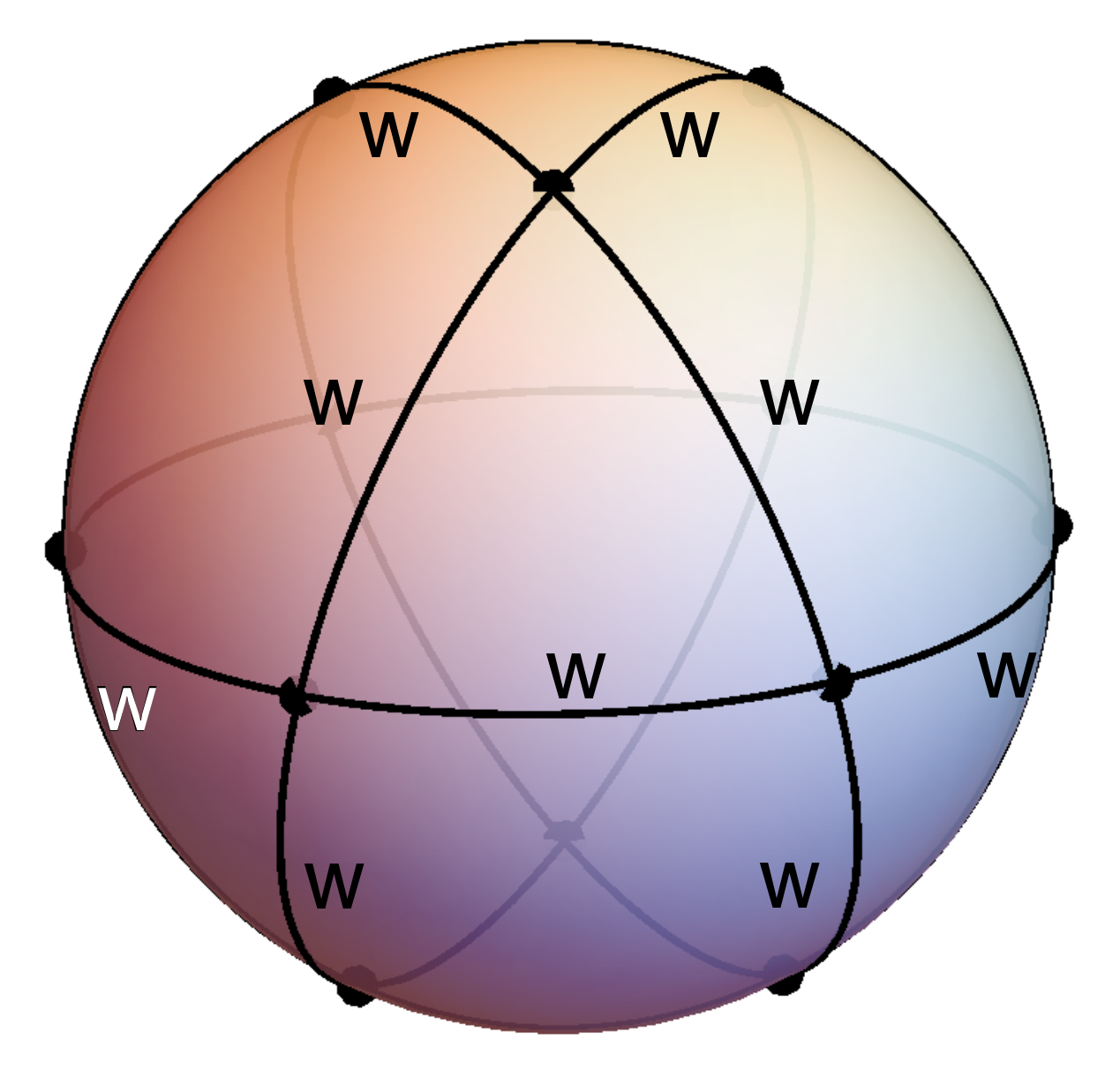} \newline
    (a) FCC Configuration\\
   \quad\quad\quad\quad \quad\includegraphics[width=2.7in]{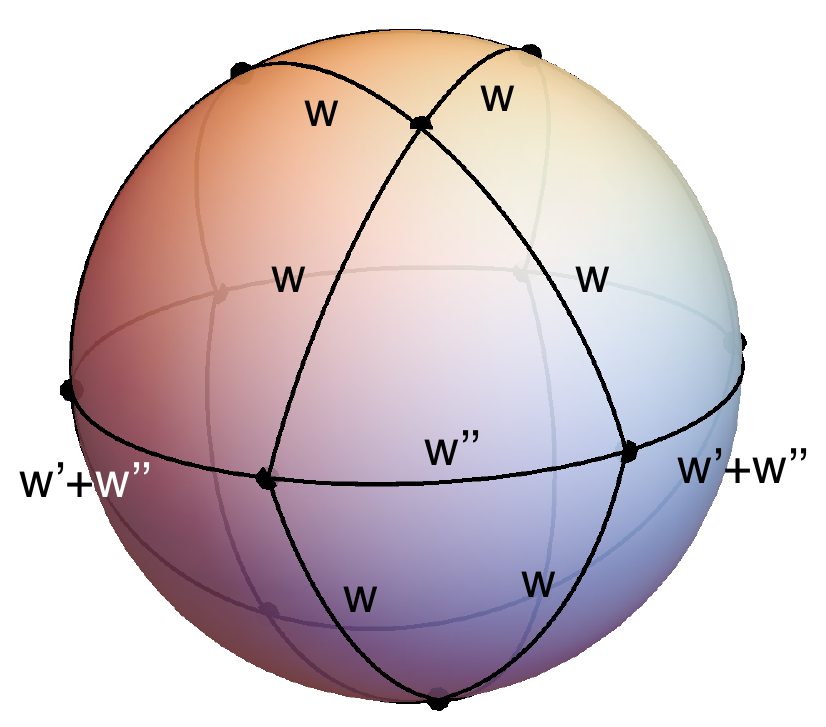} \newline
     (b) HCP Configuration\\
       \end{centering}
   \caption{Stress graphs for $\FCC$  and $\HCP$ configurations} \label{fig:stress-graph}
\end{figure}

We now indicate weight values for the $\FCC$ and $\HCP$ configurations (see Figure ~\ref{fig:stress-graph}).

(1) At radius $r=1$ for the $\FCC$ configuration, the stress graph is balanced when all the weights are equal.  
This can be seen from the cubic- or $\Sigma_4$-symmetry of the contact graph. 

(2) At radius $r=1$ for the $\HCP$ configuration, consider a weight $w_1$ on edges between triangular faces and square faces, a weight $w_2$ on edges between pairs of square faces, and a weight $w_3$ on edges between pairs of triangular faces.
From the structure of the contact graph, it is possible to choose a constant $w_1 > 0$ and find a weight $w_2 = w' >0$ which balances the associated stress graph.  
This suffices to balance this configuration with some zero weights. 
However, it is also possible to add a uniform constant weight $w'' = w_3 >0$ to the equatorial great circle, giving a balanced stress graph with positive weights on all edges of $w_1, w_2 := w'+w'', w_3 := w''$.
\end{proof}

The $\DOD$ configuration is not a critical configuration for $r=1$; 
instead it is a critical configuration at the maximal radius $r= r_{max}(12)\approx 1.10851$.
As noted in Section ~\ref{sec:2}, Fejes T\'{o}th ~\cite{FT:1943b}  conjectured that this configuration does have a certain extremality property for local packing by equal spheres, that it gives a minimizer for a single Voronoi cell of a unit sphere packing.
This statement, the Dodecahedral Conjecture, was proved in 2010  by Hales and McLaughlin ~\cite{HalesM:2010}.

\begin{theorem}\label{th514a} {\rm (Hales and McLaughlin (2010)) }
A $\DOD$ configuration of unit spheres minimizes the volume of a Voronoi cell of a unit sphere  with center at the origin of $\RR^3$ over all sphere packing configurations of unit spheres containing that sphere.
\end{theorem}

The volume of this Voronoi cell gives a local sphere packing density of approximately $0.7546$, which exceeds the sphere packing density $\frac{\pi}{\sqrt{18}} \approx 0.74048$ in $3$-dimensional space.  

\subsection{Three Remarkable Configurations of Unit Spheres: Deformation Properties}\label{sec:53}
We now show that the $\DOD$ configuration can be continuously deformed inside $\bcon(12)[1]$ to the $\FCC$ configuration  and to the $\HCP$ configuration.


\begin{theorem}\label{th55}\hspace{0em}\par~$(1)$ On the space $\con(12)[1]$ there is a continuous deformation of the $\DOD$ configuration to the $\FCC$ configuration that remains in the interior $\con^{+}(12)[1]$ of $\con(12)[1]$ till the final instant.

$(2)$ There is also a continuous deformation of the $\DOD$ configuration to the $\HCP$ configuration that remains in the interior $\con^{+}(12)[1]$ of $\con(12)[1]$ till the final instant. 
\end{theorem}

The motions of these two deformations, measured from the touching points of the $12$ spheres to the central sphere,  can be given by piecewise analytic functions on the $2$-sphere. 
The proof of Theorem ~\ref{th55} is given in Sections ~\ref{sec:531} - ~\ref{sec:546}.  
An unlocking manual for doing it is given in the Appendix (Section ~\ref{manual}). 

\subsubsection{Coordinates for the icosahedral configuration $\DOD$}\label{sec:531}

To describe the deformations we will need coordinates.
In Sections ~\ref{sec:531} - ~\ref{sec:545}  we suppose  a  ball with radius $1$ centered at $0$ touches all the $12$  balls of the same given radius $r$.
We initially allow all values of $r \in (0,r_{max}(12)]$, but in the move $M_6$ described in later subsections we will necessarily restrict to $r \le 1$. 
We take $\DOD$ to consist of $12$ equal balls of radius $r$, touching a central unit sphere at $12$ vertices of an inscribed icosahedron $I$. 
We view this icosahedron $I$ as embedded in $\mathbb{R}^{3}$ with Cartesian coordinates so that we have: 

\begin{itemize}

\item Its centroid is at the center $(0,0,0)$ of the unit sphere. 

\item It has two opposite faces parallel to the $xy$-plane. 
In other words, for some $z=h$, the intersections $I\cap\left\{z=\pm h \right\}$ are triangular faces of $I.$ 
Here $h= \frac{\phi^2}{\sqrt{3 \phi^2 +3}} \approx 0.79659,$ where $\phi= \frac{1+ \sqrt{5}}{2}$ is the golden ratio. 

\end{itemize}

\begin{figure}
\centering
\begin{minipage}{.5\textwidth}
  \centering
   \includegraphics[scale=.3]{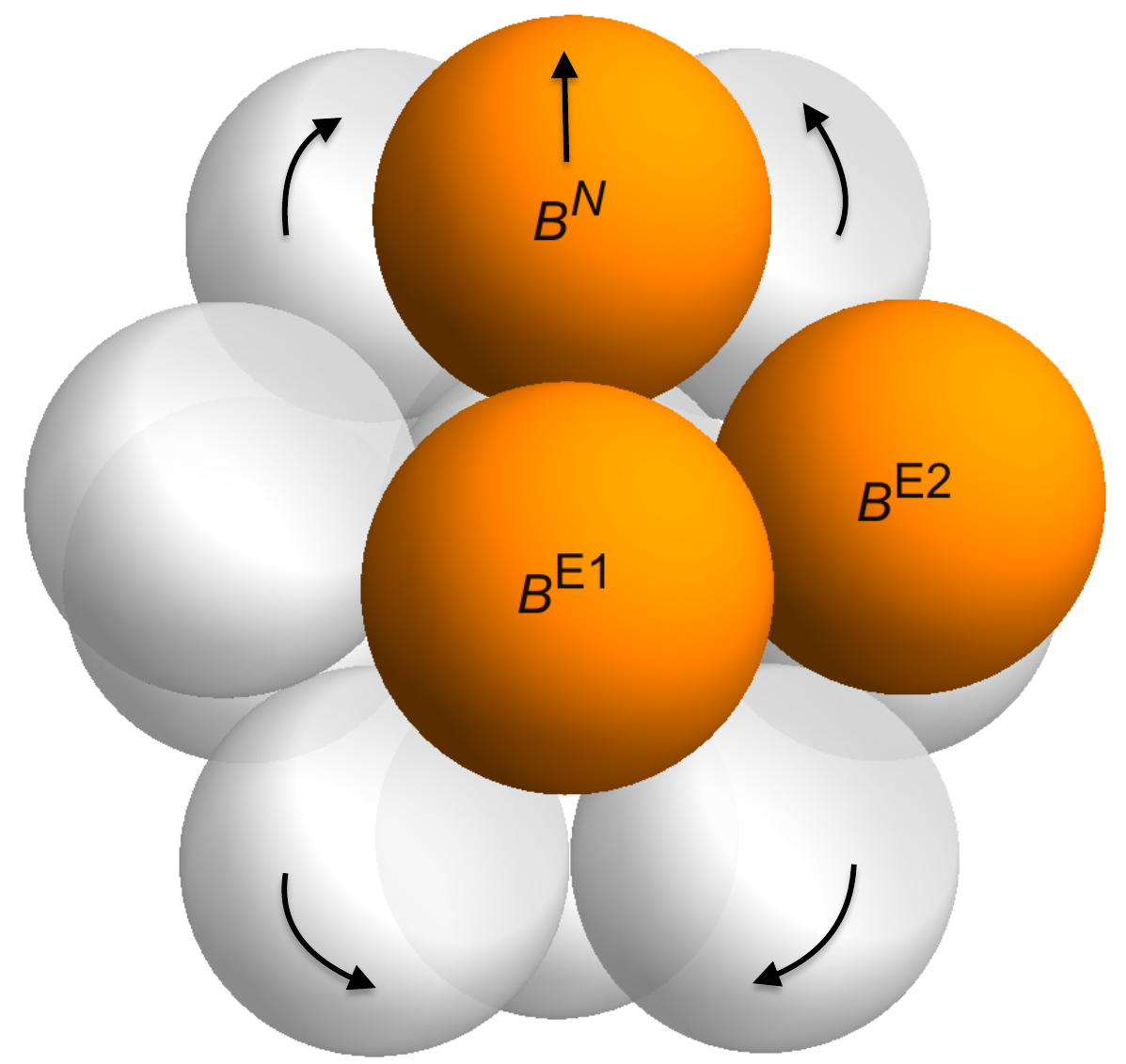} \\
 (a) Phase 1
\end{minipage}%
\begin{minipage}{.5\textwidth}
  \centering
  \includegraphics[scale=.5]{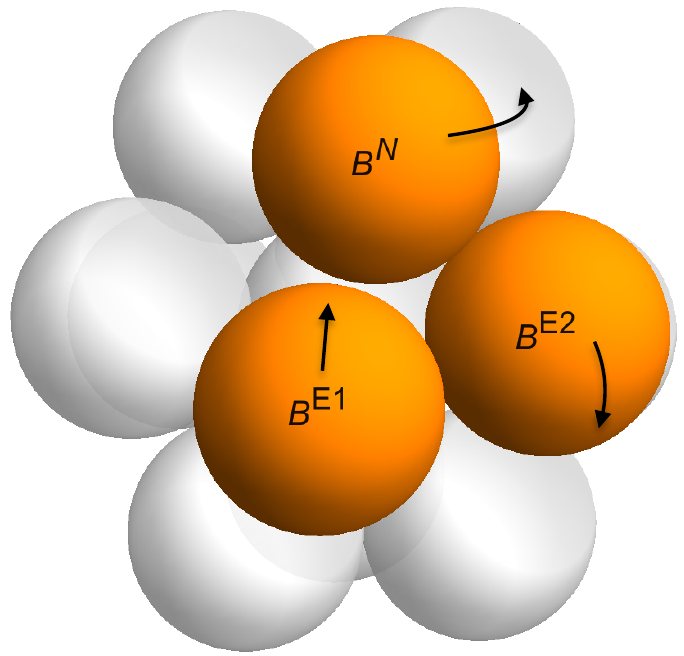} \\
(b) Phase 2
\end{minipage}
   \caption{The $6$-move $M_6$}
   \label{fig:6-move}
\end{figure}

The \textit{north balls} are those which have their centers in the plane $\left\{  z= h \right\},$  
while the \textit{south balls} have their centers in the plane $\left\{  z= -h \right\}.$ 
The three north balls form a \textit{north triangle} which is centrally symmetric to the \textit{south triangle} formed by the three south balls, as in $\FCC$.  
Together with north and south triangles, we have $6$ remaining balls, which will be called \textit{equatorial}, even though in the initial configuration they do not have their centers on the equator. 
The equator lies in the plane $z=0$. 

To fix their positions, let the {\em Greenwich meridian} be defined as $\SS^2\cap\left\{x\geq0,\ y=0\right\},$ and the longitude $\varphi\in\lbrack0,2\pi)$ be measured from it in the counterclockwise direction viewed from the north pole $(0,0,1)$. 
We require: 

\begin{itemize}

\item The center of one of the north balls is in the half-plane of the Greenwich meridian, i.e. this ball touches the Greenwich meridian. 
Let us call this ball $\ball{B}^{N}.$ 

\item The center of one of the equatorial balls is in the half-plane of the Greenwich meridian. 
Let us call it $\ball{B}^{E1}.$ It will necessarily be in the southern hemisphere.  

\end{itemize}

This fixes the location of all  $12$ balls. 
With this orientation of the icosahedron, the meridians of the north triangle are spaced by $\frac{2\pi}{3}.$ 
Furthermore the meridians of the other three balls in the northern hemisphere are also spaced by $\frac{2\pi}{3}$ and the meridians combined are spaced by ~$\frac{\pi}{3}$ as in $\FCC$. 
The same holds for the six balls in the southern hemisphere.

\subsubsection{The  $M_{6}$-move: two variants}\label{sec:532a}
We now define the ``$6$-move'' deformation $M_6$, which has two variants, one leading from $\DOD$ to the $\FCC$ configuration, and the other leading to the $\HCP$ configuration. 
This move proceeds in two phases.

The first phase is the same for both variants.
It moves the $6$ balls that are not equatorial at constant speed along meridians towards the poles, until they form north and south triangles of three mutually touching balls. 
The $6$ equatorial balls do not move. 

In the second phase, all $12$ balls are moving. 
In both variants, the $6$ equatorial balls, initially not on the equator, move towards the equator along their meridians at constant speed, to arrive on the equator at the end of the move, forming a ring of six balls on the equator.  
This ring is an allowed configuration only if $r \le 1.$  
They do not touch during this move, until the last moment, and then all touch if $r=1$.
At the same time, the north and south triangles will rotate about the polar axis at a variable speed, the same for all six, in such a way as to avoid the equatorial balls.  
They will rotate by $\frac{\pi}{3}$ to their final position. 
For the $\FCC$ move, the north triangle and south triangle rotate in the same direction, while for the $\HCP$ move they rotate in opposite directions.  
A key issue is to suitably specify the variable speed of rotation. 

\subsubsection{Frst phase of the $M_{6}$-move: two triangles move away from the equator}\label{sec:532} 
Denote by $P$ the parallel $z=z_1$ where the three centers of the north triangle stop. 
Let $-P$ be the parallel $z=-z_1$ where the south triangle stops.

\subsubsection{Second phase of the $M_{6}$-move: rotating the two triangles}\label{sec:533} 
Each of the six centers of the equatorial balls, initially not on the equator, will move at constant speed along their respective meridians towards their final positions on the equator. 
Parametrize this motion so that at $t=0$, the $6$ balls are at their initial positions, while at $t=1$, the $6$ balls are at their final positions on the equator.

The centers of the north triangle are on $P$ at $t=0$ and will remain on $P$ throughout the move.  
Similarly, the centers of the south triangle are on $-P$ at $t=0$ and will remain on $-P$ throughout the move. 
The triangles simply rotate.

It now suffices to specify functions $\varphi^{N}\left(  t\right)$, which describe the\textit{ increment} of the longitude of the north triangle during the time $\left[  0,t\right]$ and $\varphi^{S}\left(t\right)$, the\textit{ increment} of the longitude of the south triangle during the time $\left[  0,t\right]$. 
We will take $\varphi^{N}\left(  t\right)$ to be a continuous, non-decreasing function, with $\varphi^{N}\left(  0\right) =0,$ $\varphi^{N}\left(  1\right)=\frac{\pi}{6}.$

We get two different moves to $\FCC$ and $\HCP$ depending on which direction the south triangle rotates. 
The motion $\varphi^{S}=\varphi^{N}$ will take us to $\FCC$, and choosing the opposite rotation $\varphi^{S}=-\varphi^{N}$ will take us to $\HCP$.

\subsubsection{The choice of $\varphi^{N}$}\label{sec:545}
The function $\varphi^{N},$ defined so that no ball from the two triangles hits any equatorial one, is certainly not unique.

Here is a minimal definition of $\varphi^{N},$ beginning at the second phase.
Recall that our balls are open, and that: 
\begin{itemize}

\item the center of one of the three north balls, $\ball{B}^{N}$, is on the half-plane of the Greenwich meridian.  

\item the center of one of the equatorial balls, $\ball{B}^{E1}$, is also on the half-plane of the Greenwich meridian. 

\item there is an equatorial ball with the longitude $\frac{\pi}{3};$ call it $\ball{B}^{E2}.$  

\end{itemize}

\noindent
Note that the center of $\ball{B}^{E1}$ is south of the plane $z=0,$ while that of $\ball{B}^{E2}$ is north of the plane $z=0$. 

Throughout the second phase of $M_6$ the ball $\ball{B}^{E1}$ will move north while $\ball{B}^{E2}$ moves south.  
Let $\ball{B}^{E1}\left( t \right) ,$ $\ball{B}^{E2}\left( t \right) $ denote their positions, $t \in \left[  0,1 \right].$ 
Then for  every $\varphi \in \left[ 0,2\pi \right)$ define $\ball{B}^{N}\left(  \varphi \right)  $ to be the ball with the center at $P$ and with the longitude $\varphi.$ For example, $\ball{B}^{N}\left(  0 \right)  = \ball{B}^{N}.$ 

Now let us define the function $\varphi^{N}$ as follows:
\begin{equation}\label{eq721}
\varphi^{N}\left(  t\right)  = \inf \left\{ \varphi \ge 0 : \ball{B}^{N} \left( \varphi \right) \cap \ball{B}^{E1} \left(  t \right)  = \varnothing \right\}. 
\end{equation} 
Clearly, $\varphi^{N}\left(  t\right) = 0$ for all $t$ small enough. The only thing one needs to check is that
\begin{equation}\label{eq722}
\ball{B}^{N} \left( \varphi^{N} \left( t\right) \right) \cap \ball{B}^{E2} \left( t\right )  = \varnothing 
\end{equation}
holds for all $t \in \left[ 0,1 \right] $.\\ 

\begin{lemma}\label{lem:56}
An  increment function $\varphi^N(\cdot)$ exists: $\varphi^N(t)$ as defined by \eqref{eq721} satisfies \eqref{eq722}.
\end{lemma}
\begin{proof}
We use Euler coordinates on the sphere.  
The latitude $\vartheta$ of the parallel $P$ is $\left( \pi-\theta \right),$ where $\theta$ satisfies 
\[
\sin\left(\theta\right)=\frac{1}{\sqrt{3}},\ \cos\left(\theta\right)=\sqrt{\frac{2}{3}},\ \tan\left(\theta\right)
=\frac{1}{\sqrt{2}}. 
\]
\newline  

\noindent By symmetry, it is enough to consider the movement of three balls: 

\begin{itemize}
\medskip
\item $\ball{B}^{N}$ on the parallel $P$. 
Its initial angle $\varphi\left(  t=0\right)  =0.$ 
The latitude $\vartheta$ of $\ball{B}^N$ is constant. 
\medskip
\item $\ball{B}^{E1}$ on the Greenwich meridian $\varphi=0.$ 
\medskip
Its initial latitude is $\vartheta\left(  t=0\right) = \left( \frac{2\pi}{3}- \theta \right) <0$ and final latitude is $\vartheta\left(  t=1\right)  = 0.$   
On the interval, its latitude is given by $\vartheta \left(  t\right)  =\left(  \frac{2\pi}{3}-\theta\right) t.$
\medskip
\item $\ball{B}^{E2}$ on the meridian $\varphi=\frac{\pi}{3}.$ 
Its initial latitude is $\vartheta\left(  t=0\right)  =-\left(  \frac{2\pi}{3}-\theta\right)  >0$ and final latitude is $\vartheta\left(  t=1\right) = 0.$ 
On the interval, its latitude is given by $\vartheta\left(  t\right)  =\left(  \theta-\frac{2\pi}{3}\right)  t.$
\end{itemize}

\noindent The function $\varphi^{N}\left(  t\right)  $ is uniquely defined by: 

\begin{itemize}
\medskip
\item $\varphi^{N}\left(  t\right)  \geq0.$ 
\medskip
\item At every time $t<1$ and after initial contact, the ball $\ball{B}^{N}\left(  t\right)$ touches the ball $\ball{B}^{E1}\left(  t\right)$.    

\end{itemize}

\noindent To complete the proof of Lemma ~\ref{lem:56} it remains to check that the balls $\ball{B}^{N}\left(  t\right)$ and $\ball{B}^{E2}\left(t\right)$ of radius $r \le 1$ are disjoint for $0< t<1$. 
This fact will follow from the next lemma.  \end{proof}

\begin{lemma}\label{lem:57}
Let $0<t<1.$ Consider an {isosceles spherical} triangle $ABO,$ where $A$ has $\varphi=0,$ $\vartheta\left(  t\right)  =\left(  \frac{2\pi}{3}-\theta\right)  t,$ $B$ has $\varphi=\frac{\pi}{3},$ $\vartheta\left(t\right)  =-\left(  \frac{2\pi}{3}-\theta\right)  t,$ and $O=O\left(t\right)$ is defined by $\vartheta\left(  O\right)  =\pi-\theta$ and the touching condition. 
Then $\left\vert AO\right\vert =\left\vert BO\right\vert >\frac{\pi}{3}.$
\end{lemma}

\begin{proof}
Let $D$ be the middle point of the arc $AB.$ It does not depend on $t$ and is given by $\vartheta\left(  D\right)  =0,$ $\varphi\left(  D\right)  =\frac{\pi}{6}.$ 
Let $\varkappa\left(  t\right)$ be the arc perpendicular to the arc $AB$ at $D.$ Then $O\left(  t\right) $ is simply the intersection of $\varkappa\left(  t\right) $ and the parallel $P.$

The triangle $O\left(  t\right)  DA\left(  t\right) $ is a right triangle. Evidently, the legs $O\left(  t\right) D$ and $A\left(  t\right)  D$ become shorter as $t$ increases. 
Hence the hypotenuse $O\left(  t\right) A\left(  t\right)$ becomes shorter as well. 
Since $O\left(  1\right) A\left(  1\right)  =\frac{\pi}{3},$ the proof follows.
\end{proof}

\subsubsection{Completion of proof of   Theorem ~\ref{th55}}\label{sec:546}

\begin{proof}[Proof of Theorem ~\ref{th55}]  
Lemmas ~\ref{lem:56} and ~\ref{lem:57} complete a proof that there exists a deformation path from the $\DOD$ configuration to a $\FCC$ configuration and to a $\HCP$ configuration, respectively. 
However, the deformation path obtained  does not satisfy one required condition of the theorem: 
remaining in the interior of the configuration space. 
It exits from the interior of $\con(12)[1]$ at the end of the first phase and remains on the boundary during the second phase: 
the three north balls are touching and the three south balls are touching. 

We can modify the construction above so that no balls touch throughout the deformation until the final instant. 
To do this we halt the first phase just short of the three balls touching, at $z= 1 -\epsilon_1$. 
Then in the second phase, we allow $z$ to increase monotonically in the north triangle at some variable speed $\psi(t)$ 
 as the rotation proceeds, in such a way as to avoid contact between the three north balls and the equatorial balls.  
 The south triangle $z$ variable is to decrease monotonically in the reflected motion of $-z$ at the same time. 
Lemma ~\ref{lem:57} implies that if $z$ approaches $1$ rapidly enough in the motion that we can again avoid contact;
this is an open condition at each point $t$, so by compactness of the motion interval we have a finite subcover to attain it.
\end{proof}

\begin{remark}\ 

(1) This motion process can be continued by concatenation with an inverse $M_6$ using $-\varphi(1-t)$, in such a way as to arrive back at a  $\DOD$ configuration, differently labeled.  
This is possible because there are two exit directions (tangent vectors) from the $\FCC$ configuration and two exit directions from the $\HCP$ configuration in $\bcon(12)[1]$.  
Section ~\ref{sec:551} studies the group of permutations of the $12$ labels obtainable by such deformations. 

(2) Starting from the $\FCC$ or $\HCP$ configuration, there is a reference frame in which the north triangle remains fixed. 
The inverse of the second phase of $M_6$ describes a move which unlocks the $\FCC$ configuration with $6$ moving balls and $6$ fixed balls, and which unlocks the $\HCP$ configuration with $9$ moving balls and $3$ fixed balls (see Section ~\ref{manual}). 
\end{remark}

\subsection{Buckminister Fuller's ``Jitterbug'' }\label{sec:56} 
According to his recollection, on 25 April 1948 Buckminster Fuller found  a ``jitterbug'' construction given by  a jointed framework motion that, among other things,  permits an $\FCC$ configuration, given as the vertices of a cuboctahedron, to be continuously deformed into a $\DOD$ configuration, given as the vertices of an icosahedron (see ~\cite[p. 273]{Schw10}). 

In Buckminster Fuller's construction, the joint distances remain constant during the motion, so that they can be rigid bars, while the radii of the associated touching spheres continuously contract during the deformation. 
At each instant during the motion the central sphere and the $12$ touching spheres can all have equal radii without overlapping, and this radius varies monotonically in time. 

In retrospect one may see that it is possible to rescale space during the motion via homotheties varying in time such that all spheres retain the fixed radius $1$ throughout the deformation.
In this case the joint lengths will change continuously in the motion.
The rescaled motion no longer corresponds to a physical object with rigid bars, but it does give a continuous motion in the configuration space of $12$ equal spheres touching a 13-th central sphere that continuously deforms the $\FCC$ configuration to the $\DOD$ configuration. 

The work of Buckminster Fuller on the ``jitterbug'' movable jointed framework is described in Schwabe ~\cite{Schw10}.
Fuller described it in his book  Synergetics ~\cite[Sec. 460.00-463.00]{Fuller:1976}.
The construction  is also described in Edmondson ~\cite[Chap. 11]{Edmondson:1987}, with a detailed analysis in Verheyen ~\cite{Vehr89}.

\begin{remark}
{ The ``jitterbug'' motion immediately enters  the interior $\bcon^{+}(12)[1]$ after the initial instant, in contrast to the ``unlockings'' described in Appendix 8, which adhere to its boundary.}   
\end{remark}

\section{Permutability of the $\DOD$ Configuration:
Connectedness Conjectures}\label{sec:6}
The number of connected components  of the configuration space $\con(12)[r]$ is related to the ability to permute labeled spheres by deformations within $\con(12)[r]$. 
The possible permutability of the (labeled) spheres in the $\DOD$ configuration in $\con(12)[r]$ depends on the radius $r$ of the touching spheres.

\subsection{Permutations of $\DOD$ Configurations for Radius $r=1$.}\label{sec:551}
Conway and Sloane ~\cite[Chap. 1, Appendix, pp.  29--30]{Splag3} give a terse proof that for radius $1$ the labels on labeled spheres in $\DOD$ configurations can be arbitrarily permuted using continuous deformations inside  the space $\con(12)[1]$.

\begin{theorem} \label{thm:58}{\rm (Permutability at radius $r=1$)}
For the radius parameter $r=1$, each labeled $\DOD$ configuration can be continuously deformed in the configuration space $\bcon(12)[1]$ to a $\DOD$ configuration at the same $12$ touching points with any permutation of the labeling.
\end{theorem}

\begin{figure}[htbp]
   \centering
   \includegraphics[scale=1.1]{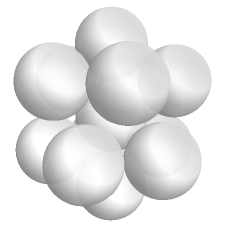}
   \includegraphics[scale=1.5]{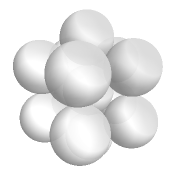}
   \caption{The $5$-move $M_5$}
   \label{fig:5-move}
\end{figure}

We follow the outline in Conway and Sloane ~\cite[Chapter 1, Appendix, pp. 29--30]{Splag3}. 
A main ingredient is an additional set of permutation moves which we call $M_5$-moves, detailed next.

\subsection{The  $M_{5}$-move}\label{sec:62}
Beginning from the $\DOD$ configuration centered at the origin, we rotate it so that two opposite balls have their centers on the $z$ axis. 
Call these balls $\ball{N}$ and $\ball{S}$. Note that the centers of $5$ of the $10$ remaining balls are in the northern half-space, while the remaining $5$ centers are in the southern half-space.  
Call these balls $\ball{U}_1 \dots \ball{U}_5$ and $\ball{V}_1 \dots \ball{V}_5$, respectively. 

\textbf{First phase.} Move the $5$ northern $\ball{U}_j$ balls towards $\ball{N}$, in such a way that their centers remain on their corresponding meridians, 
until each of them touches $\ball{N}$.   
Note that these ~$5$ ~balls do not touch each other, only $\ball{N}$. 
Indeed, because their centers are located at the latitude $\vartheta=\frac{\pi}{6},$ when viewed from the $z$-axis each of the $5$ balls subtends the dihedral angle $\arccos \left(\frac{1}{3}\right)$ of a regular tetrahedron; but the longitude difference between the neighboring ball centers is $\frac{2\pi}{5}$, so there remains a tiny longitude gap 
\[
\zeta=\frac{2\pi}{5}-\arccos\left(\frac{1}{3}\right)\approx0.025
>0.
\]
\noindent 
The $5$ southern $\ball{V}_j$ balls may be moved into the southern hemisphere in the same manner.  

\textbf{Second phase.} Note that for $r=1$, the $6$ northern balls fit into the northern half-space, while the $6$ southern balls fit into the southern half-space. 
The union of all the $6$ balls in the northern hemisphere may be rotated by $\frac{2\pi}{5}$ as a rigid body, keeping the remaining balls fixed.  

\textbf{Third phase.} Reverse the first phase.

The net result of an $M_5$-move is a cyclic permutation $\sigma_{2}$ of $\DOD$ of length $5$, which is an even permutation. 

\begin{remark}
The halfway point of a $M_5$ as illustrated in the right side of Figure ~\ref{fig:5-move} is itself a ``remarkable'' critical configuration; it is found at the center of a 4-simplex of critical configurations for $r=1$. The family of such configurations is in many ways similar to the maximal stratified set that occurs for $N=5$. 
\end{remark}
\subsection{Proof of Permutability Theorem at Radius $r=1$}\label{sec:63}
\begin{proof}[Proof of Theorem ~\ref{thm:58}]
The first step is to show that there is a continuous deformation of $\DOD$ to itself, which permutes the labels by an odd permutation.
To exhibit it, we use the move $M_{6},$ defined before, and deform $\DOD$ into an $\FCC$ configuration (note that we can do it for
all $r \leq1,$ but not for $r>1$). 
Note that the $\FCC$ configuration has three axes of 4-fold symmetry passing through the opposite squares of four balls. 
By rotating $\frac{\pi}{2}$ around any such an axis and then deforming our configuration back to $\DOD$ via $M_{6}^{-1},$ 
we induce a  a permutation $\sigma_{1}$ of 12 balls, which is a product of three (disjoint) cyclic permutations, each of length 4. 
Every such cycle is an odd permutation, hence their product $\sigma_{1}$ is also odd. 

The second step uses  $M_{5}$-moves. Each such move gives a cyclic permutation of order ~$5$.
Since there are $12$ options for choosing $\ball{N}$, we get $12$ such $5$-cycles $\sigma_{2}^{\left( i \right)}$.  
It is shown in Conway and Sloane (using an elegant argument about the Mathieu group $M_{12}$, see ~\cite[pp. 328-330]{Splag3}) 
that all such $\sigma_{2}^{\left(  i\right)}$ generate the alternating group $A_{12}$, the subgroup of even permutations of ~$\Sigma_{12}.$  

Combined with any odd permutation $\sigma_{1}$, the full permutation group $\Sigma_{12}$ is generated.  
\end{proof}

\subsection{Persistence of  $M_5$-moves to some $r>1$}\label{sec:64}
The move $M_{5}$ can be modified in such a way that it continues to work for all values $r\leq r_{1}$, for some $r_{1}$ slightly bigger than $1$.
We first explain the modification and then propose the value of $r_{1}.$ 
The modification deals only with the second phase of $M_5$. 
In order to explain it, it is enough to follow the $10$ longitude values of the touching points of our balls, which may be considered as points on the equator. 
For $r=1$, the northern $5$ balls $\ball{U}_{j}$ correspond to the longitude values $u_{j},$ for $j=1,...,5,$  
and we can suppose that at the initial moment these values are $u_{j}\left(  t=0\right)  =\left(j-1\right)  \frac{2\pi}{5}.$ 
The longitude values $v_{j}$ are defined similarly, corresponding to the southern balls $\ball{V}_j$, and $v_{j}\left(  t=0\right) =\left(  j-1\right)  \frac{2\pi}{5}+\frac{\pi}{5}.$ 
Our initial move looks now as follows:
\[
u_{j}\left(  t\right)  =\left(  j-1+t\right)  \frac{2\pi}{5},\ v_{j}\left(
t\right)  =v_{j}\left(  0\right)  .
\]
Of course there is no need for all the $u_{j}$ to move with the same speed; 
the only constraint is that the difference between consecutive $u_{j}$ should equal or exceed $\arccos \left(\frac{1}{3}\right)$ at all times. 
In particular, we can modify the speeds in such a way that at any time $t$, we have $u_{j}\left(t\right)  =v_{j}\left(  t\right)  $ for at most one value of $j.$ 

Now let the radius $r$ be slightly bigger than $1$. 
Then, at the moment $t$ when $u_{j}\left(  t\right)  =v_{j}\left(  t\right),$ the corresponding balls $\ball{U}_j,$ $\ball{V}_j$ will overlap.  

This, however, can be remedied by making the following small deformation of our $12$-configuration: 

\begin{itemize}

\item the ball $\ball{U}_j$ moves up along its meridian, by the distance $\left(r-1\right).$ 

\item the ball N moves along the same meridian in the same direction by the distance $2\left(  r-1\right).$ 

\item the ball $\ball{V}_j$ moves down along its meridian, by the distance $\left(r-1\right).$ 

\item the ball S moves along the same meridian in the same direction by the distance $2\left(  r-1\right).$ 

\item other balls may be rearranged in such a way that they do not intersect.

\end{itemize}

The non-overlap condition can be satisfied when $\left(r-1\right)>0$ is small enough, since there were no other collisions.  

Below we will show there will be  $5$  {\em bottleneck configurations} that  one encounters on the way to perform the modified $M_{5}$ move.
Each one defines a value $r_{1}^{\left(j\right)}>1$, for $1 \le j \le 5$ which is the maximal radius for which this configuration is allowed. 
We set  $r_{1} :=\min_{j=1,...,5}r_{1}^{\left(  j\right)  }>1.$  

\begin{theorem}\label{thm:62}
For every $r \leq r_{1}$ the move $M_{5}$ can be modified in such a way that one can reach from an initial labeled $\DOD$ configuration any labeled  $\DOD$ configuration whose labels are an even permutation of the initial labels.  
That is, the alternating  group  $ A_{12}$ is generated by the compositions of different  $M_{5}$ moves.
\end{theorem}

\begin{proof}
There will occur $5$ bottleneck $12$-configurations of the $r$-balls touching the unit central ball, 
described by certain touching patterns that correspond to the configurations appearing 
during the move $M_{5}$ at the moment when the ball $\ball{U}_j$ passes due north of the ball~$\ball{V}_j.$   

The $5$ bottleneck configurations have a common pattern: $4$ touching balls centered on the same meridian, 
two in the northern half-space, and the remaining two in the southern half-space.   
We denote them by $\ball{N}$, $\ball{U}_j,$ $\ball{V}_j,$ $\ball{S}$. This set of $4$ balls is symmetric with respect to the plane $z=0.$ 
Strictly speaking, as $r$ is slightly bigger than~$1,$ the balls $\ball{N}$ and $\ball{S}$ are now centered   
on the meridian {\em opposite} the one containing $\ball{U}_j$ and $\ball{V}_j$.   

The eight other balls are the remaining ones from $\ball{U}_{1},...\,,\ball{U}_{5},\ball{V}_{1},...\,,\ball{V}_{5}.$
Each pair $\{\ball{U}_i$, $\ball{V}_i\}$ touches, as do the pairs $\{\ball{U}_i$, $\ball{N}$\}, as well as the pairs \{$\ball{V}_i$, $\ball{S}\}$.     
The $5$ bottleneck configurations differ in how the additional pairs of balls touch.  
Each of the rows in the following list completes a different touching pattern that occurs as the move $M_5$ is performed:

\begin{equation}\tag{$j=1$}
\{\ball{U}_{1},\ball{U}_{2}\}, \{\ball{V}_{2},\ball{U}_{3}\}, \{\ball{V}_{3}, \ball{U}_{4}\}, \{\ball{V}_{4},\ball{U}_{5}\}, \{\ball{V}_{5},\ball{V}_{1}\},
\end{equation}
\begin{equation}\tag{$j=2$}
\{\ball{U}_{2},\ball{U}_{3}\}, \{\ball{V}_{3},\ball{U}_{4}\}, \{\ball{V}_{4}, \ball{U}_{5}\}, \{\ball{V}_{5},\ball{V}_{1}\}, \{\ball{U}_{1},\ball{U}_{2}\},  
\end{equation}
\begin{equation}\tag{$j=3$}
\{\ball{U}_{3},\ball{U}_{4}\}, \{\ball{V}_{4},\ball{U}_{5}\}, \{\ball{V}_{5}, \ball{V}_{1}\}, \{\ball{U}_{1},\ball{V}_{2}\}, \{\ball{U}_{2},\ball{U}_{3}\},  
\end{equation}
\begin{equation}\tag{$j=4$}
 \{\ball{U}_{4},\ball{U}_{5}\}, \{\ball{V}_{5},\ball{V}_{1}\}, \{\ball{U}_{1}, \ball{V}_{2}\}, \{\ball{U}_{2},\ball{V}_{3}\}, \{\ball{U}_{3},\ball{U}_{4}\},  
\end{equation}
\begin{equation}\tag{$j=5$}
\{\ball{V}_{5},\ball{V}_{1}\}, \{\ball{U}_{1},\ball{V}_{2}\}, \{\ball{U}_{2}, \ball{V}_{3}\}, \{\ball{U}_{3},\ball{V}_{4}\}, \{\ball{U}_{4},\ball{U}_{5}\}.  
\end{equation}\newline

Observe that for any $r>1$, and for any of the $5$ touching patterns,  
such a configuration is unique if it exists, and that it {\em does} exist for $\left(r-1\right)$ sufficiently small.  
We define $r_{1}^{\left(  j\right) }$ as the maximal values for which the above configurations exist. 
\end{proof}


\begin{figure}[htbp] 
   \centering
   \includegraphics[width=1.7in]{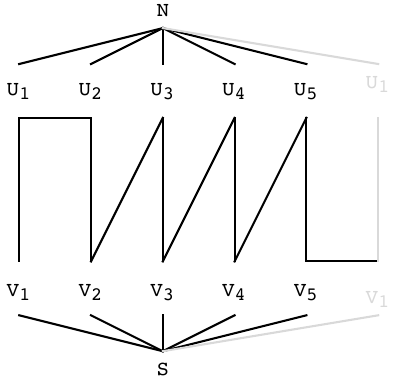} \quad\quad\quad
     \includegraphics[width=1.7in]{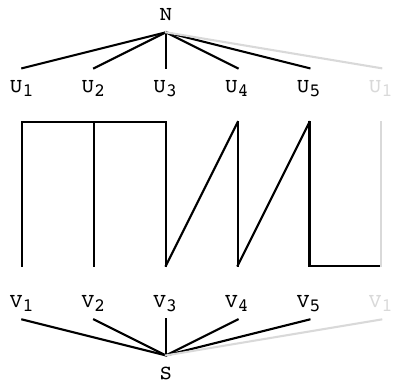} \quad \quad\quad
       \includegraphics[width=1.7in]{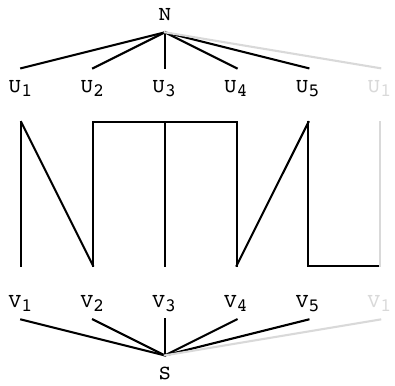} 
         \includegraphics[width=1.7in]{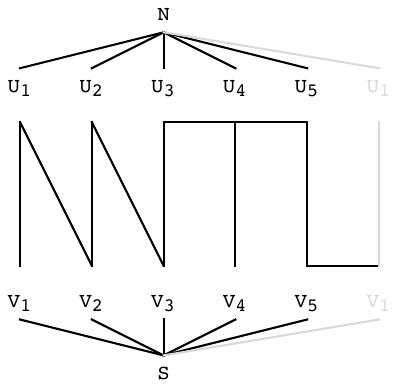} \quad \quad\quad
           \includegraphics[width=1.7in]{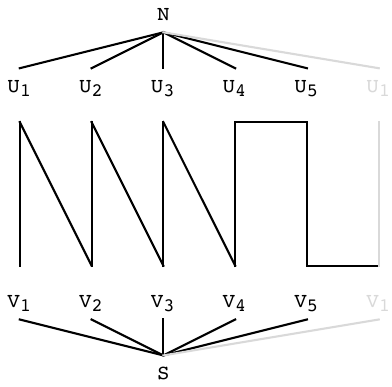} 
   \caption{Contact graphs for the bottleneck configurations $j=1$ to $j=5$}
   \label{fig:bottles}
\end{figure}

We are not asserting that the value $r_1>1$ defined above is the true critical value above which the (small perturbation of the) move $M_5$ cannot be performed. 
Indeed, we imposed some a priori constraints in making our construction of the modified $M_5$, and did not rule out the possibility of a more ``optimal'' modification of $M_5$.

\begin{definition}\label{defi:63}
{\em Let $R_1$ be the maximal value of the radius $r$ for which there exists some modified move $M_5$. We call it the {\it upper critical radius.}
}
\end{definition}

From the previous theorem we know that $R_1 \ge r_1 >1$. We expect $R_1$ to be a critical value for maximizing the radius function. 

\subsection{Connectedness Conjectures for $\con(12)[1]$}\label{sec:65} 
Based on Theorems~\ref{th55} and ~\ref{thm:58} about deformation and permutability of labeled $\DOD$ configurations, it is natural to propose the following statement.

\begin{conjecture} {\rm (Connectedness Conjecture)}\label{conj95}
The configuration space $\con(12)[1]$ is connected. 
That is, every set of $12$ distinct labeled points on the $2$-sphere pairwise separated by spherical angle at least $\frac{\pi}{3}$ can be deformed into $12$ other distinct labeled points, with all points maintaining a spherical angle at least $\frac{\pi}{3}$ apart during the deformation. \end{conjecture}

This problem appears to be approachable but difficult to prove, despite the supporting evidence of permutability in Theorem ~\ref{thm:58}.
One may approach it by cutting the space $\bcon(12)[r_0]$ into many small path-connected \textquotedblleft convex'' pieces and gluing them together in some fashion. 
The computational size of the problem, since the dimension of the space is ~$21$, and has a complicated boundary, is daunting.

We also propose a stronger statement.

\begin{conjecture} {\rm (Strong Connectedness Conjecture)}\label{conj95a}
The radius $r=1$ is the largest radius value at which configuration space $\con(12)[r]$ is connected.
\end{conjecture}

In support  of Conjecture ~\ref{conj95a}, the $M_{6}$-move appears to be possible only when $r \leq1$. 
At one time instant it  has $6$ spheres fitting in a ring around the equator, a condition which is allowed only for $r \le 1$.
We also  know that  the  $r=1$  satisfies the necessary condition of being  a \textquotedblleft critical value'' for the $r$-parameter.

We formulate one further conjecture concerning the connectivity structure of $\bcon(12)[1]$.
It is based on a further analysis not included here (see ~\cite{KKLS16+}), 
which indicates that  $\bcon(12)[1]$ has at each of the $\frac{12!}{24}$ $\FCC$-configurations, 
and at each of the $\frac{12!}{6}$ $\HCP$-configurations, a unique tangent line along which it can be approached from the interior $\bcon^{+}(12)[1]$. 
In addition, this analysis shows that each of these is a {\em local cut point},  
a point of a space which when removed, disconnects a small   
open neighborhood of the point.
That is, these  configurations are  points at which the space $\bcon(12)[r]$  locally disconnects as $r$ increases past $1$.
The following conjecture asserts  that these local cut points form unavoidable bottlenecks in $\bcon(12)[1]$ in making certain rearrangements of spheres in the configuration space. 

\begin{conjecture} {\rm ($\FCC$ and $\HCP$ Bottlenecks)}\label{conj95b}
Any piecewise smooth continuous curve in the reduced configuration space $\bcon(12)[1]$ which starts at a labeled $\DOD$ configuration and ends at another labeled $\DOD$ configuration  with the labels permuted by an odd permutation must necessarily  pass through either an $\FCC$ configuration or a $\HCP$ configuration. 
\end{conjecture} 

This conjecture asserts a specific way in which the $\FCC$ and $\HCP$ configurations may play a remarkable role in rearrangements of $12$-configurations, 
illuminating the assertion of Frank ~(1952) in Section ~\ref{sec:27}. 

\subsection{Disconnectedness Conjectures for $r>1$}\label{sec66}
For the region $1<r\leq R_{1}$ we propose the following conjecture.


\begin{conjecture}\label{conj:66} {\rm (Two Connected Components)}
Let $R_1$ be the upper  critical  radius defined after Theorem ~\ref{thm:62}.
Then the space $\con(12)[r]$  for $1 < r < R_1$ has exactly two connected components. 
Two labeled configurations, $\DOD$ and $\sigma(\DOD)$, where $\sigma \in \Sigma_{12}$ is a permutation of twelve labels, belong to different connected components of $\con(12)[r]$ if and only if the permutation $\sigma$ is odd. 
\end{conjecture}

In view of the existence of the $M_5$-move, the argument in Theorem ~\ref{thm:58} indicates that there can be at most $2$ connected components containing $\DOD$ configurations in  this region. 
Conjecture ~\ref{conj:66}  asserts there are exactly these two, and no other, connected components. 

We next note that  the  five ``bottlenecks'' in the $5$-move lead to the possibility  of connected components not containing any $\DOD$ configuration, for certain ranges of $r$. 
During the $M_{5}$ move joining two $\DOD$ configurations, there are $5$ bottlenecks all through which one can pass at least up to a radius $r_{1}>1$. 
There is however room for configurations of spheres of larger radius occurring between the bottlenecks.
If we increase the radius above the smallest two of the bottleneck radii, it may be possible for a sphere to get stuck in the middle of one of these regions, so it can neither go backwards nor forwards via the $M_5$-move to a $\DOD$ configuration.
Assuming that ``trapped'' configurations (from blocking of the $M_{5}$ move) exist containing no $\DOD$ configuration, eventually as we increase $r$ some ``trapped'' configuration must become a critical configuration. 
It would then be a local maximum in the configuration space, an isolated point in some $\bcon(12)[r]$. 
The  critical value at which this occurs would necessarily be strictly smaller than $r_{max}(12)$, using the result of Danzer (in Theorem ~\ref{th73})  that the extremal configuration for $N=12$ is unique. 

\begin{conjecture}\label{conj:65}
{\rm (Non-$\DOD$ Components)}
There is  a nonempty interval of values of $r>1$ such that the reduced configuration space $\bcon(12)[r]$ has connected components that do not contain any copy of a $\DOD$ configuration.
\end{conjecture}

A positive answer to this question  raises the possibility of a value of $r$ for which the number of connected components of $\bcon(12)[r]$  exceeds the number of labeled $\DOD$ configurations, which is $\frac{12!}{12\times5}=7983360$.
To obtain the latter number, let  us take the $\DOD$ configuration and label its $12$ balls. There are $12!$ such labelings.  
Two labeled configurations are equivalent (i.e.~the same in $\bcon(12)[r]$) iff one can be obtained from the other by the $SO(3)$ rotation action. 
Clearly, there are $12\times5$ labelings in every equivalence class. 

Finally, for the region of $r$-values very close to $r_{max}(12)$, we assert that each of the spaces $\con(12)[r]$ and $\bcon(12)[r]$  has exactly $ \frac{12!}{12\times5}=7983360$ connected components, with each component containing a   $\DOD$ configuration.
This fact follows assuming the finiteness of the set of critical radius values $r$ for $\bcon(12)$, since at the point $r_{max}(12)$ only the $\DOD$ configurations survive, according to the uniqueness result of Danzer (Theorem ~\ref{th73}), and the topology of $\bcon(12)[r]$ does not change above the next largest critical value of $r$ below~$r_{max}(12)$. 


\section{Concluding Remarks }\label{sec:7}
This paper treats configuration spaces of touching spheres for  very small values of $N$.
We have shown that the configuration space of $12$ equal spheres touching a central $13$-th sphere is already large enough to exhibit interesting behavior 
in its critical points. 
Concerning $12$-sphere configurations in the equal radius case $r=1$ we have made the following observations. 
\begin{itemize}
\item
We have clarified an assertion of Frank (1952) given in Section ~\ref{sec:27}, showing that in the space $\bcon(12)[1]$ there are deformations interconnecting all  $\FCC$, $\HCP$ and $\DOD$ configurations.
\item
We have given evidence suggesting that $\bcon(12)[1]$ is a connected space, and conjectured that $r=1$ is the largest parameter value where $\bcon(12)[r]$ is connected.
\item
We have shown that all elements of the finite set of $\FCC$ and $\HCP$ configurations lie on the boundary 
of the topological space $\bcon(12)[1]$ and are critical points for maximizing the radius parameter. 
\item 
We have conjectured that a continuous deformation of 12 spheres in a $\DOD$ configuration 
to a permutation of itself that is an 
odd permutation of its elements,  then the deformation must
pass through one of  the
 (finite set of) $\FCC$ and $\HCP$ configurations in $\bcon(12)[1]$; 
 they  are ``unavoidable'' points.  
\end{itemize}
Many challenging and computationally difficult problems remain to better understand the constrained configuration space $\bcon(12)[1]$. 

As mentioned in the introduction, configuration spaces are of interest in physics and materials science, particularly in connection with jamming in materials. 
Hard sphere models which view spheres packed inside a box have been extensively studied for jamming. 
Materials scientists have studied configuration spaces of small numbers of hard spheres by simulation in connection with nanomaterials. 
Recently, Holmes-Cerfon ~\cite{Holmes-C:2016} developed an algorithm that enumerates rigid sphere clusters and has determined those with up to $16$ spheres. 
The cases of small numbers (but larger than the $N$ treated here) of spheres were studied in Phillips ~et ~al. ~\cite{PJG12} and Glotzer et al. ~\cite{PJG14}, giving estimates for extremal configurations at values of $N$ larger than can be currently treated mathematically. 
We note that simulations of phase space can  sample only a small part of it.
In the simulation experiments reported in ~\cite{PJG12}  for $N=12$ equal spheres, the experimenters were unable to detect that the radii at which the $M_5$-move and the $M_6$-move permutation cease being feasible are in fact different (as discussed in Section ~\ref{sec:64}).  

Study of the jamming problem leads to the sub-problem concerning what is a good notion of rigidity for such configurations. 
There is a notion of ``locally jammed configuration'' in which no particle can move if its neighbors are fixed. 
The Tammes problem --- also called the (extremal) spherical codes problem --- of determining $r_{max}(N)$ is analogous to 
the problem of determining maximally dense jammed configurations of spheres in a box. 
Various notions of rigidity for spherical codes were formulated in Tarnai and G\'{a}sp\'{a}r ~\cite{TarnaiG:1983}.
More recently Cohn et al.~\cite{CohnT:2011} give a mathematical treatment of rigidity of extremal $n$-dimensional spherical codes. 

In configuration theory models like $\bcon(N)[r]$  of this paper, critical configurations at critical values  of the radius parameter might serve as a proxy  for locally jammed configurations. One can view  the  balancing condition in Theorem ~\ref{thm:converse} as a weak 
form  of the  locally jammed condition.  
However only a subclass of critical configurations will be locally jammed in the sense above. 


\section{Appendix: Unlocking Manual  for  the FCC and HCP Configurations \label{manual}}
\subsection{The $\FCC$ Configuration}
To unlock the $\FCC$ configuration, a good way is to do it with the help of a friend, hereafter called Charles\footnote{\ Apr\'{e}s Charles Radin}. 
Please follow these steps:

\begin{enumerate}[leftmargin=1.6em]
\item[(1)]  
Ask Charles to hold the $3$ north balls and the $3$ south balls firmly in their positions. 
These $6$ polar balls remain fixed during the whole process. 
As a result, the $13$-th central ball stays fixed as well. 

\item[(2)]  
Roll the remaining $6$ equatorial balls in a direction roughly parallel to the equator. 
If properly lubricated, this does not require a big effort. 

\item[(3)] 
The equatorial balls can all be pushed either to the east or to the west, in a coordinated way.

\item [(4)] 
At all times you must ensure the $6$ rolling balls touch the central ball. 
This requires some practice, but it is possible and not terribly hard. 

\item[(5)]  
Observe that the $6$ balls roll around the central ball along the equatorial ``valley'' between the polar balls kept fixed by Charles. 
These rolling balls cannot always move equatorially, but instead move north and south slightly, in an alternating manner, as you roll them.

\item[(6)] 
Because the $6$ rolling balls move north and south, some of them do not touch each other any more: 
free space may appear between them.
Also, some space can be created between them and the $6$ balls kept fixed by Charles. This is normal.

\item[(7)]  
As you proceed by $\frac{\pi}{3},$ the $6$ rolling balls realign in the equatorial plane, touching each other and the polar balls. 
Note that at this moment the configuration is locked back into $\FCC$. Each of the $6$ rolled balls is touching two of its equatorial neighbors, one ball to the north and one ball to the south.
\end{enumerate}

\subsection{The $\HCP$ Configuration}
Unlocking the $\HCP$ configuration is similar to the $\FCC$ configuration, except that Charles has somewhat more to do. 
Please follow these steps: 
\begin{enumerate}[leftmargin=1.6em]
\item[(1)]  
Ask Charles to hold firmly the three north balls and the three south balls. 
The three south balls will remain fixed during the whole process. 
But the north triangle has to be rotated as a whole in its plane, at some constant speed, which can be either eastward or westward (there are two choices). 
It will move through an angle~$\frac{2 \pi}{3}$. The $13$-th central ball stays fixed as before. 

\item[(2)] 
Roll all the remaining $6$ balls in the (roughly same) equatorial direction as the north triangle is rotating. 
This movement direction is forced on all six equatorial balls by the motion of the north triangle. 

\item[(3)]  
The rest of the process goes basically in the same way as for the $\FCC$ configuration.

\item[(4)]  
As Charles proceeds to rotate the north triangle by $\frac{2\pi}{3}$, you proceed by $\frac{\pi}{3},$ the six middle balls align back into the equatorial plane, touching each other and the six polar balls. Note that at this moment the configuration is locked back into the $\HCP$ configuration. 
Each of the $6$ rolled balls is touching two of its equatorial neighbors, one ball to the north and one ball to the south.  

\end{enumerate}
Note that for $\FCC$, the equatorial balls underwent cyclic permutation of length $6$. For $\HCP$, the equatorial balls underwent a cyclic permutation of length $6$ and the $3$ northern balls a cyclic permutation of length $3$. These give odd permutations of $\FCC$ and $\HCP$.


\section*{Acknowledgments}

The  authors were each  supported by ICERM in the Spring 2015 program on ``Phase Transitions and Emergent Properties.'' 
R.~K. was also supported by the University of Pennsylvania Mathematics Department sabbatical visitor fund and by MSRI via NSF grant DMS-1440140.  
W.~K. was also supported by Austrian Science Fund (FWF) Project 5503.  
J.~L. was supported by NSF grant DMS-1401224 and by a Clay Senior Fellowship at ICERM. 
Part of the work of S.~S. has been carried out in the framework of the Labex Archimede (ANR-11-LABX-0033) and of the A*MIDEX project (ANR-11-IDEX-0001-02), funded by the ``Investissements d'Avenir'' French Government programme managed by the French National Research Agency (ANR). 
Part of the work of S.~S. has been carried out at IITP RAS. 
The support of Russian Foundation for Sciences (Project No. 14-50-00150) is gratefully acknowledged. 
The authors thank Bob Connelly, Sharon Glotzer, Mark Goresky, Tom Hales and Oleg Musin for helpful comments. 
Parts of Section ~\ref{sec:41} are adapted from unpublished notes by R.~K. and John~Sullivan (MSRI, 1994) about critical configurations of ``electrons'' on the sphere. 


\end{document}